\documentclass[11pt,reqno]{amsart}

\usepackage{epsfig}
\usepackage{amsmath}
\usepackage{amssymb}
\usepackage{amscd}
\usepackage{graphicx}
\usepackage{color}
\usepackage{mathtools}
\usepackage{cases}

\usepackage{cite}
\usepackage{mathrsfs}
\usepackage{amssymb}
\usepackage{graphicx,latexsym,amsmath,amssymb,amsthm,enumerate}

\newtheorem{theorem}{Theorem}[section]
\newtheorem{lemma}{Lemma}[section]
\newtheorem{corollary}{Corollary}[section]
\newtheorem{proposition}{Proposition}[section]
\newtheorem{definition}{Definition}[section]
\newtheorem{remark}{Remark}[section]

\newcommand{\be}{\begin{equation}}
\newcommand{\ee}{\end{equation}}
\newcommand{\ba}{\begin{align}}
\newcommand{\ea}{\end{align}}

\numberwithin{equation}{section}

\begin{document}

\baselineskip =1.4\baselineskip
\thanks{The project is supported by the National Natural Science Foundation of China
(Grant no.12171343) and Sichuan Science and Technology Program(no.2022ZYD0009 and no.2022JDTD0019).}

\author{Juan Huang} \address[Juan Huang]
{School of Mathematical Science and V.C.\& V.R. Key Lab of Sichuan Province\\
 Sichuan Normal University   \\
 Chengdu, Sichuan 610066, P.R. China}
   \email[]{hjmath@163.com}

\author{Sheng Wang} \address[Sheng Wang]
{School of Mathematical Science and V.C.\& V.R. Key Lab of Sichuan Province\\
 Sichuan Normal University   \\
 Chengdu, Sichuan 610066, P.R. China}
   \email[]{wangsmath@163.com}

\title[Normalized ground states for the mass supercritical Schr\"{o}dinger-Bopp-Podolsky system: existence, uniqueness, limit behavior, strong instability] {Normalized ground states for the mass supercritical Schr\"{o}dinger-Bopp-Podolsky system: existence, uniqueness, limit behavior, strong instability}

\pagestyle{plain}

\begin{abstract}
{This paper concerns the normalized ground states for the nonlinear Schr\"{o}dinger equation in the Bopp-Podolsky electrodynamics. This equation has a nonlocal nonlinearity and a mass supercritical power nonlinearity, both of which have deep impact on the geometry of the corresponding functional, and thus on the existence, limit behavior and stability of the normalized ground states. In the present study, the existence of critical points is obtained by a mountain-pass argument developed on the $L^2$-spheres. To be specific, we show that normalized ground states exist, provided that spherical radius of the $L^2$-spheres is sufficiently small. Then, by discussing the relation between the normalized ground states of the Schr\"{o}dinger-Bopp-Podolsky system and the classical Schr\"{o}dinger equation, we show a precise description of the asymptotic behavior of the normalized ground states as the mass vanishes or tends to infinity. Moreover, we discuss the radial symmetry and uniqueness of the normalized ground states. Finally, the strong instability of standing waves at the mountain-pass energy level is studied by constructing an equivalent minimizing problem. Also, as a byproduct, we prove that the mountain-pass energy level gives a threshold for global existence based on this equivalent minimizing problem.}
\end{abstract}

\maketitle

\renewcommand{\theequation}
{\thesection.\arabic{equation}}
\setcounter{equation}{0}
\section{Introduction} \noindent

When solving the alleged infinity problem in classical Maxwell theory, one would find it leads to the Bopp-Podolsky electromagnetic theory, which was developed independently by Bopp \cite{Bopp1940} and then by Podolsky \cite{Podolsky1942}. From the point of view of the electromagnetic field, the Bopp-Podolsky theory can be regarded as effective theory for short distances, and it is experimentally indistinguishable from the Maxwell theory for large distance, see in \cite{Aberqi2022, Farid2023, Frenkel1996}. In the Bopp-Podolsky electromagnetic theory, for the purely electrostatic case, a Schr\"{o}dinger field coupled with its electromagnetic field is usually described by the following system(see in \cite{dAvenia2019})
\begin{numcases}{ }
 i\hbar\partial_t\psi+\frac{\hbar^2}{2m}\Delta\psi-q^2\phi\psi+|\psi|^{p-2}\psi=0, \quad t\in\mathbb{R}, ~x\in\mathbb{R}^3,\label{main01}\\
 -\Delta\phi+a^2\Delta^2\phi=4\pi|\psi|^2,\label{main02}
\end{numcases}
where $\hbar$ is the Planck constant, $m$ is the meson mass, $q>0$ is the charge, $2<p<6$, and $a>0$ is the Bopp-Podolsky parameter that has the dimension of the inverse of mass. For more details of this model in physics, please refer to \cite{Bertin2017, Farid2023} and the references therein.

An important fact concerning system (\ref{main01})-(\ref{main02}) is that it can be transformed into a Schr\"{o}dinger equation with a nonlocal term (see in \cite{dAvenia2019}). To be specific, if we denote the space $\mathcal{D}$ as the completion of $C^{\infty}_c(\mathbb{R}^3)$ with respect to the norm ${(||\nabla\phi||^2_{L^2(\mathbb{R}^3)}+a^2||\Delta\phi||^2_{L^2(\mathbb{R}^3)})}^{1\over2}$(see in Section 2 for more details), then by the Riesz Theorem, for every fixed $\psi(x)\in H^1(\mathbb{R}^3)$, there exists a unique solution of Eq.(\ref{main02}) in $\mathcal{D}$, namely $\phi=\mathcal{K}\ast|\psi|^2$. Here $\ast$ represents the convolution in $\mathbb{R}^3$ and $\mathcal{K}(x)=\frac{1-e^{-\frac{|x|}{a}}}{|x|}$. Inserting the solution into Eq.(\ref{main01}), one get
\begin{equation}\label{main10}
i\hbar\partial_t\psi+\frac{\hbar^2}{2m}\Delta\psi-q^2(\mathcal{K}\ast|\psi|^2)\psi+|\psi|^{p-2}\psi=0.
\end{equation}
For simplicity, in this paper, we treat $a$ as a parameter and consider the following dimensionless form of Eq.(\ref{main10}) (see in \cite{Zheng2022}),
\begin{equation}\label{main6}
i\partial_t\psi+\Delta\psi-(\mathcal{K}\ast|\psi|^2)\psi+|\psi|^{p-2}\psi=0,\quad t\in\mathbb{R},~~ x\in\mathbb{R}^3.
\end{equation}
Obviously, this equation has a defocusing nonlocal term and a focusing local term, which makes it interesting both in mathematics and physics.

To find normalized ground states, one need to plug the standing waves of the form $e^{i\omega t}u(x)~(\omega\in\mathbb{R})$ into Eq.(\ref{main6}). Then the function $u(x)$ satisfies the stationary equation
\begin{equation}\label{main5}
-\Delta u+\omega u+(\mathcal{K}\ast|u|^2) u=|u|^{p-2}u.
\end{equation}
From this point of view, $\omega\in\mathbb{R}$ is considered as a known fixed frequency. However, a normalized ground state is a solution of Eq.(\ref{main5}) with $||u||_{L^2(\mathbb{R}^3)} = m>0$, which can be obtained as a critical point of the functional
$$E(u):=\frac12||\nabla u||^2_{L^2(\mathbb{R}^3)}+\frac14\int_{\mathbb{R}^3}(\mathcal{K}\ast|u|^2)|u|^2dx-\frac1p||u||^p_{L^p(\mathbb{R}^3)},$$
under the constraint condition
$$S(m):=\{u\in H^1(\mathbb{R}^3):~ ||u||_{L^2(\mathbb{R}^3)}=m\}.$$
In other words, we can consider the following minimization problem:
\begin{align}\label{min-pro1}
\mathcal{E}(m):=\inf\limits_{u\in S(m)}E(u).
\end{align}
Indeed, $\omega\in\mathbb{R}$ in this approach are unknown and can be interpreted as a Lagrange multiplier. It is clear the minimizers of $\mathcal{E}(m)$ are critical points of $E(u)$ restricted to $S(m)$, and thus solutions of Eq.(\ref{main5}).

In the last few years, the existence of the nontrivial solutions for the following equation,
\begin{equation}\label{main51}
-\Delta u+\omega u+b^2(\mathcal{K}\ast|u|^2) u=|u|^{p-2}u,
\end{equation}
has been studied extensively by variational methods. As mentioned before, $\omega$ can be regarded as a fixed frequency or an unknown Lagrange multiplier. For the first case, if one further wants to find the ground states, the general method is to minimize the action functional over the nontrivial solutions, which is called least the action solutions. Applying this method, by treating ``$a$" and ``$b$" as parameters, d'Avenia and Siciliano\cite{dAvenia2019} showed that Eq.(\ref{main51}) possesses a nontrivial solution when $p\in(2, 6)$ and $|b|$ small enough or $p\in(3, 6)$ and $b\neq0$. Meanwhile, they obtained that Eq.(\ref{main51}) does not admit any nontrivial solution for $p\geq6$ via Pohozaev identity. Moreover, in the radial case, they proved the solutions tend to the solutions of the classical Schr\"{o}dinger-Poisson system as $a\rightarrow0$. After that, Silva and Siciliano \cite{Silva2020} supplemented and improved some results in \cite{dAvenia2019} by the so called ``fibering approach". More precisely, they showed that when $p\in(2, 3]$, Eq.(\ref{main51}) has no solution for $b$ large enough and has two radial solutions for $b$ small enough. The other case is that $\omega\in\mathbb{R}$ is not fixed and appears as a Lagrange multiplier. In fact, this is from the sight of the conservation law of mass, and hence attracted many physicists attention on the ``normalized ground states". In \cite{LiuC2022}, Liu proved the existence of normalized solutions by minimizing method for Eq.(\ref{main51}) with $p\in(2, \frac83]$. And then,  He, Li and Chen \cite{He2022} applied standard scaling arguments and some basic properties of energy functional to extend the range of $p$ to $(2, \frac{10}3)$. When $|u|^{p-2}u$ is replaced by $\mu|u|^{p-2}u+|u|^4u$, inspired by the idea in \cite{Jeanjean2022}, Li, Chang and Feng \cite{Li2023} established the existence of normalized solutions for Eq.(\ref{main51}) with $\mu>0$, $b^2=1$ and $p\in(2, \frac83)$ by developing a constraint minimizing approach. Moreover, under the same conditions in \cite{Li2023}, Li and Zhang \cite{Li-Zhang2023} extended the same result to the range of $p$ to $(2, 6)$ for $\mu>0$ large enough by applying a different method, namely the mountain-pass lemma. Of course, there are still some studies on the existence of the normalized ground states for Eq.(\ref{main51}) with other situations, such as where the nonlocal term is focusing and $p=6$ (see in \cite{Peng2024}) or when the equation with Neumann boundary conditions on a connected, bounded, smooth open set (see in \cite{Afonso2021}).

What calls for special attention is that when $p<\frac{10}3$, the energy functional of Eq.(\ref{main51}) is bounded from below on the $L^2$-spheres. Hence, the authors of \cite{Afonso2021, LiuC2022, He2022} can apply the constrained variational problems to show the existence of the normalized ground states by searching the critical points of the energy functional on the $L^2$-spheres. But, when $p\geq\frac{10}3$, the energy functional is unbounded from below on the $L^2$-spheres, which leads to the inapplicability of the method for mass subcritical (i.e., $2<p<{3\over10}$). So we also apply the mountain-pass lemma to prove the existence of normalized solutions. However, different from the equation in \cite{Li-Zhang2023}, Eq.(\ref{main5}) does not have the extra focusing critical nonlinear local term $|u|^4u$, which contributes to the difficulty in our study. Specifically, the energy critical term combining with the mass supercritical term makes the focusing effect stronger than the defocusing effect, and consequently impacts on the variational structure of the equation. Thus, it makes the compactness of the Palais-Smale sequence clearer. Unfortunately, for Eq.(\ref{main5}), due to the absence of focusing critical nonlinear local term, the effect of the defocusing nonlocal term $(\mathcal{K}\ast|u|^2)u$ is more evident. Therefore, the interplay between the defocusing nonlocal term and the focusing mass supercritical local term makes the compactness of the Palais-Smale sequence more subtle. This is the main motivation of considering the normalized solutions for Eq.(\ref{main5}) with $p\in(\frac{10}3, 6)$ in the current study.

Firstly, we shall show that $E(u)$ has a mountain-pass geometry on $S(m)$, defined as follows:
\begin{definition}
Given $m>0$, we say that $E(u)$ has a mountain-pass geometry on $S(m)$, if there exists a constant $K_m>0$, such that
$$\gamma(m)=\inf_{g\in\Gamma_m}\max_{t\in[0, 1]}E(g(t))>\max\{\max_{g\in \Gamma_m}E(g(0)),~~~ \max_{g\in \Gamma_m}E(g(1))\}$$
holds in the set
$$\Gamma_m=\{g\in C([0, 1], S(m)):~ g(0)\in A_{K_m}, E(g(1))<0\},$$
where $A_{K_m}=\{u\in S(m):~ ||\nabla u||_{L^2(\mathbb{R}^3)}\leq K_m\}$.
\end{definition}

Our main result concerning the existence of the normalized solutions for Eq.(\ref{main5}) is given by the following theorem.

\begin{theorem}\label{critical-point}
Suppose $p\in(\frac{10}3, 6)$ and $m>0$, then $E(u)$ has a mountain-pass geometry on $S(m)$. Moreover, there exists $m_0>0$ such that for any $m\in(0, m_0)$, Eq.(\ref{main5}) has a couple of weak solution $(u_m, \omega_m)\in H^1(\mathbb{R}^3)\times\mathbb{R}^+$ with $||u_m||_{L^2(\mathbb{R}^3)}=m$ and $E(u_m)=\gamma(m)$. In particular, the solution $u_m$ is positive for all $x\in\mathbb{R}^3$, and there exist constants $C_1, C_2>0$ and $R>0$ such that
$$|u_m(x)|\leq C_1|x|^{-\frac34}e^{-C_2\sqrt{|x|}}, \quad \forall~ |x|\geq R.$$
\end{theorem}

In fact, there are some obstacles in studying the existence of critical
points for functional $E(u)$ on $S(m)$. First, the nonlocal term and the mass supercritical nonlinear term increase the difficulty in getting the boundedness of the Palais-Smale sequence. For this reason, we introduce the Pohozaev functional
\begin{align}\nonumber
P(u):=&||\nabla u||^2_{L^2(\mathbb{R}^3)}+\frac14\int_{\mathbb{R}^3}(\mathcal{K}\ast|u|^2) |u|^2dx-\frac{3(p-2)}{2p}||u||^p_{L^p(\mathbb{R}^3)}\\ \label{Pohozaev-identity}
&-\frac1{4a}\iint_{\mathbb{R}^3\times\mathbb{R}^3}
e^{-\frac{|x-y|}{a}}|u(x)|^2|u(y)|^2dxdy,
\end{align}
 and the set
$$V(m):=\{u\in S(m):~ P(u)=0\}.$$
Then, we show that
\begin{align}\label{Pohozaev-manifold}
\gamma(m)=\inf_{u\in V(m)}E(u).
\end{align}
Hence, we prove that each critical point of $E(u$) obtained by mountain-pass procedure must lie in $V(m)$. Furthermore, we can find some suitable sequence of paths $\{g_n\}_{n=1}^{\infty}\subset\Gamma_m$ such that
$$\max_{t\in[0, 1]}E(g_n(t))\rightarrow\gamma(m), \quad\hbox{as}\,\ n\rightarrow\infty.$$
Thus, a special Palais-Smale sequence $\{v_n\}_{n=1}^{\infty}\subset H^1(\mathbb{R}^3)$ at the level $\gamma(m)$, which converges to a point in $V(m)$, has been constructed. The Palais-Smale sequence obtained in this way ensures its boundedness. Next, we face the lack of compactness when we look for solutions with a prescribed $L^2$-norm. Owing to the working space in general does not embed compactly into any space $L^p(\mathbb{R}^3)$, recovering the compactness of the sequence $\{v_n\}_{n=1}^{\infty}$ may be troublesome. And it does not seem possible to apply the celebrated concentrate compactness principle of Lions in \cite{Lions1984-1, Lions1984-2} to exclude the ``dichotomy" case, and thus obtain the compactness. Motivated by \cite{Bellazzini2013}, we show that the function $\gamma(m)$ is strictly decreasing. Combining the fact $P(v_n) = o_n(1)$, which was ensured by the construction of the Palais-Smale sequence, the compactness of the Palais-Smale sequence follows.

Until now, the existence of the normalized ground states of Eq.(\ref{main5}) are obtained for $p\in(2,3)\cup(3,{10\over3})$ (see in \cite{He2022, LiuC2022}) and $p\in({10\over3},6)$ (see Theorem \ref{critical-point}). It is thus of interest to discuss the property of the normalized ground states. For the Schr\"{o}dinger equation, there have been some works about the concentration behavior of normalized ground states, see in \cite{Cazenave2003, Guo2014, Maeda2010}. What about Eq.(\ref{main5})? In order to answer this question, we start by recalling the sharp Gagliardo-Nirenberg inequality \cite{Weinstein1983}:
\begin{align}\label{G-N-ineq}
||u||^p_{L^p(\mathbb{R}^3)}\leq\frac{p}{2||Q||^{p-2}_{L^2(\mathbb{R}^3)}}||u||^{\frac{6-p}2}_{L^2(\mathbb{R}^3)}
||\nabla u||^{\frac{3(p-2)}2}_{L^2(\mathbb{R}^3)}, \quad\forall~u\in H^1(\mathbb{R}^3),
\end{align}
where $p\in[2, 6)$. The equality in (\ref{G-N-ineq}) holds for $u=Q$, where $Q$ is, up to translations, a unique positive radially symmetric solution of
\begin{align}\label{elliptic}
-\frac{3(p-2)}4\Delta Q+\frac{6-p}4Q=|Q|^{p-2}Q, \quad~x\in\mathbb{R}^3,
\end{align}
see in \cite{Kwong1989}. By Eq.(\ref{elliptic}) and the related Pohozaev identity, we see that $Q$ satisfies
\begin{align}\label{elliptic1}
||\nabla Q||^2_{L^2(\mathbb{R}^3)}=||Q||^2_{L^2(\mathbb{R}^3)}=\frac2{p}||Q||^p_{L^p(\mathbb{R}^3)}.
\end{align}

To state our another main result, as mentioned above, suppose that $u_m\in S(m)$ is either a minimizer of $\mathcal{E}(m)$ when $p\in(2, 3)\cup(3, \frac{10}3)$ or a minimizer of $\gamma(m)$ when $p\in(\frac{10}3, 6)$, then there exists $\omega_m\in\mathbb{R}$ such that $(u_m, \omega_m)$ is a couple of solution to Eq.(\ref{main5}).
Define
\begin{equation}\label{solution-set}
\Lambda:=\left\{(m, u_m, \omega_m):~ \hbox{for each}~ m,
\begin{array}{lll}
&E(u_m)=\mathcal{E}(m)~\hbox{when}~p\in(2, 3)\cup(3, \frac{10}3);\\
&E(u_m)=\gamma(m)~\hbox{when}~p\in(\frac{10}3, 6)
\end{array}
\right\}.
\end{equation}
Obviously, $\Lambda\neq\emptyset$.
\begin{theorem}\label{concentration-behavior}
For any sequence $\{(m_n, u_{m_n}, \omega_{m_n})\}^{\infty}_{n=1}\subset\Lambda$ with
$\{m_n\}^{\infty}_{n=1}$ satisfying one of the following conditions:
\begin{itemize}
\item[(a)] if $p\in(2, 3)$ or $p\in(\frac{10}3, 6)$, $m_n\rightarrow0^+$ as $n\rightarrow+\infty$;
\item[(b)] if $p\in(3, \frac{10}3)$, $m_n\rightarrow+\infty$ as $n\rightarrow+\infty$.
\end{itemize}
Then, there exists a sequence $\{z_n\}^{\infty}_{n=1}\subset\mathbb{R}^3$ such that
\begin{align}\label{converges-Q00}
\left[\frac4{3(p-2)}\left(\frac{||Q||_{L^2(\mathbb{R}^3)}}{m_n}\right)^{\frac43}\right]^{\frac3{10-3p}}
u_{m_n}\left(\left[\frac4{3(p-2)}\left(\frac{||Q||_{L^2(\mathbb{R}^3)}}{m_n}\right)^{p-2}\right]
^{\frac2{10-3p}}x+z_n\right)\rightarrow Q(x)
\end{align}
in $H^1(\mathbb{R}^3)$ as $n\rightarrow+\infty$. Moreover,
$$\lim_{n\rightarrow+\infty}\left(\frac{||Q||_{L^2(\mathbb{R}^3)}}{m_n}\right)^{\frac{4(p-2)}{10-3p}}\omega_{m_n}=\frac{6-p}{3(p-2)}\left(\frac{3(p-2)}4\right)^{\frac4{10-3p}}.$$
\end{theorem}

Theorem \ref{concentration-behavior} is proved in Section 6. Here, we give the main idea of the proof. We recall that
$$E(u)=I(u)+\frac14\int_{\mathbb{R}^3}(\mathcal{K}\ast|u|^2) u^2dx,$$
where $I(u)=\frac12||\nabla u||^2_{L^2(\mathbb{R}^3)}-\frac1{p}||u||^p_{L^p(\mathbb{R}^3)}$ is the energy functional related to the classical nonlinear Schr\"{o}dinger equation $-\Delta u+\omega u=|u|^{p-2}u$. It is known that for each $m>0$, $I(u)$ has a minimizer $v_m\in S(m)$ with
\begin{align}\label{min-pro2}
I(v_m)=\widetilde{\mathcal{E}}(m):=\inf_{u\in S(m)}I(u), \quad\hbox{when}~p\in(2, \frac{10}3)
\end{align}
and
\begin{align}\label{min-pro3}
I(v_m)=\beta(m):=\inf_{u\in \widetilde{V}(m)}I(u) ~\hbox{and}~v_m\in\widetilde{V}(m), \quad\hbox{when}~p\in(\frac{10}3, 6),
\end{align}
see in \cite{Cazenave2003, Jeanjean1997}. Here
$$\widetilde{V}(m)=\left\{u\in S(m):~ G(u):=||\nabla u||^2_{L^2(\mathbb{R}^3)}-\frac{3(p-2)}{2p}||u||^p_{L^p(\mathbb{R}^3)}=0\right\}.$$
To show the concentration behavior of the normalized ground states, we need to know more about $v_m$ and the related energy $\widetilde{\mathcal{E}}(m)$ and $\beta(m)$. Indeed, by scaling, we find that, up to translations,
$$v_m(x)=\frac{m}{||Q||_{L^2(\mathbb{R}^3)}}\left[\frac{3(p-2)}4\left(\frac{m}{||Q||_{L^2(\mathbb{R}^3)}}\right)^{p-2}\right]^{\frac3{10-3p}}
Q\left(\left[\frac{3(p-2)}4\left(\frac{m}{||Q||_{L^2(\mathbb{R}^3)}}\right)^{p-2}\right]^{\frac2{10-3p}}x\right),$$
(see Lemma \ref{relationship} below). For the sequence $\{(m_n, u_{m_n}, \omega_{m_n})\}^{\infty}_{n=1}$
given in Theorem \ref{concentration-behavior}, by using the accurate expression of $v_{m_n}$, we deduce that $\frac{I(u_{m_n})}{\widetilde{\mathcal{E}}(m_n)}\rightarrow1$ (or $\frac{I(u_{m_n})}{\beta(m_n)}\rightarrow1$), which implies that $u_{m_n}$ behaves like $v_{m_n}$.

It is known that $Q$ is the unique positive radially symmetric solution of Eq.(\ref{elliptic}). According to Theorem \ref{concentration-behavior}, a natural question arises: is the solution to Eq.(\ref{main5}) unique? To our knowledge, in $\mathbb{R}^3$, the uniqueness heavily depends on the radial symmetry property of the solution. However, both the moving plane method (see in \cite{Gidas1979, Lieb1997, Lopes2008}) and Schwartz rearrangement (see in \cite{Lieb2001}), which are the two main tools used to prove radial symmetry, do not work in our case since the potential energy in Eq.(\ref{main5}) is defocusing on the nonlocal term and focusing on the local term. The first breakthrough in this area was made by Georgiev \cite{Georgiev2012}, who constructed a implicit function framework to obtain radial symmetry of normalized ground states for the Schr\"{o}dinger-Poisson system with a mass subcritical power nonlinearity (i.e. $p\in(2, 3)$). Inspired by Georgiev \cite{Georgiev2012}, in this paper, we show the radiality and uniqueness of the normalized ground states for Eq.(\ref{main5}) for the almost all range of $p$ in the energy space (i.e. $p\in(2,3)\cup(3,\frac{10}3)\cup(\frac{10}3, 6)$).

\begin{theorem}\label{unique-normalized-solution}
For Eq.(\ref{main5}) with $2<p<6$, we have
\begin{itemize}
 \item[(i)] if $p\in(2, 3)$, then there exists a sufficiently small constant $m_1>0$ such that for all $0<m<m_1$, every minimizer to (\ref{min-pro1}) is positive, radially symmetric up to translations, and unique up to translations;
 \item[(ii)] if $p\in(3, \frac{10}3)$, then there exists a sufficiently large constant $m_2>0$ such that for all $m>m_2$, every minimizer to (\ref{min-pro1}) is positive, radially symmetric up to translations, and unique up to translations;
 \item[(iii)] if $p\in(\frac{10}3, 6)$, then there exists a sufficiently small constant $0<m_3\leq m_0$ such that for all $0<m<m_3$, every minimizer to (\ref{Pohozaev-manifold}) is positive, radially symmetric up to translations, and unique up to translations.
\end{itemize}
\end{theorem}

With the existence and the uniqueness of the normalized solution in hand, we now turn to the evolution equation (\ref{main6}) with initial condition $\psi(0, x)=\psi_0\in H^1(\mathbb{R}^3)$. Zheng \cite{Zheng2022} proved the local well-posedness in $H^1(\mathbb{R}^3)$ when $p\in(2, 6)$ via the fixed point argument with Strichartz's estimate, and the finite time blow up result for the solutions with negative energy when $p\in(\frac{10}3, 6)$ by the Principles of Uncertainty. Based on the above results, when $p\in(\frac{10}3, 6)$, the sharp conditions for the global existence and blow up below the energy threshold was obtained by Gao \cite{Gao2023}. Moreover, for $p\in(2, \frac83)\cup(\frac83, 6)$, she applied the result of Murphy and Nakanishi \cite{Murphy2021} to get that without adding any decay assumption, the solution of Eq.(\ref{main6}) does not scatter in $H^1(\mathbb{R}^3)$. As far as the stability of the normalized solutions are concerned, when $p\in(2, \frac{10}3)$, He, Li and Chen \cite{He2022} proved the orbitally stable by using the approach of Cazenave-Lions \cite{Cazenave1982}. Is the normalized solution stable for $p\in(\frac{10}3,6)$? This is the problem we will focus on in the rest of this paper. In fact, we prove that the standing waves corresponding to elements of $\mathcal{M}_m:=\left\{u_m\in V(m):~ E(u_m)=\inf\limits_{u\in V(m)}E(u)\right\}$ are unstable.
\begin{theorem}\label{unstable}
Let $p\in(\frac{10}3, 6)$ and $m>0$. For each $u_m\in\mathcal{M}_m$, the standing wave $e^{i\omega_m t}u_m$ of Eq.(\ref{main6}), where $\omega_m\in\mathbb{R}$ is the Lagrange multiplier, is strongly unstable in the following sense: for any $\varepsilon>0$, there exists $\psi_0\in H^1(\mathbb{R}^3)$ which is positive and exponentially decaying (as $|x|\rightarrow\infty$), such that $||\psi_0-u_m||_{H^1(\mathbb{R}^3)}<\varepsilon$, and the solution $\psi$ of Eq.(\ref{main6}) with initial data $\psi_0$ blows up in finite time.
\end{theorem}

Finally, as a byproduct, the global existence of certain solutions for Eq.(\ref{main6}) with $p\in(\frac{10}3, 6)$ is obtained in this paper.

\begin{theorem}\label{global}
Let $p\in(\frac{10}3, 6)$ and initial data $\psi_0\in H^1(\mathbb{R}^3)$ satisfy
$||\psi_0||_{L^2}=m$. If $$P(\psi_0)>0 \quad\hbox{and}\quad E(\psi_0)<\gamma(m),$$
then the solution of Eq.(\ref{main6}) associated with initial data $\psi_0$ exists globally in time.
\end{theorem}

This paper is organized as follows. In Section 2, we introduce some lemmas such as the local well-posedness, the Hardy-Littlewood-Sobolev inequality. In Section 3, we establish the mountain-pass geometry of $E(u)$ on $S(m)$. In Section 4, we prove the convergence of the Palais-Smale sequence and we conclude the proof of Theorem \ref{critical-point}. In Section 5, we give the behaviour of the function $m\mapsto\gamma(m)$. In Section 6, we discuss the mass concentration behavior of the normalized ground states and prove Theorem \ref{concentration-behavior}. In Section 7, we show the radial symmetry and uniqueness of the normalized ground states for Eq.(\ref{main5}). In Section 8, we show the strong instability of standing waves and the global existence of Eq.(\ref{main6}).

Throughout the paper we make use of the following notations:
\begin{itemize}
 \item $||\cdot||_{L^r}$ denotes the usual norm of $L^r(\mathbb{R}^3)$ for $r\in [1, \infty)$, $||\cdot||_{L^\infty}$ denotes the norm of $L^{\infty}(\mathbb{R}^3)$, and $||\cdot||_{H^1}$ denotes the usual norm of $H^1(\mathbb{R}^3)$;
 \item we use the symbol $o_n(1)$ for a vanishing sequence in the specified space;
 \item if not specified, the domain of the integrals is $\mathbb{R}^3$;
 \item we use $C, C_1, C_2, \cdots$ to denote suitable positive constants whose value may also change
     from line to line;
 \item $B_R(x_0)$ denotes the ball centered at $x_0$ with radius $R>0$.
\end{itemize}

\renewcommand{\theequation}
{\thesection.\arabic{equation}}
\setcounter{equation}{0}
\section{Preliminaries}

In this section, we will give some preliminary results. Let us begin with recalling the local well-posedness for Eq.(\ref{main6}) obtained by Zheng \cite{Zheng2022}.

\begin{proposition}[Local well-posedness, \cite{Zheng2022}]\label{local}
Let $p\in(2, 6)$, $\psi_0\in H^1(\mathbb{R}^3)$. There exists a $T$ only depends on $||\psi_0||_{H^1} $, such that Eq.(\ref{main6}) has a unique solution
$$\psi\in\left\{\psi\in C([0,T], H^1(\mathbb{R}^3)):~ \psi, \nabla\psi\in L_t^{\infty}L_x^2\cap L_t^{\frac{4p}{3(p-2)}}L^p_x([0, T]\times\mathbb{R}^3)\right\}.$$
Let $[0, T^*)$ be the maximal time interval on which the solution $\psi$ is well-defined, if $T^*<\infty$, then $$||\nabla\psi||_{L^2}\rightarrow\infty~~ \hbox{as}~~ t\rightarrow T^*.$$
Moreover, the solution $\psi(t)$ enjoys conservation of mass and energy, i.e., $M(\psi(t))=||\psi(t)||^2_{L^2}=||\psi_0||^2_{L^2}$ and $E(\psi(t))=E(\psi_0)$ for all $t\in[0, T^*)$.
\end{proposition}

\begin{lemma}[Hardy-Littlewood-Sobolev inequality, \cite{Lieb2001}]\label{HLS}
Let $s, r>1$ with $\frac1s+\frac1r+\frac13=2$. For $u\in L^s(\mathbb{R}^3)$, $v\in L^r(\mathbb{R}^3)$, there exists a constant $C(s, r)$ such that
$$\iint\frac{|u(x)v(y)|}{|x-y|}dxdy\leq C(s, r)||u||_{L^s}||v||_{L^r}.$$
\end{lemma}

\begin{corollary}
It is clear that
$$\int(\mathcal{K}\ast|u|^2)|u|^2dx\leq\int(\frac1{|x|}\ast|u|^2)|u|^2dx\leq C||u||^4_{L^{\frac{12}5}}.$$
\end{corollary}

 As we know, the homogeneous Sobolev spcae $\dot{H}^1(\mathbb{R}^3)$ is continuously embedded into $L^6(\mathbb{R}^3)$, and there exists the best constant $\mathcal{S}>0$ such that
\begin{align}\label{best-constant}
\mathcal{S}=\inf_{u\in \dot{H}^1(\mathbb{R}^3)\setminus\{0\}}\frac{\int|\nabla u|^2dx}{(\int|u|^6dx)^{\frac13}}.
\end{align}
Let $\mathcal{D}$ be the completion of $C^{\infty}_c$ with respect to the norm $||\cdot||_{\mathcal{D}}$ induced by the scalar product
$$\langle\varphi, \psi\rangle_{\mathcal{D}}:=\int\nabla\varphi\cdot\nabla\psi dx+a^2\int\Delta\varphi\Delta\psi dx.$$
Then $\mathcal{D}$ is an Hilbert space continuously embedded into $\dot{H}^1(\mathbb{R}^3)$ and consequently in $L^6(\mathbb{R}^3)$.

\begin{lemma}[\cite{dAvenia2019}]\label{continuously-embedded}
The Hilbert space $\mathcal{D}$ is continuously embedded in $L^{\infty}(\mathbb{R}^3)$.
\end{lemma}

\begin{lemma}[\cite{dAvenia2019}]\label{dense}
The space $C^{\infty}_c(\mathbb{R}^3)$ is dense in $\mathcal{A}$, where
$$\mathcal{A}:=\{\phi\in \dot{H}^1(\mathbb{R}^3):~ \Delta\phi\in L^2(\mathbb{R}^3)\}$$
normed by $\sqrt{\langle\phi, \phi\rangle_{\mathcal{D}}}$ and, consequently $\mathcal{D}=\mathcal{A}$.
\end{lemma}

\begin{lemma}[\cite{dAvenia2019}]\label{property1}
For all $y\in\mathbb{R}^3$, $\mathcal{K}(\cdot-y)$ solves in the sense of distributions
$$-\Delta\phi+a^2\Delta^2\phi=4\pi\delta_y.$$
Moreover,
\begin{itemize}
 \item[(1)] if $f\in L^1_{loc}(\mathbb{R}^3)$ and, for a.e. $x\in\mathbb{R}^3$, the map $\mathbb{R}^3\ni y\mapsto\frac{f(y)}{|x-y|}$ is summable, then $\mathcal{K}\ast f\in L^1_{loc}(\mathbb{R}^3)$;
 \item[(2)] if $f\in L^s(\mathbb{R}^3)$ with $1\leq s<\frac32$, then $\mathcal{K}\ast f\in L^q(\mathbb{R}^3)$ for $q\in(\frac{3s}{3-2s}, +\infty]$.
\end{itemize}
In both cases $\mathcal{K}\ast f$ solves
$$-\Delta\phi+a^2\Delta^2\phi=4\pi f$$
in the sense of distributions, and we have the following distributional derivatives
$$\nabla(\mathcal{K}\ast f)=(\nabla\mathcal{K})\ast f \quad\hbox{and}\quad \Delta(\mathcal{K}\ast f)=(\Delta\mathcal{K})\ast f \quad\hbox{a.e. in}~~ \mathbb{R}^3.$$
\end{lemma}

Denote
$$\phi_u:=\mathcal{K}\ast |u|^2=\int\frac{1-e^{-\frac{|x-y|}a}}{|x-y|}|u(y)|^2dy,$$
then we have the following lemma:

\begin{lemma}[\cite{dAvenia2019}]\label{property2}
For every $u\in H^1({\mathbb{R}^3})$, we have
\begin{itemize}
 \item[(1)] for every $y\in\mathbb{R}^3$, $\phi_{u(\cdot+y)}=\phi_u(\cdot+y)$;
 \item[(2)] $\phi_u\geq 0$;
 \item[(3)] for every $s\in(3, +\infty]$, $\phi_u\in L^s(\mathbb{R}^3)\cap C_0(\mathbb{R}^3)$;
 \item[(4)] for every $s\in(\frac32, +\infty]$, $\nabla\phi_u=\nabla\mathcal{K}\ast |u|^2\in L^s(\mathbb{R}^3)\cap C_0(\mathbb{R}^3)$;
 \item[(5)] $\phi_u\in\mathcal{D}$;
 \item[(6)] $||\phi_u||_{L^6}\leq C||u||^2_{H^1}$;
 \item[(7)] $\phi_u$ is the unique minimizer of the functional
  $$F(\phi)=\frac12||\nabla\phi||^2_{L^2}+\frac{a^2}2||\Delta\phi||^2_{L^2}-\int\phi|u|^2dx, \quad \phi\in\mathcal{D}.$$
\end{itemize}
Moreover, if $v_n\rightharpoonup v$ in $H^1(\mathbb{R}^3)$, then $\phi_{v_n}\rightharpoonup\phi_v$ in $\mathcal{D}$.
\end{lemma}

Now, we define the function $\Psi: H^1(\mathbb{R}^3)\rightarrow\mathbb{R}$ by
$$\Psi(u)=\int\phi_u(x)|u(x)|^2 dx.$$
It is clear that for all fixed $u\in H^1(\mathbb{R}^3)$ then $\Psi(u(\cdot+y))=\Psi(u)$ for any $y\in\mathbb{R}^3$ and that $\Psi$ is weakly lower semi-continuous in $H^1(\mathbb{R}^3)$. The next lemma shows that the functional $\Psi$ and its derivative $\Psi^{'}$ have the splitting property, which is similar to the well-known Br\'{e}zis-Lieb lemma.

\begin{lemma}[\cite{dAvenia2019}]\label{BL}
If $u_n\rightharpoonup u$ in $H^1(\mathbb{R}^3)$ and $u_n\rightarrow u$ a.e. in $\mathbb{R}^3$, then
\begin{itemize}
 \item[(1)] $\Psi(u_n-u)=\Psi(u_n)-\Psi(u)+o_n(1)$;
 \item[(2)] $\Psi^{'}(u_n-u)=\Psi^{'}(u_n)-\Psi^{'}(u)+o_n(1).$
\end{itemize}
\end{lemma}

We end this section by the following lemma.

\begin{lemma}[\cite{Lions1984-2}]\label{vanishing}
Assume that $\{u_n\}_{n=1}^{\infty}$ is bounded sequence in $H^1(\mathbb{R}^3)$ such that
$$\limsup_{n\rightarrow\infty}\int_{B_R(y)}|u_n(x)|^2dx=0$$
for some $R>0$. Then $u_n\rightarrow0$ in $L^r(\mathbb{R}^3)$ for every $r$, with $2<r<6$.
\end{lemma}

\renewcommand{\theequation}
{\thesection.\arabic{equation}}
\setcounter{equation}{0}
\section{The mountain-pass geometry on the constraint} \noindent

In this section, we will discuss the mountain-pass geometry of the functional $E(u)$ on the $L^2$-constraint $S(m)$. We first introduce the scaling, for $u\in S(m)$, we set $\kappa(u, \theta)(x)=e^{\frac32 \theta}u(e^{\theta}x), \theta\in\mathbb{R}$. Then, we establish some lemmas as follows.

\begin{lemma}\label{lemma1}
For $m>0$, suppose $u\in S(m)$ and $p\in(\frac{10}3, 6)$, we have
\begin{itemize}
 \item[(1)] $||\nabla\kappa(u, \theta)||_{L^2}\rightarrow0^{+}$ and $E(\kappa(u, \theta))\rightarrow0^+$ as $\theta\rightarrow-\infty$;
 \item[(2)] $||\nabla\kappa(u, \theta)||_{L^2}\rightarrow+\infty$ and $E(\kappa(u, \theta))\rightarrow-\infty$ as $\theta\rightarrow+\infty$;
 \item[(3)] there exists $C_0>0$ such that $P(u)>0$ if $||\nabla u||_{L^2}\leq C_0$, and $||u||^p_{L^p}\geq C_0$ if $P(u)=0$;
 \item[(4)] if $E(u)<0$, then $P(u)<0$.
\end{itemize}
\end{lemma}

\begin{proof}
One can easily get that
\begin{align}\label{scaling1}
||\nabla\kappa(u, \theta)||^2_{L^2}=e^{2\theta}||\nabla u||^2_{L^2},
\end{align}
then
$$||\nabla\kappa(u,\theta)||_{L^2}\rightarrow0^{+} \quad \hbox{as} \quad \theta\rightarrow-\infty.$$

Note that
\begin{align}\nonumber
E(\kappa(u, \theta))=&\frac{e^{2\theta}}2||\nabla u||^2_{L^2}+\frac{e^{\theta}}4\iint\frac{1-e^{-\frac{|x-y|}{ae^\theta}}}{|x-y|}|u(x)|^2|u(y)|^2dxdy\\\label{scaling2}
&-\frac{e^{\frac{3(p-2)}2\theta}}{p}||u||^p_{L^p}.
\end{align}
Since $\frac{3(p-2)}2>2$, we have
$$E(\kappa(u, \theta))\rightarrow0^+ \quad\hbox{as}\quad \theta\rightarrow-\infty.$$
Thus, (1) holds. Similarly, we can easily get that (2) holds.

We note that
\begin{align}\nonumber
E(u)-\frac{2}{3(p-2)}P(u)=&\frac{3p-10}{6(p-2)}||\nabla u||^2_{L^2}+\left(\frac14-\frac1{6(p-2)}\right)\iint\frac{1-e^{-\frac{|x-y|}a}}{|x-y|}|u(x)|^2|u(y)|^2dxdy\\\label{EP}
&+\frac{1}{6a(p-2)}\iint e^{-\frac{|x-y|}a}|u(x)|^2|u(y)|^2dxdy.
\end{align}
Due to $\frac{3p-10}{6(p-2)}>0$, $\frac14-\frac1{6(p-2)}>0$ and $a>0$, we get $E(u)-\frac{2}{3(p-2)}P(u)$ is always positive. Thus, (4) holds.

Moreover, thanks to the fact $\frac{1-e^{-\frac{|x-y|}a}}{|x-y|}-\frac1a e^{-\frac{|x-y|}a}\geq0$ and (\ref{G-N-ineq}), there exists a constant $C(p)>0$ such that
\begin{align*}
P(u)&\geq||\nabla u||^2_{L^2}-\frac{3(p-2)}{2p}||u||^p_{L^p}\\
&\geq||\nabla u||^2_{L^2}-C(p)||\nabla u||^{\frac{3(p-2)}2}_{L^2}||u||^{\frac{6-p}2}_{L^2}.
\end{align*}
The fact that $\frac{3(p-2)}2>2$ ensures that $P(u)>0$ for sufficiently small $||\nabla u||_{L^2}$.\\
Also, if $P(u)=0$, we get
\begin{align*}
||u||^p_{L^p}&=\frac{2p}{3(p-2)}\left[||\nabla u||^2_{L^2}+\frac14\iint\left(\frac{1-e^{-\frac{|x-y|}a}}{|x-y|}-\frac1a e^{-\frac{|x-y|}a}\right)|u(x)|^2|u(y)|^2dxdy\right]\\
&\geq\frac{2p}{3(p-2)}||\nabla u||^2_{L^2},
\end{align*}
and this ends the proof of (3).
\end{proof}

\begin{lemma}\label{unique0}
We denote
\begin{align*}
h_1(\theta)=4e^{\frac{3(p-2)}{2}\theta}-3(p-2)e^{2\theta}+3p-10,
\end{align*}
\begin{align*}
h_2(\theta, r)=\frac{2e^{\frac{3(p-2)}{2}\theta}+3p-8}{3(p-2)}\cdot\frac{1-e^{-\frac{r}{a}}}{r}-\frac{e^{\theta}\left(1-e^{-\frac{r}{ae^{\theta}}}\right)}{r}
-\frac{2\left[e^{\frac{3(p-2)}{2}\theta}-1\right]}{3a(p-2)}e^{-\frac{r}{a}},
\end{align*}
where $p\in(\frac{10}3, 6)$ and $a>0$, $r>0$, $\theta\in\mathbb{R}$. Then,
$$h_1(0)=h_2(0, r)=0, \quad h_1(\theta)>0, \quad h_2(\theta, r)>0, \quad\forall~ \theta\in(-\infty, 0)\cup(0, +\infty),~r>0.$$
\end{lemma}

\begin{proof}
By simple calculations, one has
\begin{align*}
h'_1(\theta)=6(p-2)e^{2\theta}\left(e^{\frac{3p-10}{2}\theta}-1\right),
\end{align*}
\begin{align*}
\frac{\partial}{\partial\theta}h_2(\theta, r)=e^{\theta}\left[e^{\frac{3p-8}{2}\theta}\left(\frac{1-e^{-\frac{r}{a}}}{r}-\frac1{a}e^{-\frac{r}{a}}\right)-
\left(\frac{1-e^{-\frac{r}{ae^{\theta}}}}{r}-\frac1{ae^{\theta}}e^{-\frac{r}{ae^{\theta}}}\right)\right].
\end{align*}
Thus, $h_1(0)=h_2(0, r)=0$ and $h_1(\theta)>0$ for all $\theta\in(-\infty, 0)\cup(0, +\infty)$. Let $h_3(\theta, r)=\frac{1-e^{-\frac{r}{ae^{\theta}}}}{r}-\frac1{ae^{\theta}}e^{-\frac{r}{ae^{\theta}}}$, then
$$\frac{\partial}{\partial\theta}h_3(\theta, r)=-\frac{r}{a^2e^{2\theta}}e^{-\frac{r}{ae^{\theta}}}\leq0, ~~\hbox{for all}~ \theta\in\mathbb{R}~\hbox{and}~r>0.$$
Since $h_3(0)\geq0$, $h_3(\theta, r)\geq0$ for all $\theta\in\mathbb{R}$ and $r>0$, we have
\begin{equation} \nonumber
\frac{\partial}{\partial\theta}h_2(\theta, r)\left\{
\begin{aligned}
&<e^{\theta}h_3(\theta)\left(e^{\frac{3p-8}{2}\theta}-1\right)<0, &\theta\in(-\infty, 0),\\
&>e^{\theta}h_3(0)\left(e^{\frac{3p-8}{2}\theta}-1\right)>0, &\theta\in(0, +\infty),
\end{aligned}
\right.
\end{equation}
which implies that $h_2(\theta, r)\geq h_2(0, r)=0$ for all $\theta\in\mathbb{R}$ and $r>0$.
\end{proof}

\begin{lemma}\label{unique1}
Let $p\in(\frac{10}3, 6)$, then for all $u\in H^1(\mathbb{R}^3)$ and $\theta\in\mathbb{R}$, we have
\begin{align}\label{unique2}
E(u)\geq E(\kappa(u, \theta))+\frac{2\left[1-e^{\frac{3(p-2)}{2}\theta}\right]}{3(p-2)}P(u)+\frac{h_1(\theta)}{6(p-2)}||\nabla u||^2_{L^2},
\end{align}
where $h_1(\theta)$ is defined in Lemma \ref{unique0}.
\end{lemma}

\begin{proof}
By Lemma \ref{unique0}, we have
\begin{align*}
E(\kappa(u, \theta))-E(u)=&\frac{e^{2\theta}-1}2||\nabla u||^2_{L^2}+\frac{e^{\theta}}4\iint\frac{1-e^{-\frac{|x-y|}{ae^{\theta}}}}{|x-y|}|u(x)|^2|u(y)|^2dxdy\\
&-\frac14\iint\frac{1-e^{-\frac{|x-y|}a}}{|x-y|}|u(x)|^2|u(y)|^2dxdy+\frac{1-e^{\frac{3(p-2)}{2}\theta}}{p}||u||^p_{L^p}\\
=&\frac{2\left[e^{\frac{3(p-2)}{2}\theta}-1\right]}{3(p-2)}P(u)-\left\{\frac{2\left[e^{\frac{3(p-2)}{2}\theta}-1\right]}{3(p-2)}-\frac{e^{2\theta}-1}2\right\}||\nabla u||^2_{L^2}\\
&-\frac14\iint\left\{\frac{2e^{\frac{3(p-2)}{2}\theta}+3p-8}{3(p-2)}\cdot\frac{1-e^{-\frac{|x-y|}{a}}}{|x-y|}-\frac{e^{\theta}\left(1-e^{-\frac{|x-y|}{ae^{\theta}}}\right)}{|x-y|}\right.\\
&\left.-\frac{2\left[e^{\frac{3(p-2)}{2}\theta}-1\right]}{3a(p-2)}e^{-\frac{|x-y|}{a}}\right\}|u(x)|^2|u(y)|^2dxdy\\
\leq&\frac{2\left[e^{\frac{3(p-2)}{2}\theta}-1\right]}{3(p-2)}P(u)-\frac{h_1(\theta)}{6(p-2)}||\nabla u||^2_{L^2},\quad\forall~u\in H^1(\mathbb{R}^3), ~\theta\in\mathbb{R}.
\end{align*}
This shows that (\ref{unique2}) holds.
\end{proof}

\begin{lemma}\label{unique}
Let $p\in(\frac{10}3, 6)$ and $u\in S(m)$ be arbitrary but fixed. We have the following statements hold.
\begin{itemize}
 \item[(1)] There exists a unique number $\theta(u)\in\mathbb{R}$ such that the function $f_u: \mathbb{R}\rightarrow\mathbb{R}$ defined by $f_u(\theta)=E(\kappa(u, \theta))$ reaches its maximum. Moreover, $\kappa(u, \theta(u))\in V(m)$.
 \item[(2)] $E(\kappa(u, \theta))<E(\kappa(u, \theta(u)))$ for any $\theta\neq\theta(u)$. In particular, $E(\kappa(u, \theta(u)))>0$.
\end{itemize}
\end{lemma}

\begin{proof}
Since
$$f_u(\theta)=\frac{e^{2\theta}}2||\nabla u||^2_{L^2}+\frac{e^{\theta}}4\iint\frac{1-e^{-\frac{|x-y|}{ae^\theta}}}{|x-y|}|u(x)|^2 |u(y)|^2 dxdy
-\frac{e^{\frac{3(p-2)}2\theta}}{p}||u||^p_{L^p},$$
we have that
\begin{align*}
f'_u(\theta)=&e^{2\theta}||\nabla u||^2_{L^2}-\frac{3(p-2)}{2p}e^{\frac{3(p-2)}2\theta}||u||^p_{L^p}+\frac{e^{\theta}}4\iint\frac{1-e^{-\frac{|x-y|}{ae^\theta}}}{|x-y|}|u(x)|^2|u(y)|^2dxdy\\
&-\frac1{4a}\iint e^{-\frac{|x-y|}{ae^\theta}}|u(x)|^2|u(y)|^2dxdy\\
=&P(\kappa(u, \theta)).
\end{align*}
By Lemma \ref{lemma1}, we know that there exists at least a $\theta(u)\in\mathbb{R}$ such that $f'_u(\theta(u))=0$, which means that $\kappa(u, \theta(u))\in V(m)$.

Next, we claim that $\theta(u)$ is unique for any $u\in S(m)$. Otherwise, for any given $u\in S(m)$, there exists $\theta_1\neq\theta_2$ such that $\kappa(u, \theta_1), \kappa(u, \theta_2)\in V(m)$, that is, $P(\kappa(u, \theta_1))=P(\kappa(u, \theta_2))=0$, then Lemmas \ref{unique0} and \ref{unique1} lead to
\begin{align*}
E(\kappa(u, \theta_1))&>E(\kappa(u, \theta_2))+\frac{2\left[e^{\frac{3(p-2)}{2}\theta_1}-e^{\frac{3(p-2)}{2}\theta_2}\right]}{3(p-2)e^{\frac{3(p-2)}{2}\theta_1}}
P(\kappa(u, \theta_1))=E(\kappa(u, \theta_2))\\
&>E(\kappa(u, \theta_1))+\frac{2\left[e^{\frac{3(p-2)}{2}\theta_2}-e^{\frac{3(p-2)}{2}\theta_1}\right]}{3(p-2)e^{\frac{3(p-2)}{2}\theta_2}}P(\kappa(u, \theta_2))=E(\kappa(u, \theta_1)).
\end{align*}
This contradiction shows that $\theta(u)\in\mathbb{R}$ is unique for any $u\in S(m)$. Then, (2) is a direct consequence of the proof above.
\end{proof}

\begin{lemma}\label{AB}
Let $p\in(\frac{10}3, 6)$ , then there exists $C_1>0$ and $\bar{m}>0$ such that for all $m\in(0, \bar{m})$,
$$0<\sup_{u\in A}E(u)<\inf_{u\in B}E(u),$$
where
$$A=\{u\in S(m):~ ||\nabla u||_{L^2}\leq C_1\}, \quad B=\{u\in S(m):~ ||\nabla u||_{L^2}=2C_1\}.$$
\end{lemma}

\begin{proof}
From (\ref{G-N-ineq}), for $u\in S(m)$, we have
\begin{align*}
E(\kappa(u, \theta))-E(u)=&\frac12||\nabla\kappa(u, \theta)||^2_{L^2}-\frac12||\nabla u||^2_{L^2}+\frac14\int\phi_{\kappa(u, \theta)}|\kappa(u, \theta)|^2dx-\frac14\int\phi_u |u|^2dx\\
&-\frac1p||\kappa(u, \theta)||^p_{L^p}+\frac1p||u||^p_{L^p}\\
\geq&\frac12||\nabla\kappa(u, \theta)||^2_{L^2}-\frac12||\nabla u||^2_{L^2}-C||u||^4_{L^2}-\frac1p||\kappa(u, \theta)||^p_{L^p}\\
\geq&\frac12||\nabla\kappa(u, \theta)||^2_{L^2}-\frac12||\nabla u||^2_{L^2}-Cm^4-\frac{C_{GN}}pm^{\frac{6-p}2}||\nabla\kappa(u, \theta)||^{\frac{3(p-2)}2}_{L^2}.
\end{align*}
Let $||\nabla u||_{L^2}\leq C_1$ and $||\nabla\kappa(u, \theta)||_{L^2}=2C_1$, where $C_1$ is a positive constant to be chosen later. Then, for $m<\bar{m}:=\left(\frac{C^2_1}{16C}\right)^{\frac14}$, we obtain
\begin{align*}
E(\kappa(u, \theta))-E(u)\geq&\frac32C^2_1-Cm^4-\frac{C_{GN}}p2^{\frac{3(p-2)}2}m^{\frac{6-p}2}C^{\frac{3(p-2)}2}_1\\
\geq&\frac32C^2_1-\frac1{16}C^2_1-\frac{C_{GN}}p2^{\frac{3(p-2)}2}\left(\frac{C^2_1}{16C}\right)^{\frac{6-p}8}C^{\frac{3(p-2)}2}_1.
\end{align*}
Put
$$C_1=\left(\frac{23pC^{\frac{6-p}8}}{2^{\frac{4(p-1)}2}C_{GN}}\right)^{\frac4{5p-14}}>0,$$
we know
$$\frac32 C^2_1-Cm^4-\frac{C_{GN}}p2^{\frac{3(p-2)}2}m^{\frac{6-p}2}C^{\frac{3(p-2)}2}_1>0.$$
Hence, Lemma \ref{AB} is proved.
\end{proof}

From Lemma \ref{AB}, we deduce the following corollary.
\begin{corollary}\label{AB1}
For $C_1>0$ and $\bar{m}>0$ given in Lemma \ref{AB}, let $u\in S(m)$ and $||\nabla u||_{L^2}\leq C_1$, then $E(u)>0$. Moreover,
$$E_{*}:=\inf\left\{E(u): ~u\in S(m), ||\nabla u||_{L^2}=\frac12C_1\right\}>0.$$
\end{corollary}

\begin{proof}
Analogously to the proof of Lemma \ref{AB}, we can get that
$$E(u)\geq\frac12||\nabla u||^2_{L^2}-Cm^4-\frac{C_{GN}}pm^{\frac{6-p}2}||\nabla u||^{\frac{3(p-2)}2}_{L^2}.$$
By the definition of $C_1$, we conclude that the corollary holds.
\end{proof}

By using Lemma \ref{lemma1}, for $u\in S(m)$, there exists $\theta_1<0$ and $\theta_2>0$ satisfying
$$||\nabla\kappa(u, \theta_1)||_{L^2}\leq\frac12C_1, \quad ||\nabla\kappa(u, \theta_2)||_{L^2}>C_1,$$
and
$$E(\kappa(u, \theta_1))>0, \quad E(\kappa(u, \theta_2))<0.$$

In the following, we will discuss the mountain-pass geometry of the functional $E(u)$ on the $S(m)$. Now, we
introduce an auxiliary functional
$$\tilde{E}: ~S(m)\times\mathbb{R}\rightarrow\mathbb{R}, \quad (u, \theta)\mapsto E(\kappa(u, \theta)),$$
and the set of paths
\begin{align}\label{path1}
\tilde{\Gamma}(m):=\{\tilde{g}\in C([0, 1], S(m)\times\mathbb{R}): ~ \tilde{g}(0)=(u_1, 0), \tilde{g}(1)=(u_2, 0)\},
\end{align}
for a fixed $u_1$ satisfying $u_1\in S(m)$ and $||\nabla u_1||_{L^2}<\frac12 C_1$ ($C_1$ is defined in Lemma \ref{AB}), $u_2(x)=e^{\frac32\theta}u_1(e^{\theta}x)$, $\theta\in\mathbb{R}$ such that $E(u_2)<0$. Define a set of paths
\begin{align}\label{path2}
\Gamma(m):=\{g\in C([0, 1], S(m)):~ g(0)=u_1, g(1)=u_2\}
\end{align}
and a min-max value
\begin{align}\label{min-max}
\gamma(m):=\inf_{g\in\Gamma(m)}\max_{t\in[0, 1]}E(g(t)).
\end{align}
Clearly, $\Gamma(m)\neq\emptyset$, as $g(t)=(1+te^{\theta}-t)^{\frac32}u_1(x+t(e^{\theta}-1)x)\in\Gamma(m)$. By Corollary \ref{AB1}, for any $g\in\Gamma(m)$ and $||\nabla u_2||_{L^2}>C_1$, thus, there exists $t_0\in(0, 1)$ such that $||\nabla g(t_0)||_{L^2}=\frac12C_1$,
$$\max_{t\in[0, 1]}E(g(t))\geq E(g(t_0))\geq E_{*}>0.$$
Namely, $E(u)$ has a mountain pass geometry on $S(m)$.

Defining that
$$\tilde{\gamma}(m):=\inf_{\tilde{g}\in\tilde{\Gamma}(m)}\max_{t\in[0, 1]}\tilde{E}(\tilde{g}(t)),$$
we have
$$\tilde{\gamma}(m)=\gamma(m).$$
Indeed, by the definition of $\tilde{\gamma}(m)$ and $\gamma(m)$, this identity follows immediately from
the fact that the maps
$$\sigma: \Gamma(m)\rightarrow\tilde{\Gamma}(m), ~g\rightarrow\sigma(g):=(g, 0)$$
and
$$\tau: \tilde{\Gamma}(m)\rightarrow\Gamma(m), ~\tilde{g}\rightarrow\tau(\tilde{g}):=\kappa\circ\tilde{g}$$
satisfy
$$\tilde{E}(\sigma(g))=E(g) ~~\hbox{and}~~ E(\tau(\tilde{g}))=\tilde{E}(\tilde{g}).$$

\begin{lemma}\label{Pohozaev-manifold1}
For $p\in(\frac{10}3, 6)$, we have
$$\gamma(m)=\inf_{u\in V(m)}E(u).$$
\end{lemma}

\begin{proof}
Assume for contradiction that there exists $v\in V(m)$ such that $E(v)<\gamma(m)$. Defining the map $J_v: \mathbb{R}\rightarrow S(m)$ with
$J_v(\theta)=\kappa(v, \theta),$
from Lemma \ref{lemma1}, we find there exists $\theta_0$ such that $J_v(-\theta_0)\in A$ ($A$ is defined in Lemma \ref{AB}) and $E(J_v(-\theta_0))<0$. Now let $\tilde{J}_v: [0, 1]\rightarrow S(m)$ be the path defined by
$$\tilde{J}_v(\theta)=\kappa(v, (2\theta-1)\theta_0).$$
Clearly, $\tilde{J}_v(0)=J_v(-\theta_0)$ and $\tilde{J}_v(1)=J_v(\theta)$. Moreover by Lemma \ref{unique}
$$\gamma(m)\leq\max_{\theta\in[0, 1]}E(\tilde{J}_v(\theta))=E(\tilde{J}_v(\frac12))=E(v),$$
and thus
$$\gamma(m)\leq\inf_{u\in V(m)}E(u).$$
On the other hand, thanks to Lemma \ref{lemma1} any path in $\Gamma(m)$ passes through the set $V(m)$ and hence
$$\max_{t\in[0, 1]}E(g(t))\geq\inf_{u\in V(m)}E(u).$$
\end{proof}

\renewcommand{\theequation}
{\thesection.\arabic{equation}}
\setcounter{equation}{0}
\section{Compactness of the Palais-Smale sequence} \noindent

In this section, we will prove the strong convergence of the Palais-Smale sequence. First, we show that there exists a Palais-Smale sequence for $E$ restricted to $S(m)$ at the level $\gamma(m)$, which is bounded in $H^1(\mathbb{R}^3)$.

For $y\in\mathbb{R}$, we define $|y|_\mathbb{R}=|y|$. Let $H=H^1(\mathbb{R}^3)\times\mathbb{R}$ be equipped with the scalar product $\langle\cdot, \cdot\rangle_H=\langle\cdot, \cdot\rangle_{H^1(\mathbb{R}^3)}+\langle\cdot, \cdot\rangle_{\mathbb{R}}$ and corresponding norm $||\cdot||^2_H=||\cdot||^2_{H^1}+|\cdot|^2_\mathbb{R}$. In addition, we denote by $H^{*}$ the dual space of $H$. Then, we give a useful result, which was proved by Ekeland's variational principle (see in \cite{Jeanjean1997}).

\begin{lemma}[\cite{Jeanjean1997}]\label{useful-result}
Let $\varepsilon>0 $, Suppose that $\tilde{g}_0\in\tilde{\Gamma}(m)$ satisfies
$$\max_{t\in[0, 1]}\tilde{E}(\tilde{g}_0(t))\leq\tilde{\gamma}(m)+\varepsilon.$$
Then there exists a pair of $(u_0, \theta_0)\in S(m)\times\mathbb{R}$ such that
\begin{itemize}
  \item[(1)] $\tilde{E}(u_0, \theta_0)\in[\tilde{\gamma}(m)-\varepsilon, \tilde{\gamma}(m)+\varepsilon]$;
  \item[(2)] $\min\limits_{t\in[0, 1]}||(u_0, \theta_0)-\tilde{g}_0(t)||_{H}\leq\sqrt{\varepsilon}$;
  \item[(3)] $\left\|\tilde{E}'|_{S(m)\times\mathbb{R}}(u_0, \theta_0)\right\|_{H^{*}}\leq 2\sqrt{\varepsilon}$, i.e.,
      $$\left|\langle\tilde{E}'(u_0, \theta_0), z\rangle_{H^{*}\times H}\right|\leq 2\sqrt{\varepsilon}||z||_{H},$$
      holds for all $z\in\tilde{T}_{(u_0, \theta_0)}:=\left\{(z_1, z_2)\in H: \langle u_0, z_1\rangle_{L^2}=0\right\}$.
\end{itemize}
\end{lemma}

\begin{lemma}\label{useful-result2}
Let $p\in(\frac{10}3, 6)$. Then there exists $m_0>0$ and a sequence $\{v_n\}_{n=1}^{\infty}\subset S(m)$ such that for any $m\in(0, m_0)$, when $n\rightarrow\infty$, we have
\begin{equation}\label{properties}
\left\{
\begin{aligned}
&E(v_n)\rightarrow\gamma(m),\\
&\left\|E'|_{S(m)}(v_n)\right\|_{H^{-1}}\rightarrow 0,\\
&P(v_n)\rightarrow 0.
\end{aligned}
\right.
\end{equation}
Here $H^{-1}$ denotes the dual space of $H^1$.
\end{lemma}

\begin{proof}
From the definition of $\gamma(m)$, we know that for each $n\in\mathbb{N}^+$, there exists a sequence $\{g_n\}_{n=1}^{\infty}\subset\Gamma(m)$ such that
$$\max_{t\in[0,1]}E(g_n(t))\leq\gamma(m)+\frac1n.$$
Since $\tilde{\gamma}(m)=\gamma(m)$, then for each $n\in\mathbb{N}^+$, $\tilde{g}_n:=(g_n, 0)\in\tilde{\Gamma}(m)$ satisfies
$$\max_{t\in[0,1]}\tilde{E}(\tilde{g}_n(t))\leq\tilde{\gamma}(m)+\frac1n.$$
Thus applying Lemma \ref{useful-result}, we obtain a sequence $\{(u_n, \theta_n)\}_{n=1}^{\infty}\subset S(m)\times\mathbb{R}$ such that
\begin{itemize}
 \item[(1)] $\tilde{E}(u_n, \theta_n)\in[\tilde{\gamma}(m)-\frac1n, \tilde{\gamma}(m)+\frac1n]$;
 \item[(2)] $\min\limits_{t\in[0, 1]}||(u_n, \theta_n)-(g_n(t), 0)||_{H}\leq\frac1{\sqrt{n}}$;
 \item[(3)] $\left\|\tilde{E}'|_{S(m)\times\mathbb{R}}(u_n, \theta_n)\right\|_{H^{*}}\leq \frac2{\sqrt{n}}$, i.e.,
     $$\left|\langle\tilde{E}'(u_n, \theta_n), z\rangle_{H^{*}\times H}\right|\leq \frac2{\sqrt{n}}||z||_{H},$$
     holds for all $z\in\tilde{T}_{(u_n, \theta_n)}:=\left\{(z_1, z_2)\in H: \langle u_n, z_1\rangle_{L^2}=0\right\}$.
\end{itemize}
For each $n\in\mathbb{N}^+$, let $v_n=\kappa(u_n, \theta_n)$. Next, we will prove that $\{v_n\}_{n=1}^{\infty}\subset S(m)$ satisfies (\ref{properties}). Firstly, from (1) we have that $E(v_n)\rightarrow\gamma(m)$ as $n\rightarrow\infty$, since $E(v_n)=E(\kappa(u_n, \theta_n))=\tilde{E}(u_n, \theta_n)$. Secondly, let $\partial_{\theta_n}\tilde{E}(u_n, \theta_n)=\langle\tilde{E}'(u_n, \theta_n), (0, 1)\rangle_{H^*\times H}$. By simple calculation, we have that
\begin{align*}
P(v_n)=&||\nabla v_n||^2_{L^2}+\frac14\int\phi_{v_n}|v_n|^2dx-\frac1{4a}\iint e^{-\frac{|x-y|}a}|v_n(x)|^2|v_n(y)|^2dxdy-\frac{3(p-2)}{2p}||v_n||^p_{L^p}\\
=&e^{2\theta_n}||\nabla u_n||^2_{L^2}+\frac{e^{\theta_n}}4\iint\frac{1-e^{-\frac{|x-y|}{ae^{\theta_n}}}}{|x-y|}|u_n(x)|^2|u_n(y)|^2dxdy\\
&-\frac1{4a}\iint e^{-\frac{|x-y|}{ae^{\theta_n}}}|u_n(x)|^2|u_n(y)|^2dxdy-\frac{3(p-2)}{2p}e^{\frac{3(p-2)}2\theta_n}||u_n||^p_{L^p}\\
=&\langle\tilde{E}'(u_n, \theta_n), (0, 1)\rangle_{H^*\times H},
\end{align*}
and $(0, 1)\in\tilde{T}_{(u_n, \theta_n)}$. Thus (3) yields that $P(v_n)\rightarrow0$ as $n\rightarrow\infty$. Finally, we prove that
$$E'|_{S(m)}(v_n)\rightarrow0 \qquad\hbox{as}~~ n\rightarrow\infty.$$
We claim that for $n\in\mathbb{N}^+$ sufficiently large,
\begin{align}\label{E-estimate}
|\langle E'(v_n), \eta\rangle_{H^{-1}\times H^1}|\leq\frac{2\sqrt{2}}{\sqrt{n}}||\eta||_{H^1},\quad \forall\eta\in T_{v_n},
\end{align}
where $T_{v_n}=\{\eta\in H^1(\mathbb{R}^3): \langle v_n, \eta\rangle_{L^2}=0\}$. To this end, we note that, for each $\eta\in T_{v_n}$, setting $\tilde{\eta}=\kappa(\eta, -\theta_n)$, one has by direct calculations that
$$\langle E'(v_n), \eta\rangle_{H^{-1}\times H^1}=\langle\tilde{E}'(u_n, \theta_n), (\tilde{\eta}, 0)\rangle_{H^*\times H}.$$
Now, we will show that $(\tilde{\eta}, 0)\in\tilde{T}_{(u_n, \theta_n)}$. Indeed
\begin{align*}
(\tilde{\eta}, 0)\in\tilde{T}_{(u_n, \theta_n)}&\Leftrightarrow \langle u_n, \tilde{\eta}\rangle_{L^2}=0\\
&\Leftrightarrow\int u_n(x)e^{-\frac32\theta_n}\eta(e^{-\theta_n}x)dx=0\\
&\Leftrightarrow\int e^{\frac32\theta_n}u_n(e^{\theta_n}x)\eta(x)dx=0\\
&\Leftrightarrow\langle v_n, \eta\rangle_{L^2}=0\\
&\Leftrightarrow\eta\in T_{v_n}.
\end{align*}
From (2), we have
$$|\theta_n|=|\theta_n-0|\leq\min_{t\in[0, 1]}||(u_n, \theta_n)-(g_n(t), 0)||_{H}\leq\frac{1}{\sqrt{n}}.$$
It follows that, for $n$ large enough,
\begin{align*}
||(\tilde{\eta}, 0)||^2_{H}&=||\tilde{\eta}||^2_{H^1}\\
&=\int|\tilde{\eta}(x)|^2dx+\int|\nabla\tilde{\eta}(x)|^2dx\\
&=\int|\eta(x)|^2dx+e^{-2\theta_n}\int|\nabla\eta(x)|^2dx\\
&\leq2||\eta||^2_{H^1},
\end{align*}
if $e^{-2\theta_n}\leq2$. Thus, by (3) we have
$$|\langle E'(v_n), \eta\rangle_{H^{-1}\times H^1}|=|\langle\tilde{E}'(u_n, \theta_n), (\tilde{\eta}, 0)\rangle_{H^*\times H}|\leq\frac2{\sqrt{n}}||(\tilde{\eta}, 0)||_{H}\leq\frac{2\sqrt{2}}{\sqrt{n}}||\eta||_{H^1},\quad \forall\eta\in T_{v_n}.$$
Consequently,
$$\left\|E'|_{S(m)}(v_n)\right\|_{H^{-1}}=\sup_{\eta\in T_{v_n}, ||\eta||_{H^1}\leq 1}|\langle E'(v_n), \eta\rangle_{H^{-1}\times H^1}|\leq\frac{2\sqrt{2}}{\sqrt{n}}\rightarrow 0\qquad\hbox{as} ~~n\rightarrow\infty.$$
\end{proof}

Analogously to Lemma 2.3 and Lemma 2.4 in \cite{Jeanjean1997}, one can conclude that $\{v_n\}_{n=1}^{\infty}$ is
a bounded sequence. For reader's convenience, we sketch the proof in the following.

\begin{lemma}\label{bounded-s}
The Palais-Smale sequence $\{v_n\}_{n=1}^{\infty}$ obtained in Lemma \ref{useful-result2} is bounded in $H^1(\mathbb{R}^3)$.
\end{lemma}

\begin{proof}
Since $\partial_{\theta_n}\tilde{E}(u_n, \theta_n)=\langle\tilde{E}'(u_n, \theta_n), (0, 1)\rangle_{H^*\times H}$, from Lemma \ref{useful-result2}, we see that
$$\partial_{\theta_n}\tilde{E}(u_n, \theta_n)\rightarrow0 \qquad\hbox{as}~~ n\rightarrow\infty,$$
that is
\begin{align*}
\partial_{\theta_n}\tilde{E}(u_n, \theta_n)=&||\nabla \kappa(u_n, \theta_n)||^2_{L^2}+\frac14\int\phi_{\kappa(u_n, \theta_n)}|\kappa(u_n, \theta_n)|^2dx-\frac{3(p-2)}{2p}||\kappa(u_n, \theta_n)||^p_{L^p}\\
&-\frac1{4a}\iint e^{-\frac{|x-y|}a}|\kappa(u_n, \theta_n)(x)|^2|\kappa(u_n, \theta_n)(y)|^2 dxdy
\rightarrow0 \qquad\hbox{as}~~ n\rightarrow\infty.
\end{align*}
Denoting $v_n=\kappa(u_n, \theta_n)$,
arguing as in (\cite[Lemma 2.3]{Jeanjean1997}), we can get that
$$\tilde{E}(u_n, \theta_n)=\frac12||\nabla v_n||^2_{L^2}+\frac14\int\phi_{v_n}|v_n|^2dx-\frac1p||v_n||^p_{L^p}$$
is bounded. Then, there exists $C>0$ independent of $n$ such that
$$|3\tilde{E}(u_n, \theta_n)+\partial_{\theta_n}\tilde{E}(u_n, \theta_n)|\leq C.$$
Since
\begin{align*}
3\tilde{E}(u_n, \theta_n)+\partial_{\theta_n}\tilde{E}(u_n, \theta_n)=&\frac52||\nabla v_n||^2_{L^2}+\int\phi_{v_n}|v_n|^2dx\\
&-\frac1{4a}\iint e^{-\frac{|x-y|}a}|v_n(x)|^2|v_n(y)|^2 dxdy-\frac32||v_n||^p_{L^p}\\
\leq&\frac52||\nabla v_n||^2_{L^2}+\int\phi_{v_n}|v_n|^2dx-\frac32||v_n||^p_{L^p},
\end{align*}
we know
\begin{align}\label{bounded1}
\frac52||\nabla v_n||^2_{L^2}+\int\phi_{v_n}|v_n|^2dx-\frac32||v_n||^p_{L^p}\geq-C.
\end{align}
Moreover, from the boundedness of $\tilde{E}(u_n, \theta_n)$, it follows that
\begin{align}\label{bounded2}
\frac12||\nabla v_n||^2_{L^2}+\frac14\int\phi_{v_n}|v_n|^2dx\leq C+\frac1p||v_n||^p_{L^p}.
\end{align}
Combining (\ref{bounded1}) with (\ref{bounded2}), we deduce that
$$\left(\frac5p-\frac32\right)||v_n||^p_{L^p}-\frac14\int\phi_{v_n}|v_n|^2dx\geq -C.$$
Due to $\frac5p-\frac32<0$, we can conclude that $||v_n||^p_{L^p}$ and $\int\phi_{v_n}|v_n|^2dx$ are bounded, hence $||\nabla v_n||^2_{L^2}$ is also bounded. Note that $\{v_n\}_{n=1}^{\infty}\subset S(m)$, thus $\{v_n\}_{n=1}^{\infty}$ is bounded in $H^1(\mathbb{R}^3)$.
\end{proof}

\begin{proposition}\label{compact-PS}
Let $\{v_n\}_{n=1}^{\infty}\subset S(m)$ be a bounded Palais-Smale sequence for $E(u)$ restricted to $S(m)$ such that $E(u)\rightarrow\gamma(m)$. Then there exists a $v\in H^1(\mathbb{R}^3)$, a sequence $\{\omega_n\}_{n=1}^{\infty}\subset\mathbb{R}$ and an $\omega_m\in\mathbb{R}$ such that, up to a subsequence, as $n\rightarrow\infty$
\begin{itemize}
 \item[(1)] $v_n\rightharpoonup v$ ~~\hbox{weakly in}~~ $H^1(\mathbb{R}^3)$;
 \item[(2)] $\omega_n\rightarrow\omega_m ~~\hbox{in}~~ \mathbb{R}$;
 \item[(3)] $E'(v_n)+\omega_n v_n\rightarrow0 ~~\hbox{in}~~ H^{-1}(\mathbb{R}^3)$;
 \item[(4)] $E'(v_n)+\omega_m v_n\rightarrow0 ~~\hbox{in}~~ H^{-1}(\mathbb{R}^3)$;
 \item[(5)] $E'(v)+\omega_m v=0 ~~\hbox{in}~~ H^{-1}(\mathbb{R}^3)$.
\end{itemize}
\end{proposition}

\begin{proof}
The proof of Proposition \ref{compact-PS} is standard, and we refer to Proposition 4.1 in \cite{Bellazzini2013} for a proof in a similar context.
\end{proof}

\begin{lemma}\label{non-vanishing}
Let $p\in(\frac{10}3, 6)$ and $\{v_n\}_{n=1}^{\infty}\subset S(m)$ be a bounded sequence such that
$$P(v_n)\rightarrow0 \quad\hbox{and}\quad E(v_n)\rightarrow\gamma(m) ~~\hbox{with}~~ \gamma(m)>0,$$
then, up to a subsequence and up to translation $v_n\rightharpoonup\tilde{v}\neq0$, as $n\rightarrow\infty$.
\end{lemma}

\begin{proof}
Assume for contradiction that $\{v_n\}_{n=1}^{\infty}\subset S(m)$ is vanishing. Thus, $||v_n||^p_{L^p}\rightarrow0$ (see \cite[Lemma I.1]{Lions1984-2}). Since $P(v_n)\rightarrow0$, we get
$$||\nabla v_n||^2_{L^2}+\frac14\iint\left(\frac{1-e^{-\frac{|x-y|}a}}{|x-y|}-\frac1ae^{-\frac{|x-y|}a}\right)
|v_n(x)|^2|v_n(y)|^2dxdy\rightarrow0, \qquad\hbox{as}~~n\rightarrow\infty.$$
It follows from $\frac{1-e^{-\frac{|x-y|}a}}{|x-y|}-\frac1a e^{-\frac{|x-y|}a}\geq0$ that
$$||\nabla v_n||^2_{L^2}\rightarrow0 \quad \hbox{and} \quad \iint\left(\frac{1-e^{-\frac{|x-y|}a}}{|x-y|}-\frac1ae^{-\frac{|x-y|}a}\right)
|v_n(x)|^2|v_n(y)|^2 dxdy\rightarrow0, \qquad\hbox{as}~~n\rightarrow\infty.$$
Now from (\ref{EP}) and Lemma \ref{HLS}, we immediately deduce that $E(v_n)\rightarrow0$ and this contradicts the assumption that $E(v_n)\rightarrow\gamma(m)>0$.
\end{proof}

\begin{lemma}\label{null-function}
Let $p\in(\frac{10}3, 6)$, $\omega\in\mathbb{R}$. If $v\in H^1(\mathbb{R}^3)$ is a weak solution of
\begin{equation}\label{main4}
-\Delta v+\omega v+\phi_v v=|v|^{p-2}v,
\end{equation}
then $P(v)=0$. Moreover, if $\omega\leq0$, then there exists a constant $m_0$ independent of $\omega\in\mathbb{R}$ such that the only solution of (\ref{main4}) fufilling $||v||_{L^2}\leq m_0$ is null function.
\end{lemma}

\begin{proof}
The following Pohozaev-type identity holds for $v\in H^1(\mathbb{R}^3)$, which is the weak solution of (\ref{main4}) (see in \cite{dAvenia2019}),
$$\frac12||\nabla v||^2_{L^2}+\frac{3\omega}2||v||^2_{L^2}+\frac54\int\phi_v |v|^2dx+\frac1{4a}\iint e^{-\frac{|x-y|}a}|v(x)|^2|v(y)|^2dxdy=\frac3p||v||^p_{L^p}.$$
By multiplying (\ref{main4}) by $v$ and integrating, we derive a second identity
$$||\nabla v||^2_{L^2}+\omega||v||^2_{L^2}+\int\phi_v |v|^2dx=||v||^p_{L^p}.$$
With simple calculus, we obtain the following relations
\begin{align}\label{Pohozaev}
||\nabla v||^2_{L^2}+\frac14\int\phi_v |v|^2dx-\frac1{4a}\iint
e^{-\frac{|x-y|}a}|v(x)|^2|v(y)|^2dxdy-\frac{3(p-2)}{2p}||v||^p_{L^p}=0,
\end{align}
\begin{align}\nonumber
-\omega||v||^2_{L^2}=&\frac{p-6}{3(p-2)}||\nabla v||^2_{L^2}+\int\phi_v |v|^2dx\\\label{relation}
&-\frac{p}{6(p-2)}\iint\left(\frac{1-e^{-\frac{|x-y|}a}}{|x-y|}-\frac1a e^{-\frac{|x-y|}a}\right)|v(x)|^2 |v(y)|^2 dxdy.
\end{align}
The identity (\ref{Pohozaev}) is $P(v)=0$. This identity together with (\ref{G-N-ineq}) ensures the existence of a constant $C(p)$ such that
\begin{align*}
||\nabla v||^2_{L^2}-C(p)||\nabla v||^{\frac{3(p-2)}2}_{L^2}||v||^{\frac{6-p}2}_{L^2}&\leq||\nabla v||^2_{L^2}-\frac{3(p-2)}{2p}||v||^p_{L^P}\\
&=-\frac14\iint\left(\frac{1-e^{-\frac{|x-y|}a}}{|x-y|}-\frac1a e^{-\frac{|x-y|}a}\right)|v(x)|^2 |v(y)|^2 dxdy\\
&\leq0,
\end{align*}
that is,
\begin{align}\label{relation1}
||\nabla v||^{\frac{10-3p}2}_{L^2}\leq C(p)||v||^{\frac{6-p}2}_{L^2}.
\end{align}
By the Hardy-Littlehood-Sobolev inequality and the Gagliardo-Nirenberg inequality, we have
 \begin{align}\label{inq}
\int\phi_v |v|^2dx\leq C||\nabla v||_{L^2}||v||^3_{L^2}.
\end{align}
Then, from the identity (\ref{relation}), we obtain
\begin{align}\label{relation2}
-\omega||v||^2_{L^2}\leq\frac{p-6}{3(p-2)}||\nabla v||^2_{L^2}+C||\nabla v||_{L^2}||v||^3_{L^2}.
\end{align}
Note that (\ref{relation1}) shows that, for any solution $v$ of (\ref{main4}) with small $L^2$-norm, $||\nabla v||_{L^2}$ must be large. This implies that the left-hand side of (\ref{relation2}) cannot be nonnegative with $||v||_{L^2}$
sufficient small.
\end{proof}

\begin{lemma}\label{convergence}
Let $p\in(\frac{10}3, 6)$. Suppose that the bounded Palais-Smale sequence $\{v_n\}_{n=1}^{\infty}\subset S(m)$ given by Lemma \ref{useful-result2} converges weakly, up to translations, to the nonzero function $\tilde{v}$. Moreover, assume that
\begin{align}\label{assume}
\forall m_1\in(0, m), \quad \gamma(m_1)>\gamma(m),
\end{align}
then $v_n\rightarrow\tilde{v}$ in $H^1(\mathbb{R}^3)$. In particular, it follows that $\tilde{v}\in S(m)$ and $E(\tilde{v})=\gamma(m)$.
\end{lemma}

\begin{proof}
We define $B(v):=\int\phi_v |v|^2dx$, $A(v):=\iint e^{-\frac{|x-y|}a}|v(x)|^2|v(y)|^2dxdy$ and $D(v):=\frac14B(v)-\frac1p||v||^p_{L^p}$. Obviously,
$$E(v):=\frac12||\nabla v||^2_{L^2}+D(v).$$
From Lemma \ref{BL}, we know that the nonlinear term $D$ fulfills the following splitting properties of Br\'{e}zis-Lieb type
\begin{align}\label{BL1}
D(v_n-\tilde{v})+D(\tilde{v})=D(v_n)+o_n(1).
\end{align}
Now, we argue by contradiction. Assume that $m_1=||\tilde{v}||_{L^2}<m$. By Proposition \ref{compact-PS}(5) and Lemma \ref{null-function}, we have $P(\tilde{v})=0$ and thus $\tilde{v}\in V(m_1)$. Since $v_n-\tilde{v}\rightharpoonup0$ in $H^1(\mathbb{R}^3)$, we have
\begin{align}\label{BL2}
||\nabla(v_n-\tilde{v})||^2_{L^2}+||\nabla\tilde{v}||^2_{L^2}=||\nabla v_n||^2_{L^2}+o_n(1).
\end{align}
Also since $\{v_n\}_{n=1}^{\infty}\subset S(m)$ is a sequence at the level $\gamma(m)$, we get
\begin{align}\label{BL3}
E(v_n)=\frac12||\nabla v_n||^2_{L^2}+D(v_n)=\gamma(m)+o_n(1).
\end{align}
Combining (\ref{BL1})-(\ref{BL3}), we deduce that
\begin{align}\label{BL4}
\frac12||\nabla(v_n-\tilde{v})||^2_{L^2}+\frac12||\nabla\tilde{v}||^2_{L^2}+D(v_n-\tilde{v})+D(\tilde{v})=\gamma(m)+o_n(1).
\end{align}
Thus, using that $\tilde{v}\in V(m_1)$ and Lemma \ref{Pohozaev-manifold1} we get from (\ref{BL4}) that
\begin{align}\label{BL5}
E(v_n-\tilde{v})+\gamma(m_1)\leq\gamma(m)+o_n(1).
\end{align}
On the other hand,
\begin{align}\nonumber
&P(v_n)-P(\tilde{v})\\ \nonumber
=&||\nabla(v_n-\tilde{v})||^2_{L^2}+\frac14B(v_n-\tilde{v})-\frac1{4a}(A(v_n)-A(\tilde{v}))-\frac{3(p-2)}{2p}||v_n-\tilde{v}||^p_{L^p}\\
\label{BL6}=&o_n(1)
\end{align}
and
\begin{align}\nonumber
&E(v_n-\tilde{v})-\frac2{3(p-2)}\left(P(v_n)-P(\tilde{v})\right)\\ \label{BL7}
=&\frac{3p-10}{6(p-2)}||\nabla(v_n-\tilde{v})||^2_{L^2}+\frac{3p-8}{12(p-2)}B(v_n-\tilde{v})+\frac1{6a(p-2)}\left(A(v_n)-A(\tilde{v})\right).
\end{align}
From (\ref{EP}) and (\ref{BL7}), we deduce that $E(v_n-\tilde{v})\geq o_n(1)~(n\rightarrow\infty)$. But from (\ref{BL5}) we obtain a contradiction with (\ref{assume}). This contradiction proves that $||\tilde{v}||_{L^2}=m$ and $E(\tilde{v})\geq\gamma(m)$. Now still by (\ref{BL5}), we get $E(v_n-\tilde{v})\leq o_n(1)~(n\rightarrow\infty)$ and thanks to (\ref{BL6}) and (\ref{BL7}) $||\nabla(v_n-\tilde{v})||_{L^2}=o_n(1)~(n\rightarrow\infty)$. To prove that $v_n\rightarrow\tilde{v}$ in $H^1(\mathbb{R}^3)$, it remains only to prove that $v_n\rightarrow\tilde{v}$ in $L^2(\mathbb{R}^3)$. By Proposition \ref{compact-PS}, it can be deduced that
\begin{align*}
&\int|\nabla(v_n-\tilde{v})|^2dx+\int(\omega_n v_n-\omega\tilde{v})(v_n-\tilde{v})+\int(\phi_{v_n}v_n-\phi_{\tilde{v}}\tilde{v})(v_n-\tilde{v})dx\\
=&\int(|v_n|^{p-2}v_n-|\tilde{v}|^{p-2}\tilde{v})(v_n-\tilde{v})dx+o_n(1).
\end{align*}
Thus, as $n\rightarrow\infty$ in preceding equality, we know
$$0=\lim_{n\rightarrow\infty}\int(\omega_n v_n-\omega\tilde{v})(v_n-\tilde{v})dx=\lim_{n\rightarrow\infty}\omega\int(v_n-\tilde{v})^2dx.$$
The proof is completed.
\end{proof}

Before we show that the solution of Eq.(\ref{main5}) is positive, let us introduce the following action functional and corresponding Nehari manifold
$$I_{\omega}(u):=\frac12\int|\nabla u|^2dx+\frac{\omega}2\int|u|^2dx+\frac14\int\phi_u |u|^2dx-\frac1{p}\int|u|^pdx,$$
and
$$\mathcal{N}_{\omega}:=\left\{u\in S(m):~ u\neq0, \langle I'_{\omega}(u), u\rangle=0\right\}.$$
\begin{lemma}\label{positive-solution0}
Assume $(u, \omega)\in H^1(\mathbb{R}^3)\times\mathbb{R}^+$ is a couple of weak solution to Eq.(\ref{main5}) with $p\in(\frac{10}3, 6)$, then $u>0$ for all $x\in\mathbb{R}^3$.
\end{lemma}
\begin{proof}
We firstly prove that $u$ does not change sign. For any $\theta\in\mathbb{R}$, we define the functions $h, h^+, h^-$ as follows:
\begin{align*}
h(\theta)=I_{\omega}(\kappa(u, \theta))=&\frac{e^{2\theta}}{2}\int|\nabla u|^2dx+\frac{\omega}2\int|u|^2dx-\frac{e^{\frac{3(p-2)}2\theta}}{p}\int|u|^pdx\\
&+\frac{e^{\theta}}4\iint\frac{1-e^{-\frac{|x-y|}{ae^{\theta}}}}{|x-y|}|u(x)|^2|u(y)|^2dxdy,
\end{align*}
\begin{align*}
h^+(\theta)=I_{\omega}(\kappa(u^+, \theta))=&\frac{e^{2\theta}}{2}\int|\nabla (u^+)|^2dx+\frac{\omega}2\int|u^+|^2dx-\frac{e^{\frac{3(p-2)}2\theta}}{p}\int|u^+|^pdx\\
&+\frac{e^{\theta}}4\iint\frac{1-e^{-\frac{|x-y|}{ae^{\theta}}}}{|x-y|}(u^+(x))^2(u^+(y))^2dxdy,
\end{align*}
\begin{align*}
h^-(\theta)=I_{\omega}(\kappa(u^-, \theta))=&\frac{e^{2\theta}}{2}\int|\nabla (u^-)|^2dx+\frac{\omega}2\int|u^-|^2dx-\frac{e^{\frac{3(p-2)}2\theta}}{p}\int|u^-|^pdx\\
&+\frac{e^{\theta}}4\iint\frac{1-e^{-\frac{|x-y|}{ae^{\theta}}}}{|x-y|}(u^-(x))^2(u^-(y))^2dxdy,
\end{align*}
where $u^+ (u^-)$ denotes the positive (negative) part of $u$, that is, $u^+=\max\{u, 0\}$, $u^-=\max\{-u, 0\}$. It is obvious that $h(\theta)=h^+(\theta)+h^-(\theta)$. Here we claim that
$$\max\limits_{\theta\in\mathbb{R}}h(\theta)=h(0).$$
Indeed, for any $u\in H^1(\mathbb{R}^3)\setminus\{0\}$, from a direct computation, there exists a unique $\theta=\theta(u)$ such that
\begin{align*}
I_{\omega}(\kappa(u, \theta(u)))=\max\limits_{s\in\mathbb{R}}I_{\omega}(\kappa(u, s)), \quad \kappa(u, \theta(u))\in\mathcal{N}_{\omega},
\end{align*}
that is
\begin{align}\label{positive-solution1}
e^{2\theta}||\nabla u||^2_{L^2}+\omega||u||^2_{L^2}+e^{\theta}\iint\frac{1-e^{-\frac{|x-y|}{ae^{\theta}}}}{|x-y|}|u(x)|^2|u(y)|^2dxdy=e^{\frac{3(p-2)}2\theta}||u||^p_{L^p}.
\end{align}
Note that $\langle I'_{\omega}(u), u\rangle=0$, we know
\begin{align}\label{positive-solution2}
||\nabla u||^2_{L^2}+\omega||u||^2_{L^2}+\iint\frac{1-e^{-\frac{|x-y|}a}}{|x-y|}|u(x)|^2|u(y)|^2dxdy=||u||^p_{L^p}.
\end{align}
Combining (\ref{positive-solution1}) with (\ref{positive-solution2}), we have
\begin{align*}
\left(e^{\frac{(3p-10)}2\theta}-1\right)||u||^p_{L^p}=&\omega\left(e^{-2\theta}-1\right)||u||^2_{L^2}+e^{-\theta}\iint\frac{1-e^{-\frac{|x-y|}{ae^{\theta}}}}{|x-y|}|u(x)|^2|u(y)|^2dxdy\\
&-\iint\frac{1-e^{-\frac{|x-y|}a}}{|x-y|}|u(x)|^2|u(y)|^2dxdy.
\end{align*}
When $\theta<0$, the left side of the above equality less than zero while the right side more
than zero, that is impossible. Analogously, we can deduce a contradiction when $\theta>0$. Therefore, $\theta=0$, we get the claim. Hence,
$$\max\limits_{\theta\in\mathbb{R}}h(\theta)=h(0)=I_{\omega}(u)=\gamma(m).$$
Denote by $\theta_1$, $\theta_2$ the maxima point of functions $h^+$, $h^-$ respectively. We suppose $\theta_1\leq\theta_2$, then $h^-(\theta_1)\geq0$. Therefore,
\begin{align*}
\max\limits_{\theta\in\mathbb{R}}h^+(\theta)=h^+(\theta_1)\leq h^+(\theta_1)+h^-(\theta_1)=h(\theta_1)\leq\max\limits_{\theta\in\mathbb{R}}h(\theta)=\gamma(m).
\end{align*}
By the definition of $\gamma(m)$, we have
\begin{align*}
\max\limits_{\theta\in\mathbb{R}}h^+(\theta)=h^+(\theta_1)+h^-(\theta_1).
\end{align*}
It follows that $h^-(\theta_1)=0$, that is $u^-=0$. When $\theta_1>\theta_2$, we can easily get that $u^+=0$ by using the same argument as above. Thus, we may assume that $u\geq0$ regarding to the change of sign.

Applying the standard arguments as in \cite{Dibenedetto1983, Li1989}, we can conclude that $u\in L^{\infty}(\mathbb{R}^3)$ and $u\in C^{1, \alpha_0}_{loc}(\mathbb{R}^3)$, where $0<\alpha_0<1$. Furthermore, from Harnack's inequality in \cite{Trudinger1967}, we have $u>0$ for all $x\in\mathbb{R}^3$.
\end{proof}

\begin{lemma}\label{exponential-decay}
Let $p\in(\frac{10}3, 6)$ and $\omega>0$. If $u\in H^1(\mathbb{R}^3)$ is a solution to Eq.(\ref{main5}), then $u\in C^2(\mathbb{R}^3\setminus\{0\})$. Moreover, there exists constants $C_1, C_2>0$ and $R>0$ such that
\begin{align}\label{decay-est0}
|u(x)|\leq C_1|x|^{-\frac34}e^{-C_2\sqrt{|x|}}, \quad \forall~ |x|\geq R.
\end{align}
\end{lemma}
\begin{proof}
It uses arguments from \cite{Bellazzini2013, Cazenave2003}. We divide the proof into two steps.\\
\textbf{Step 1.} Regularity and vanishing.

Let $\zeta$ be a smooth function satisfying
\begin{equation} \nonumber
\zeta(x)=\left\{
\begin{aligned}
&0, &|x|\leq1,\\
&1, &|x|\geq2.
\end{aligned}
\right.
\end{equation}
Given $R>0$, we denote $\zeta_R(x)=\zeta(\frac{x}{R})$. A direct computation gives
\begin{align*}
-\Delta(\zeta_R u)+\omega\zeta_R u=\left(-\Delta\zeta_R-\zeta_R\phi_u+\zeta_R|u|^{p-2}\right)u-2\nabla\zeta_R\cdot\nabla u.
\end{align*}
Since $\zeta_R$ is a smooth bounded function supported away from the origin and $\nabla\zeta_R$ is a smooth compactly supported function, it follows from the standard argument (see \cite[Theorem 8.1.1]{Cazenave2003}) that $\zeta_R u\in W^{3, q}(\mathbb{R}^3)$ for all $q\in [2, \infty)$. In particular, by Sobolev's embedding, $\zeta_R u\in C^{2, \nu}(\mathbb{R}^3)$ for all $0<\nu<1$. Since $\zeta_R u=u$ on $\mathbb{R}^3\setminus B(0, R)$ for any $R>0$, it follows that $u\in C^2(\mathbb{R}^3\setminus\{0\})$ and $|\partial^{\beta}u(x)|\rightarrow 0$ as $|x|\rightarrow\infty$ for all $|\beta|\leq 2$.\\
\textbf{Step 2.} Exponential decay estimate.

First, we show that there exists a constant $C_0>0$ such that
\begin{align}\label{decay-est1}
\phi_u\geq\frac{C_0}{|x|}, \quad \forall~|x|\geq1.
\end{align}
From \cite[Appendix A.1]{dAvenia2019}, we have $\phi_u\in C^{2, \lambda}$, $\forall \lambda\in(0, \frac12]$. In partical, $C_0=\min\limits_{\partial B_1}\phi_u(x)>0$, where $B_R=\{x\in\mathbb{R}^3:~ |x|\leq1\}$. Indeed, if $\phi_u(x_0)=0$ at some $x_0\in\mathbb{R}^3$ with $|x_0|=1$, then $u(x)=0$ a.e. in $\mathbb{R}^3$, which is a contradiction.

Now, for any $R_0\geq1$, let $v_1=\phi_u-\frac{C_0}{|x|}$. Then,
\begin{equation} \nonumber
\left\{
\begin{aligned}
&-\Delta v_1=4\pi|u|^2-a^2\Delta^2\phi_u\geq0,~~&\hbox{in}~~B_{R_0}\setminus B_1,\\
&v_1\geq0,~~&\hbox{on}~~\partial B_1,\\
&v_1\geq-\frac{C_0}{R_0},~~&\hbox{on}~~\partial B_{R_0},
\end{aligned}
\right.
\end{equation}
where we have used the fact that $\Delta\mathcal{K}=-\frac{e^{-\frac{|x|}a}}{a^2|x|}$ and $\Delta^2\mathcal{K}=\frac1{a^2}\left(\Delta\mathcal{K}+\hbox{div}\frac{x}{|x|^3}\right)$ (see in \cite{dAvenia2019}). By the maximum principle, we have
$$v_1\geq-\frac{C_0}{R_0},~~\hbox{in}~~B_{R_0}\setminus B_1.$$
Letting $R_0\rightarrow+\infty$, it follows that $v_1\geq0$ in $\mathbb{R}^3\setminus B_1$ and thus (\ref{decay-est1}) holds.

Next, we denote by $u^+ (u^-)$ the positive (negative) part of $u$. By Kato's inequality, we get that $\Delta u^+\geq\chi_{[u\geq0]}\Delta u$ (see in \cite{Brezis2004}). Therefore,
\begin{align}\label{decay-est2}
-\Delta u^+ +\omega u^++\phi_u u^+\leq|u^+|^{p-2}u^+,\quad\hbox{in}~\mathbb{R}^3.
\end{align}

Now, let us show that there exist constants $\tilde{C}>0$ and $R_1>0$ such that
\begin{align}\label{decay-est3}
u^+(x)\leq\tilde{C}\phi_u (x), \quad\forall~|x|\geq R_1.
\end{align}
To demonstrate this, we consider $v_2=u^+-\phi_u-\frac{C}{|x|}$, for a constant $C>0$. Then (\ref{decay-est2}) and $\omega>0$ imply that
\begin{align*}
-\Delta v_2&=-\Delta u^++\Delta\phi_u\\
&\leq|u^+|^{p-2}u^+-4\pi|u|^2+a^2\Delta^2\phi_u\\
&\leq|u^+|^{p-2}u^+-4\pi|u|^2, \quad\hbox{for}~~|x|\geq1.
\end{align*}
Since $u(x)\rightarrow0$ as $|x|\rightarrow\infty$ and $p>\frac{10}3$, then $|u^+|^{p-2}u^+-4\pi|u|^2\leq0$ holds in $|x|\geq R_1$ for some $R_1$ large enough. Thus, for any $R\geq R_1$ and taking $C>0$ large enough, we have
\begin{equation} \nonumber
\left\{
\begin{aligned}
&-\Delta v_2\leq0,~~&\hbox{in}~~B_{R}\setminus B_{R_1},\\
&v_2\leq0,~~&\hbox{on}~~\partial B_{R_1},\\
&v_2\leq\max\limits_{\partial B_R}u^+-\frac{C}{R},~~&\hbox{on}~~\partial B_R.
\end{aligned}
\right.
\end{equation}
Then, form the maximum principle, we obtain $v_2\leq\max\limits_{\partial B_R}u^+-\frac{C}{R}$ in $B_R\setminus B_{R_1}$. Letting $R\rightarrow+\infty$, we conclude that $v_2\leq0$ in $\mathbb{R}^3\setminus B_{R_1}$. This, together with (\ref{decay-est1}), implies that (\ref{decay-est3}) holds.

By (\ref{decay-est2}), we have for any $\sigma>0$ and since $\omega>0$,
\begin{align}\nonumber
-\Delta u^++\frac{\sigma}{|x|}u^+&\leq\frac{\sigma}{|x|}u^+-\omega u^+-\phi_u u^++|u^+|^{p-2}u^+\\
\label{decay-est4}
&\leq\left(\frac{\sigma}{|x|}-\phi_u+|u^+|^{p-2}\right)u^+.
\end{align}
Using (\ref{decay-est1}) and (\ref{decay-est3}), for any $|x|\geq R_1>1$, by choosing $\sigma\in(0, C_0)$, we have
\begin{align*}
\frac{\sigma}{|x|}-\phi_u+|u^+|^{p-2}&\leq\frac{\sigma}{C_0}\phi_u-\phi_u+|u^+|^{p-2}\\
&\leq-\frac{\left(1-\frac{\sigma}{C_0}\right)}{\tilde{C}}u^++|u^+|^{p-2}\\
&=\left(-\frac{\left(1-\frac{\sigma}{C_0}\right)}{\tilde{C}}+|u^+|^{p-3}\right)u^+,
\end{align*}
where $\frac{\left(1-\frac{\sigma}{C_0}\right)}{\tilde{C}}>0$. Since $u(x)\rightarrow0$ as $|x|\rightarrow\infty$ and $p>\frac{10}3$, for $R_1>1$ large enough, we get that $-\frac{\left(1-\frac{\sigma}{C_0}\right)}{\tilde{C}}+|u^+|^{p-3}\leq0$ for $|x|\geq R_1$. Hence, it follows from (\ref{decay-est4}) that $-\Delta u^++\frac{\sigma}{|x|}u^+\leq0$ in $\mathbb{R}^3\setminus B_{R_1}$. If we denote by $\bar{C}_1=\max\limits_{\partial B_{R_1}}u^+$, applying the maximum principle, we have
\begin{align*}
u^+\leq\bar{C}_1w,\quad\hbox{in}~ \mathbb{R}^3\setminus B_{R_1},
\end{align*}
where $w$ is the radial solution of
\begin{equation} \nonumber
\left\{
\begin{aligned}
&-\Delta w+\frac{\sigma}{|x|}=0,~~&\hbox{if}~~|x|>R_1,\\
&w(x)=1,~~&\hbox{if}~~|x|=R_1,\\
&w(x)\rightarrow0,~~&\hbox{if}~~|x|\rightarrow\infty.
\end{aligned}
\right.
\end{equation}
Then, there exist constants $C'>0$ and $R'>0$ such that $w$ satisfies (see in \cite{Ambrosetti2006})
\begin{align}\label{decay-est5}
w(x)\leq\frac{C'}{|x|^{\frac34}}e^{-2\sqrt{\sigma|x|}}, \quad \forall~ |x|>R'.
\end{align}

Finally, we observe that if $u$ is a solution of Eq.(\ref{main5}), then $-u$ is also a solution. Thus, since $u^-=(-u)^+$, following the same arguments, we obtain that there exists a constant $\bar{C}_2>0$ such that
\begin{align*}
u^-\leq\bar{C}_2w,\quad\hbox{in}~ \mathbb{R}^3\setminus B_{R_1}.
\end{align*}
Therefore, $|u|=u^++u^-\leq(\bar{C}_1+\bar{C}_2)w$, in $\mathbb{R}^3\setminus B_{R_1}$ for $R_1$ large enough. At this point, we see from (\ref{decay-est5}) that $u$ satisfies the exponential decay (\ref{decay-est0}) and thus completes the proof.
\end{proof}

Admitting for the moment that $\gamma(m)$ is nonincreasing (we shall prove it in the next section), we can now complete the proof of Theorem \ref{critical-point}.

\begin{proof}[Proof of Theorem \ref{critical-point}]
By Lemmas \ref{bounded-s} and \ref{non-vanishing}, there exists a bounded Palais-Smale
sequence $\{u_n\}_{n=1}^{\infty}\subset S(m)$ such that, up to translation, $u_n\rightharpoonup u_m\neq0$. Thus, by Proposition \ref{compact-PS}, there exists a $\omega_m\in\mathbb{R}$ such that $(u_m, \omega_m)\in H^1(\mathbb{R}^3)\setminus\{0\}\times\mathbb{R}$ solves Eq.(\ref{main5}). Now, by Lemma \ref{null-function}, there exists a $m_0>0$ such that $\omega_m>0$ if $m\in(0, m_0)$. Also we know from Proposition \ref{function-properties} (i) that (\ref{assume}) holds. At this point, the proof follows from Lemma \ref{convergence}. Finally, according to Lemmas \ref{positive-solution0} and \ref{exponential-decay}, we know that $u_m$ is positive and has exponential decay at infinity.
\end{proof}

\renewcommand{\theequation}
{\thesection.\arabic{equation}}
\setcounter{equation}{0}
\section{The behaviour of the function $m\mapsto\gamma(m)$} \noindent

In this section, we will solve the remaining question in proving the existence, that is to show that the hypotheses on $\gamma(m)$ in Lemma \ref{convergence} is true. Therefore, we study the behaviour of the function $\gamma(m)$ and summarize its properties in the following.

\begin{proposition}\label{function-properties}
Let $p\in(\frac{10}3, 6)$ and any $m>0$. Then the function $\gamma(m)$ satisfies the following properties:
\begin{itemize}
 \item[(i)] $\gamma(m)$ is continuous;
 \item[(ii)] $\gamma(m)$ is nonincreasing;
 \item[(iii)] there exists $m_0>0$ such that in $(0, m_0)$, $\gamma(m)$ is strictly decreasing;
 \item[(iv)] $\lim\limits_{m\rightarrow0^+}\gamma(m)=+\infty$.
\end{itemize}
\end{proposition}

\begin{remark}\label{remark1}
In fact, in this current study we show that $\gamma(m)$ is nonincreasing in $m>0$ and satisfies a property similar to Lemma \ref{strictly-decreasing}. And it also deduce that the Lagrange multiplier, which concerned with the mass concentration behavior, is positive.
\end{remark}

Let us denote
\begin{align}\label{gamma1}
\gamma_1(m)=\inf_{u\in S(m)}\max_{\theta\in\mathbb{R}}E(\kappa(u, \theta))
\end{align}
and
\begin{align}\label{gamma2}
\gamma_2(m)=\inf_{u\in V(m)}E(u).
\end{align}

\begin{lemma}\label{gamma3}
For $p\in(\frac{10}3, 6)$, we have
$$\gamma(m)=\gamma_1(m)=\gamma_2(m).$$
\end{lemma}

\begin{proof}
When $p\in(\frac{10}3, 6)$, from Lemma \ref{Pohozaev-manifold1}, we know that $\gamma(m)=\gamma_2(m)$. It follows from (\ref{EP}) that $\gamma_2(m)$ is well-defined. In addition, by Lemma \ref{unique}, it is clear that for any $u\in S(m)$, there exists a unique $\theta_0\in\mathbb{R}$, such that $\kappa(u, \theta_0)\in V(m)$ and $\max\limits_{\theta\in\mathbb{R}}E(\kappa(u, \theta))=E(\kappa(u, \theta_0))\geq\inf\limits_{u\in V(m)}E(u)$, thus we get $\gamma_1(m)\geq\gamma_2(m)$. Meanwhile, for any $u\in V(m)$, $E(u)=\max\limits_{\theta\in\mathbb{R}}E(\kappa(u, \theta))$ and this readily implies that $\gamma_2(m)\geq\gamma_1(m)$.  Thus, we conclude that $\gamma_1(m)=\gamma_2(m)$.
\end{proof}

\begin{lemma}\label{E-estimate0}
Let $p\in(\frac{10}3, 6)$. Then, for any $m>0$, there exists $\delta=\delta(m)>0$ small enough such that
\begin{align}\label{E-estimate01}
E(u)\geq\frac14||\nabla u||^2_{L^2}
\end{align}
for all $u\in S(m)$ satisfying $||\nabla u||_{L^2}\leq\delta$.
\end{lemma}

\begin{proof}
Using (\ref{G-N-ineq}), we get
\begin{align*}
E(u)&\geq\frac12||\nabla u||^2_{L^2}-\frac1p||u||^p_{L^p}\\
&\geq\frac12||\nabla u||^2_{L^2}-\frac{C_{GN}}p||u||^{\frac{6-p}2}_{L^2}||\nabla u||^{\frac{3(p-2)}2}_{L^2}\\
&=\left(\frac12-\frac{C_{GN}}pm^{\frac{6-p}2}||\nabla u||^{\frac{3p-10}2}_{L^2}\right)||\nabla u||^2_{L^2}.
\end{align*}
Clearly, we obtain (\ref{E-estimate01}) by setting $\delta=\delta(m):=\left(\frac p{4C_{GN}m^{\frac{6-p}2}}\right)^{\frac2{3p-10}}>0$.
\end{proof}

\begin{lemma}\label{map}
 Let $p\in(\frac{10}3, 6)$, $u\in S(m)$, and for $\theta(u)\in\mathbb{R}$, denote $\kappa(u, \theta(u))(x)=e^{\frac32\theta(u)}u(e^{\theta(u)}x)$. Then we have
\begin{itemize}
 \item[(1)] the mapping $u\mapsto\theta(u)$ is continuous in $u\in H^1(\mathbb{R}^3)\setminus\{0\}$;
 \item[(2)] $\theta(u(\cdot+y))=\theta(u)$ for any $y\in\mathbb{R}^3$.
\end{itemize}
\end{lemma}

\begin{proof}
(1) By Lemma \ref{unique} (1), we know that the mapping $u\mapsto\theta(u)$ is well-defined. Let $u\in H^1(\mathbb{R}^3)\setminus\{0\}$ and $\{u_n\}_{n=1}^{\infty}\subset H^1(\mathbb{R}^3)\setminus\{0\}$ be any sequence such that $u_n\rightarrow u$ in $H^1(\mathbb{R}^3)$. Setting $\theta_n:=\theta(u_n)$ for any $n\geq 1$, we only need to prove that up to a subsequence $\theta_n\rightarrow\theta(u)$ as $n\rightarrow\infty$.

We first show that $\{\theta_n\}_{n=1}^{\infty}$ is bounded. Indeed, if up to a subsequence $\theta_n\rightarrow+\infty$, by Fatou's lemma and the fact that $u_n\rightarrow u\neq0$ almost everywhere in $\mathbb{R}^3$, we have
$$\lim_{n\rightarrow\infty}||\kappa(u_n, \theta_n)||^p_{L^p}=\lim_{n\rightarrow\infty}e^{\frac{3(p-2)}2\theta_n}||u_n||^p_{L^p}=+\infty.$$
In view of Lemma \ref{unique} and Lemma \ref{E-estimate}, we obtain
\begin{align}\nonumber
0\leq E(\kappa(u_n, \theta_n))=&\frac{e^{2\theta_n}}2||\nabla u_n||^2_{L^2}-\frac{e^{\frac{3(p-2)}2\theta_n}}p||u_n||^p_{L^p}\\ \label{E-estimate1}
&+\frac{e^{\theta_n}}4\iint\frac{1-e^{-\frac{|x-y|}{ae^{\theta_n}}}}{|x-y|}u^2_n(x)u^2_n(y)dxdy\rightarrow-\infty,\quad\hbox{as}~n\rightarrow\infty,
\end{align}
which is a contradiction. Therefore, the sequence $\{\theta_n\}_{n=1}^{\infty}$ is bounded from above. On the other hand, from Lemma \ref{unique} (2), one has
$$E(\kappa(u_n, \theta_n))\geq E(\kappa(u_n, \theta(u))) \quad \hbox{for any} ~~n\geq1.$$
Since $\kappa(u_n, \theta(u))\rightarrow\kappa(u, \theta(u))$ in $H^1(\mathbb{R}^3)$, it follows that
$$E(\kappa(u_n, \theta(u)))=E(\kappa(u, \theta(u)))+o_n(1),$$
and thus
\begin{align}\label{E-estimate2}
\liminf_{n\rightarrow\infty}E(\kappa(u_n, \theta_n))\geq E(\kappa(u, \theta(u)))>0.
\end{align}
As $\kappa(u_n, \theta_n)\in S(m)$ for $m>0$, in view of Lemma \ref{E-estimate0} and the fact that
$$||\nabla\kappa(u_n, \theta_n)||^2_{L^2}=e^{2\theta_n}||\nabla u_n||^2_{L^2},$$
we deduce from (\ref{E-estimate2}) that $\{\theta_n\}_{n=1}^{\infty}$ is bounded also from below.

Without loss of generality, we can now assume that
$$\theta_n\rightarrow\theta_* \quad \hbox{for some} ~~\theta_*\in\mathbb{R}.$$
Recalling that $u_n\rightarrow u$ in $H^1(\mathbb{R}^3)$, one then has $\kappa(u_n, \theta_n)\rightarrow\kappa(u, \theta_*)$ in $H^1(\mathbb{R}^3)$. From Lemma \ref{unique}, $P(\kappa(u_n, \theta_n))=0$ for any $n\geq1$, it follows that
$$P(\kappa(u, \theta_*))=0.$$
Again from Lemma \ref{unique}, we see that $\theta_*=\theta(u)$ and thus (1) is proved.

(2) For any $y\in\mathbb{R}^3$, by changing variables in the integrals, we have
$$P(\kappa(u(\cdot+y), \theta(u)))=P(\kappa(u, \theta(u)))=0,$$
and thus $\theta(u(\cdot+y))=\theta(u)$ via Lemma \ref{unique} (1).
\end{proof}

\begin{lemma}\label{continuous}
We denote
$$f(a, b, c)=\max_{\theta\in\mathbb{R}}\{ae^{2\theta}+be^{\theta}-ce^{\frac{3(p-2)}2\theta}\},$$
where $p\in(\frac{10}3, 6)$ and $a>0$, $b\geq0$, $c>0$ are totally independent of $\theta$. Then the function: $(a, b, c)\mapsto f(a, b, c)$ is continuous in $\mathbb{R}^+\times\mathbb{R}^c_-\times\mathbb{R}^+$ (here, $\mathbb{R}^c_-$ denotes the nonnegative real number set).
\end{lemma}

\begin{proof}
Set $t:=e^{\theta}$, then $f(a, b, c)=\max\limits_{t>0}\{at^2+bt-ct^{\frac{3(p-2)}2}\}$. The rest of the argument is as in the proof of Lemma 5.2 in \cite{Bellazzini2013}.
\end{proof}

\begin{lemma}\label{bounded0}
Let $p\in(\frac{10}3, 6)$. Then
\begin{itemize}
 \item[(1)] $\inf\limits_{u\in V(m)}E(u)>0$;
 \item[(2)]  $\{u_n\}_{n=1}^{\infty}$ is bounded in $H^1(\mathbb{R}^3)$ for any sequence $\{u_n\}_{n=1}^{\infty}\subset V(m)$ and $\sup\limits_{n\geq1}E(u_n)<+\infty.$
\end{itemize}
\end{lemma}

\begin{proof}
(1) For any $u\in V(m)$, by Lemma \ref{unique}, we have
$$E(u)=E(\kappa(u, 0))\geq E(\kappa(u, \theta)) \quad \hbox{for all}~~ \theta\in\mathbb{R}.$$
Let $\delta>0$ be the number given by Lemma \ref{E-estimate0} and $
\theta:=\ln\left(\frac{\delta}{||\nabla u||_{L^2}}\right)$. Since $||\nabla\kappa(u, \theta)||_{L^2}=\delta$, from Lemma \ref{E-estimate0}, we deduce that
$$E(u)\geq E(\kappa(u, \theta))\geq\frac14||\nabla\kappa(u, \theta)||^2_{L^2}=\frac14\delta^2$$
and thus item (1) holds.

(2) By contradiction, we assume that there exists $\{u_n\}_{n=1}^{\infty}\subset V(m)$ such that $||u_n||_{H^1(\mathbb{R}^3)}\rightarrow\infty (n\rightarrow\infty)$ but $\sup\limits_{n\geq1}E(u_n)\leq C$ for some $C\in(0, +\infty)$. For any $n\geq1$, set
$$\theta_n:=\ln(||\nabla u_n||_{L^2}) \quad \hbox{and} \quad v_n:=\kappa(u_n, -\theta_n).$$
Clearly, $\theta_n\rightarrow+\infty (n\rightarrow\infty)$, $\{v_n\}_{n=1}^{\infty}\subset S(m)$ and $||\nabla v_n||_{L^2}=1$ for any $n\geq1$. Let
$$\rho:=\limsup_{n\rightarrow\infty}\left(\sup_{y\in\mathbb{R}^3}\int_{B(y, 1)}|v_n|^2dx\right).$$
To derive a contradiction, we distinguish the two cases: non-vanishing and vanishing.

$\bullet$ Non-vanishing (i.e., $\rho>0$). Up to a subsequence, there exists $\{y_n\}_{n=1}^{\infty}\in\mathbb{R}^3$ and $w\in H^1(\mathbb{R}^3)\setminus\{0\}$ such that
$$w_n:=v_n(\cdot+y_n)\rightharpoonup w \quad\hbox{in}~~ H^1(\mathbb{R}^3) \quad\hbox{and}\quad w_n\rightarrow w \quad\hbox{a.e. in}~~ \mathbb{R}^3.$$
Since $\theta_n\rightarrow+\infty$, it follows that
\begin{align*}
\lim_{n\rightarrow\infty}||u_n||^p_{L^p}=&\lim_{n\rightarrow\infty}||\kappa(v_n, \theta_n)||^p_{L^p}=\lim_{n\rightarrow\infty}e^{\frac{3(p-2)}2\theta_n}||v_n||^p_{L^p}\\
=&\lim_{n\rightarrow\infty}e^{\frac{3(p-2)}2\theta_n}||w_n||^p_{L^p}=+\infty.
\end{align*}
In view of item (i), we have
\begin{align*}
0\leq E(u_n)=&E(\kappa(v_n, \theta_n))\\
=&\frac{e^{2\theta_n}}2||\nabla v_n||^2_{L^2}+\frac{e^{\theta_n}}4\iint\frac{1-e^{-\frac{|x-y|}{ae^{\theta_n}}}}{|x-y|}v^2_n(x)v^2_n(y)dxdy
-\frac{e^{\frac{3(p-2)}2\theta_n}}p||v_n||^p_{L^p}\\
=&\frac{e^{2\theta_n}}2||\nabla w_n||^2_{L^2}+\frac{e^{\theta_n}}4\iint\frac{1-e^{-\frac{|x-y|}{ae^{\theta_n}}}}{|x-y|}w^2_n(x)w^2_n(y)dxdy\\
&-\frac{e^{\frac{3(p-2)}2\theta_n}}p||w_n||^p_{L^p}\rightarrow-\infty,\quad\hbox{as}~n\rightarrow\infty,
\end{align*}
which is a contradiction.

$\bullet$ Vanishing (i.e., $\rho=0$). By Lemma \ref{vanishing}, we have
$$||v_n||^p_{L^p}\rightarrow0 \quad\hbox{and}\quad \iint\frac{1-e^{-\frac{|x-y|}a}}{|x-y|}v^2_n(x)v^2_n(y)dxdy\rightarrow0, \quad\hbox{as}~n\rightarrow\infty.$$
Thus,
$$\lim_{n\rightarrow\infty}e^{-3\theta}||e^{\frac32\theta}v_n||^p_{L^p}=0 \quad\hbox{and}\quad \lim_{n\rightarrow\infty}e^{-5\theta}\iint\frac{1-e^{-\frac{|x-y|}{ae^{\theta}}}}{|x-y|}|e^{\frac32\theta}v_n(x)|^2|e^{\frac32\theta}v_n(y)|^2dxdy=0.$$
Since $P(\kappa(v_n, \theta_n))=P(u_n)=0$, by Lemma \ref{unique}, it follows that, for any $\theta\in\mathbb{R}$,
\begin{align*}
C\geq E(u_n)&=E(\kappa(v_n, \theta_n))\geq E(\kappa(v_n, \theta))\\
&=\frac12e^{2\theta}+\frac{e^{-5\theta}}4\iint\frac{1-e^{-\frac{|x-y|}{ae^{\theta}}}}{|x-y|}
|e^{\frac32\theta}v_n(x)|^2|e^{\frac32\theta}v_n(y)|^2dxdy-\frac{e^{-3\theta}}p||e^{\frac32\theta}v_n||^p_{L^p}\\
&=\frac12e^{2\theta}+o_n(1).
\end{align*}
Clearly, this leads a contradiction for $\theta>\frac{\ln(2C)}2$. Therefore, $\{u_n\}_{n=1}^{\infty}$ is bounded in $H^1(\mathbb{R}^3)$.
\end{proof}

\begin{lemma}\label{continuous1}
When $p\in(\frac{10}3, 6)$, the function $m\mapsto\gamma(m)$ is continuous at each $m>0$.
\end{lemma}

\begin{proof}
It is equivalent to prove that for a given $m>0$ and any positive sequence $\{m_k\}_{k=1}^{\infty}$ such that $m_k\rightarrow m$ as $k\rightarrow\infty$, one has $\lim\limits_{k\rightarrow\infty}\gamma(m_k)=\gamma(m)$. We first show that
\begin{align}\label{gamma-estimate}
\limsup_{k\rightarrow\infty}\gamma(m_k)\leq\gamma(m).
\end{align}
For any $u\in V(m)$, we define
$$u_k:=\frac{m_k}m u\in S(m), \quad k\in\mathbb{N}^+.$$
Since $u_k\rightarrow u$ in $H^1(\mathbb{R}^3)$, by Lemma \ref{map} (1), we have $\lim\limits_{k\rightarrow\infty}\theta(u_k)=\theta(u)=0$, and thus
$$\kappa(u_k, \theta(u_k))\rightarrow\kappa(u, \theta(u))=u \quad\hbox{in}~~ H^1(\mathbb{R}^3)~~\hbox{as}~~ k\rightarrow\infty.$$
As a consequence,
$$\limsup_{k\rightarrow\infty}\gamma(m_k)\leq\limsup_{k\rightarrow\infty}E(\kappa(u_k, \theta(u_k)))=E(u).$$
Noting that $u\in V(m)$ is arbitrary, we deduce that (\ref{gamma-estimate}) holds.

To complete the proof, it remains to show that
\begin{align}\label{gamma-estimate2}
\liminf_{k\rightarrow\infty}\gamma(m_k)\geq\gamma(m).
\end{align}
For each $k\in\mathbb{N}^+$, there exists $v_k\in V(m_k)$ such that
\begin{align}\label{gamma-estimate3}
E(v_k)\leq\gamma(m_k)+\frac1k.
\end{align}
Setting
$$t_k:=\left(\frac{m}{m_k}\right)^{\frac23} \quad \hbox{and} \quad \tilde{v}_k:=v_k\left(\frac{\cdot}{t_k}\right)\in S(m),$$
by Lemma \ref{unique} (2) and (\ref{gamma-estimate3}), we have
\begin{align*}
\gamma(m)\leq E(\kappa(\tilde{v}_k, \theta(\tilde{v}_k)))&\leq E(\kappa(v_k, \theta(\tilde{v}_k)))+\left|E(\kappa(\tilde{v}_k, \theta(\tilde{v}_k)))-E(\kappa(v_k, \theta(\tilde{v}_k)))\right|\\
&\leq E(v_k)+\left|E(\kappa(\tilde{v}_k, \theta(\tilde{v}_k)))-E(\kappa(v_k, \theta(\tilde{v}_k)))\right|\\
&\leq\gamma(m_k)+\frac1k+\left|E(\kappa(\tilde{v}_k, \theta(\tilde{v}_k)))-E(\kappa(v_k, \theta(\tilde{v}_k)))\right|\\
&=:\gamma(m_k)+\frac1k+L(k).
\end{align*}
It is clear that (\ref{gamma-estimate2}) holds if
\begin{align}\label{Lk}
\lim_{k\rightarrow\infty}L(k)=0.
\end{align}
Noting that,
\begin{align*}
L(k)=&\left|\frac12(t_k-1)\int|\nabla\kappa(v_k, \theta(\tilde{v}_k))|^2dx-\frac1p(t^3_k-1)\int|\kappa(v_k, \theta(\tilde{v}_k))|^pdx\right.\\
&\left.+\frac14(t^5_k-1)\iint\frac{1-e^{-\frac{|x-y|}{at^{-1}_k}}}{|x-y|}|\kappa(v_k, \theta(\tilde{v}_k))|^2(x)|\kappa(v_k, \theta(\tilde{v}_k))|^2(y)dxdy\right|\\
\leq&\frac12|t_k-1|\cdot ||\nabla\kappa(v_k, \theta(\tilde{v}_k))||^2_{L^2}
+\frac14|t^5_k-1|\cdot ||\kappa(v_k, \theta(\tilde{v}_k))||^4_{L^{\frac{12}5}}
+\frac1p|t^3_k-1|\cdot ||\kappa(v_k, \theta(\tilde{v}_k))||^p_{L^p}\\
=&:\frac12|t_k-1| \cdot L_1(k)+\frac14|t^5_k-1| \cdot L_2(k)+\frac1p|t^3_k-1| \cdot L_3(k).
\end{align*}
Since $t_k\rightarrow1$, the proof of (\ref{Lk}) and thus of (\ref{gamma-estimate2}) is reduced to showing that
\begin{align}\label{L-estimate}
\limsup_{k\rightarrow\infty}L_j(k)<+\infty, \quad j=1, 2, 3.
\end{align}
To justify (\ref{L-estimate}), we prove below three claims in turn.

\textbf{Claim 1.} The sequence $\{v_k\}_{k=1}^{\infty}$ is bounded in $H^1(\mathbb{R}^3)$.

Indeed, by (\ref{gamma-estimate}) and (\ref{gamma-estimate3}),
$$\limsup_{k\rightarrow\infty}E(v_k)\leq\gamma(m).$$
Since $v_k\in V(m_k)$ and $m_k\rightarrow m (k\rightarrow\infty)$, we deduce Lemma \ref{bounded0} that Claim 1 holds.

\textbf{Claim 2.} The sequence $\{\tilde{v}_k\}_{k=1}^{\infty}$ is bounded in $H^1(\mathbb{R}^3)$, and there exists $\{y_k\}_{k=1}^{\infty}\subset\mathbb{R}^3$ and $v\in H^1(\mathbb{R}^3)$ such that up to a subsequence $\tilde{v}_k(\cdot+y_k)\rightarrow v\neq0$ almost everywhere in $\mathbb{R}^3$.

Indeed, since $t_k\rightarrow1 (k\rightarrow\infty)$, it follows from Claim 1 that $\{\tilde{v}_k\}_{k=1}^{\infty}$ is bounded in $H^1(\mathbb{R}^3)$. Set
$$\rho:=\limsup_{k\rightarrow\infty}\left(\sup_{y\in\mathbb{R}^3}\int_{B(y, 1)}|\tilde{v}_k|^2dx\right).$$
We now only need to rule out the case $\rho=0$. If $\rho=0$, then $\tilde{v}_k\rightarrow0 (k\rightarrow\infty)$ in $L^r(\mathbb{R}^3)$ ($2<r<6$) by Lemma \ref{vanishing}. As a consequence,
$$\int|v_k|^rdx=\int|\tilde{v}_k(t_k\cdot)|^rdx=t^{-3}_k\int|\tilde{v}_k|^rdx\rightarrow0,\quad\hbox{as}~k\rightarrow\infty,$$
and so we have
$$||v_k||^p_{L^p}\rightarrow0 \quad\hbox{and}\quad \int\phi_{v_k}v^2_kdx\rightarrow0,\quad\hbox{as}~k\rightarrow\infty.$$
Combining $P(v_k)=0$ and the fact $\frac{1-e^{-\frac{|x-y|}a}}{|x-y|}-\frac1ae^{-\frac{|x-y|}a}\geq0$, we have
$$\iint e^{-\frac{|x-y|}a}v^2_k(x)v^2_k(y)dxdy\rightarrow0,\quad\hbox{as}~k\rightarrow\infty.$$
Therefore,
$$||\nabla v_k||^2_{L^2}\rightarrow0, \quad\hbox{as}~k\rightarrow\infty.$$
In view of Lemma \ref{E-estimate0}, we immediately deduce that
$$0<\frac14||\nabla v_k||^2_{L^2}\leq E(v_k)\rightarrow0, \quad\hbox{for $k$ large enough,}$$
which is a contradiction. The proof of Claim 2 is completed.

\textbf{Claim 3.} $\limsup\limits_{k\rightarrow\infty}\theta(\tilde{v}_k)<+\infty$.

Indeed, if Claim 3 does not hold, then up to a subsequence
\begin{align}\label{subsequence}
\theta(\tilde{v}_k)\rightarrow+\infty \quad\hbox{as}~~k\rightarrow\infty.
\end{align}
On the one hand, by Claim 2, we see that up to a subsequence
\begin{align}\label{subsequence1}
\tilde{v}_k(\cdot+y_k)\rightarrow v\neq0 \quad \hbox{a.e. in}~~ \mathbb{R}^3.
\end{align}
On the other hand, Lemma \ref{map} (2) and (\ref{subsequence}) imply
\begin{align}\label{subsequence2}
\theta(\tilde{v}_k(\cdot+y_k))=\theta(\tilde{v}_k)\rightarrow+\infty,\quad\hbox{as}~k\rightarrow\infty,
\end{align}
and Lemma \ref{unique} (2) gives us that
\begin{align}\label{subsequence3}
E\left(\kappa(\tilde{v}_k(\cdot+y_k), \theta(\tilde{v}_k(\cdot+y_k)))\right)\geq0.
\end{align}
Now, using (\ref{subsequence1}), (\ref{subsequence2}) and (\ref{subsequence3}), we clearly obtain a contradiction in the same way as the derivation of (\ref{E-estimate1}). The proof of Claim 3 is completed.

Now, by Claims 1 and 3, we have
$$\limsup_{k\rightarrow\infty}||\kappa(v_k, \theta(\tilde{v}_k))||_{H^1}<+\infty.$$
Since $H^1(\mathbb{R}^3)\hookrightarrow L^q(\mathbb{R}^3)$ with $2\leq q\leq6$, it is clear that (\ref{L-estimate}) holds and the lemma is proved.
\end{proof}

\begin{lemma}\label{nonincreasing}
When $p\in(\frac{10}3, 6)$, the function $m\mapsto\gamma(m)$ is nonincreasing on $(0, \infty)$.
\end{lemma}

\begin{proof}
We only need to show that for any $0<m_1<m_2$ and any arbitrary $\varepsilon>0$ one has
\begin{align}\label{nonincreasing1}
\gamma(m_2)<\gamma(m_1)+\varepsilon.
\end{align}
By definition of $\gamma_2(m_1)$, there exists $u\in V(m_1)$ such that
$$E(u)\leq\gamma_2(m_1)+\frac{\varepsilon}2.$$
Thus, by Lemma \ref{gamma3}, we have
\begin{align}\label{nonincreasing2}
E(u)\leq\gamma(m_1)+\frac{\varepsilon}2,
\end{align}
and alao
\begin{align}\label{nonincreasing3}
E(u)=\max_{\theta\in\mathbb{R}}E(\kappa(u, \theta)).
\end{align}
We truncate $u$ into a function with compact support $u_{\delta}$ as follows. Let $\chi\in C^{\infty}_0$ be radial and such that
\begin{equation}
\chi(x)=\left\{
\begin{aligned}
&1, &|x|\leq1,\\
&\in[0, 1], &1<|x|<2,\\
&0,  &|x|\geq2.
\end{aligned}
\right.
\end{equation}
For any small $\delta>0$, we define $u_{\delta}(x)=u(x)\cdot\chi(\delta x)\in H^1(\mathbb{R}^3)\setminus\{0\}$. It is standard to show that $u_{\delta}\rightarrow u$ in $H^1(\mathbb{R}^3)$ as $\delta\rightarrow0^+$. Then, by continuity, we have, as $\delta\rightarrow0^+$,
$$||\nabla u_{\delta}||^2_{L^2}\rightarrow||\nabla u||^2_{L^2}, \quad \int\phi_{u_{\delta}}|u_{\delta}|^2dx\rightarrow\int\phi_u|u|^2dx \quad\hbox{and}\quad ||u_{\delta}||^p_{L^p}\rightarrow||u||^p_{L^p}.$$
At this point applying Lemma \ref{continuous}, we deduce that there exists $\delta>0$ small enough, such that
\begin{align}\nonumber
\max_{\theta\in\mathbb{R}}E(\kappa(u_{\delta}, \theta))&=\max_{\theta\in\mathbb{R}}\left\{\frac{e^{2\theta}}2||\nabla u_{\delta}||^2_{L^2}+\frac{e^{\theta}}4\iint\frac{1-e^{-\frac{|x-y|}{ae^{\theta}}}}{|x-y|}|u_{\delta}(x)|^2|u_{\delta}(y)|^2dxdy
-\frac{e^{\frac{3(p-2)}2\theta}}p||u_{\delta}||^p_{L^p}\right\}\\ \nonumber
&\leq\max_{\theta\in\mathbb{R}}\left\{\frac{e^{2\theta}}2||\nabla u||^2_{L^2}+\frac{e^{\theta}}4\iint\frac{1-e^{-\frac{|x-y|}{ae^{\theta}}}}{|x-y|}|u(x)|^2|u(y)|^2dxdy
-\frac{e^{\frac{3(p-2)}2\theta}}p||u||^p_{L^p}\right\}+\frac{\varepsilon}4\\ \nonumber
&=\max_{\theta\in\mathbb{R}}E(\kappa(u, \theta))+\frac{\varepsilon}4\\\label{nonincreasing4}
&=E(u)+\frac{\varepsilon}4.
\end{align}
Now take $v\in C^{\infty}_0(\mathbb{R}^3)$ be radial and such that supp$(v)\subset B(0, 1+\frac4{\delta})\setminus B(0, \frac4{\delta})$ and set
$$v_0=\frac{m^2_2-||u_\delta||^2_{L^2}}{||v||^2_{L^2}}v,$$
for which we have $||v_0||^2_{L^2}=m^2_2-||u_{\delta}||^2_{L^2}$.

For any $\lambda<0$, we define $w_{\lambda}=u_{\delta}+\kappa(v_0, \lambda)$. We observe that
\begin{align}\label{nonincreasing5}
\hbox{dist}\left\{\hbox{supp}(u_{\delta}), \hbox{supp}(\kappa(v_0, \lambda))\right\}\geq\frac2{\delta}\left(\frac2{e^{\lambda}}-1\right).
\end{align}
Thus, $||w_{\lambda}||^2_{L^2}=||u_{\delta}||^2_{L^2}+||\kappa(v_0, \lambda)||^2_{L^2}$ and $w_{\lambda}\in S(m)$. Also
\begin{align}\label{nonincreasing6}
||\nabla w_{\lambda}||^2_{L^2}=||\nabla u_{\delta}||^2_{L^2}+||\nabla\kappa(v_0, \lambda)||^2_{L^2} \quad\hbox{and}\quad ||w_{\lambda}||^p_{L^p}=||u_{\delta}||^p_{L^p}+||\kappa(v_0, \lambda)||^p_{L^p}.
\end{align}
We claim that, for any $\lambda<0$,
\begin{align}\label{nonincreasing7}
\left|\int\phi_{w_{\lambda}}|w_{\lambda}|^2dx-\int\phi_{u_{\delta}}|u_{\delta}|^2dx-\int\phi_{\kappa(v_0, \lambda)}|\kappa(v_0, \lambda)|^2dx\right|\leq e^{\lambda}||u_{\delta}||^2_{L^2}||\kappa(v_0, \lambda)
||^2_{L^2}.
\end{align}
Indeed, from (\ref{nonincreasing5}),
$$|u_{\delta}(x)+\kappa(v_0, \lambda)(x)|^2=|u_{\delta}(x)|^2+|\kappa(v_0, \lambda)(x)|^2, \quad |u_{\delta}(y)+\kappa(v_0, \lambda)(y)|^2=|u_{\delta}(y)|^2+|\kappa(v_0, \lambda)(y)|^2.$$
Thus,
\begin{align*}
\int\phi_{w_{\lambda}}|w_{\lambda}|^2dx=&\iint\frac{1-e^{-\frac{|x-y|}a}}{|x-y|}|u_{\delta}(x)+\kappa(v_0, \lambda)(x)|^2|u_{\delta}(y)+\kappa(v_0, \lambda)(y)|^2dxdy\\
=&\iint\frac{1-e^{-\frac{|x-y|}a}}{|x-y|}|u_{\delta}(x)|^2|u_{\delta}(y)|^2dxdy+2\iint\frac{1-e^{-\frac{|x-y|}a}}{|x-y|}|u_{\delta}(x)|^2|\kappa(v_0, \lambda)(y)|^2dxdy\\
&+\iint\frac{1-e^{-\frac{|x-y|}a}}{|x-y|}|\kappa(v_0, \lambda)(x)|^2|\kappa(v_0, \lambda)(y)|^2dxdy\\
=&\int\phi_{u_{\delta}}|u_{\delta}|^2dx+\int\phi_{\kappa(v_0, \lambda)}|\kappa(v_0, \lambda)|^2dx+2\iint\frac{1-e^{-\frac{|x-y|}a}}{|x-y|}|u_{\delta}(x)|^2|\kappa(v_0, \lambda)(y)|^2dxdy,
\end{align*}
with
\begin{align*}
&\iint\frac{1-e^{-\frac{|x-y|}a}}{|x-y|}|u_{\delta}(x)|^2|\kappa(v_0, \lambda)(y)|^2dxdy\\
&\leq\iint\frac{|u_{\delta}(x)|^2|\kappa(v_0, \lambda)(y)|^2}{|x-y|}dxdy\\
&=\int_{\hbox{supp}(u_{\delta})}\int_{\hbox{supp}(\kappa(v_0, \lambda))}\frac{|u_{\delta}(x)|^2|\kappa(v_0, \lambda)(y)|^2}{|x-y|}dxdy\\
&\leq\frac{\delta e^{\lambda}}{2(2-e^{\lambda})}\int_{\hbox{supp}(u_{\delta})}\int_{\hbox{supp}(\kappa(v_0, \lambda))}|u_{\delta}(x)|^2|\kappa(v_0, \lambda)(y)|^2dxdy\\
&\leq\frac{\delta e^{\lambda}}{2(2-e^{\lambda})}||u_{\delta}||^2_{L^2}||\kappa(v_0, \lambda)||^2_{L^2}\\
&\leq\frac{e^{\lambda}}2||u_{\delta}||^2_{L^2}||\kappa(v_0, \lambda)||^2_{L^2},
\end{align*}
and then (\ref{nonincreasing7}) holds. Now from (\ref{nonincreasing6}) and (\ref{nonincreasing7}), we see that
$$||\nabla w_{\lambda}||^2_{L^2}\rightarrow||\nabla u_{\delta}||^2_{L^2}, \quad \int\phi_{w_{\lambda}}|w_{\lambda}|^2dx\rightarrow\int\phi_{u_{\delta}}|u_{\delta}|^2dx \quad\hbox{and}\quad ||w_{\lambda}||^p_{L^p}\rightarrow||u_{\delta}||^p_{L^p}, \quad\hbox{as}~~\lambda\rightarrow-\infty.$$
Thus, from Lemma \ref{continuous}, we have that, fixing $\lambda<0$ such that $|\lambda|$ large enough,
\begin{align}\label{nonincreasing8}
\max_{\theta\in\mathbb{R}}E(\kappa(w_{\lambda}, \theta))\leq\max_{\theta\in\mathbb{R}}E(\kappa(u_{\delta}, \theta))+\frac{\varepsilon}4.
\end{align}
Now, using Lemma \ref{gamma3}, (\ref{nonincreasing2}), (\ref{nonincreasing4}) and (\ref{nonincreasing8}), we get
\begin{align*}
\gamma(m_2)\leq\max_{\theta\in\mathbb{R}}E(\kappa(w_{\lambda}, \theta))\leq\max_{\theta\in\mathbb{R}}E(\kappa(u_{\delta}, \theta))+\frac{\varepsilon}4\leq E(u)+\frac{\varepsilon}2\leq\gamma(m_1)+\varepsilon,
\end{align*}
This concludes the proof of the lemma.
\end{proof}

\begin{lemma}\label{strictly-decreasing}
For $p\in(\frac{10}3, 6)$, if that there exists $u\in S(m)$ and $\omega\in\mathbb{R}$ such that
$$-\Delta u+\omega u+\phi_u u=|u|^{p-2}u$$
and $E(u)=\gamma(m)$, then $\gamma(m)>\gamma(m')$ for any $m'>m$ close enough to $m$ if $\omega>0$ and for each $m'<m$ near enough to $m$ if $\omega<0$.
\end{lemma}

\begin{proof}
For any $s>0$ and $\theta\in\mathbb{R}$, we denote $u_{s, \theta}:=\kappa(su, \theta)\in S(ms)$. Since
\begin{align*}
\alpha(s, \theta):=E(u_{s, \theta})=\frac12s^2e^{2\theta}||\nabla u||^2_{L^2}+\frac14s^4e^{\theta}\iint\frac{1-e^{-\frac{|x-y|}{ae^{\theta}}}}{|x-y|}|u(x)|^2|u(y)|^2dxdy
-\frac1ps^pe^{\frac{3(p-2)}2\theta}||u||^p_{L^p},
\end{align*}
it is clear that
\begin{align*}
\frac{\partial}{\partial s}\alpha(s, \theta)=se^{2\theta}||\nabla u||^2_{L^2}+s^3e^{\theta}\iint\frac{1-e^{-\frac{|x-y|}{ae^{\theta}}}}{|x-y|}|u(x)|^2|u(y)|^2dxdy-s^{p-1}e^{\frac{3(p-2)}2\theta}||u||^p_{L^p}.
\end{align*}
When $\omega>0$, we have
$$\frac{\partial}{\partial s}\alpha(s, \theta)|_{(1, 0)}=-\omega||u||^2_{L^2}=-\omega m^2<0.$$
One can fix a $\delta>0$ small enough such that
$$\frac{\partial}{\partial s}\alpha(s, \theta)<0 \quad\hbox{for any}~~(s, \theta)\in(1, 1+\delta]\times[-\delta, \delta].$$
From the mean value theorem, we then obtain
\begin{align}\label{sd1}
\alpha(s, \theta)=\alpha(1, \theta)+(s-1)\cdot\frac{\partial}{\partial s}\alpha(\beta, \theta)<\alpha(1, \theta),
\end{align}
where $1<\beta<s\leq1+\delta$ and $|\theta|\leq\delta$. Note that $\theta(su)\rightarrow\theta(u)$ as $s\rightarrow1^+$ by Lemma \ref{map}(1). For any $m'>m$ close enough to $m$, we have
$$s:=\sqrt{\frac{m'}m}\in(1, 1+\delta] \quad\hbox{and}\quad \theta:=\theta(su)\in[-\delta, \delta],$$
and thus, using (\ref{sd1}) and Lemma \ref{unique} (2),
$$\gamma(m')\leq\alpha(s, \theta(su))<\alpha(1, \theta(su))=E(\kappa(u, \theta(su)))\leq E(u)=\gamma(m).$$
The case $\omega<0$ can be proved similarly.
\end{proof}

At this point, from Lemma \ref{nonincreasing} and Lemma \ref{strictly-decreasing}, we directly obtain the following lemma.

\begin{lemma}\label{strictly-decreasing1}
When $p\in(\frac{10}3, 6)$, if there exists $u\in S(m)$ and $\omega\in\mathbb{R}$ such that
$$-\Delta u+\omega u+\phi_u u=|u|^{p-2}u$$
and $E(u)=\gamma(m)$, then $\omega\geq0$. If in addition $\omega>0$, then $\gamma(m)>\gamma(m')$ for any $m'>m$.
\end{lemma}

As the end of this section, we study the limit behavior of $\gamma(m)$ when $m>0$ tends to zero.

\begin{lemma}\label{limit-behavior}
Let $p\in(\frac{10}3, 6)$, then $\gamma(m)\rightarrow+\infty$ as $m\rightarrow0^+$.
\end{lemma}

\begin{proof}
By Theoreom \ref{critical-point}, we know that for any $m>0$ sufficiently small there exists a couple of weak solution $(u_m, \omega_m)\in H^1(\mathbb{R}^3)\times\mathbb{R}^+$ to Eq.(\ref{main5}) with $||u_m||_{L^2}=m$ and $E(u_m)=\gamma(m)$. In addition, by Lemma \ref{null-function}, we have $P(u_m)=0$. Thus, $u_m\in H^1(\mathbb{R}^3)$ fulfills
\begin{equation}\label{lb1}
0=||\nabla u_m||^2_{L^2}+\frac14\iint\left(\frac{1-e^{-\frac{|x-y|}a}}{|x-y|}-\frac1ae^{-\frac{|x-y|}a}\right)|u_m(x)|^2|u_m(y)|^2dxdy
-\frac{3(p-2)}{2p}||u_m||^p_{L^p},
\end{equation}
\begin{align}\label{lb2}
\gamma(m)=E(u_m)=\frac12||\nabla u_m||^2_{L^2}+\frac14\iint\frac{1-e^{-\frac{|x-y|}a}}{|x-y|}|u_m(x)|^2|u_m(y)|^2dxdy-\frac1p||u_m||^p_{L^p}.
\end{align}
Since $\frac{1-e^{-\frac{|x-y|}a}}{|x-y|}-\frac1ae^{-\frac{|x-y|}a}\geq0$, we deduce from (\ref{lb1}) that $||\nabla u_m||^2_{L^2}\leq\frac{3(p-2)}{2p}||u_m||^p_{L^p}$ and thus it follows from Gagliardo-Nirenberg inequality that
$$||\nabla u_m||^2_{L^2}\leq C(p)||\nabla u_m||^{\frac{3(p-2)}2}_{L^2}||u_m||^{\frac{6-p}2}_{L^2},$$
that is,
\begin{align}\label{lb3}
1\leq C(p)m^{\frac{6-p}2}||\nabla u_m||^{\frac{3p-10}2}_{L^2}.
\end{align}
Since $p\in(\frac{10}3, 6)$, we obtain that
\begin{align}\label{lb4}
||\nabla u_m||^2_{L^2}\rightarrow+\infty, \quad\hbox{as}\quad m\rightarrow0^+.
\end{align}
Now, from (\ref{lb1}) and (\ref{lb2}), we deduce that
\begin{align*}
\gamma(m)=E(u_m)=&\frac{3p-10}{6(p-2)}||\nabla u_m||^2_{L^2}+\frac{3p-8}{12(p-2)}\iint\frac{1-e^{-\frac{|x-y|}a}}{|x-y|}|u_m(x)|^2|u_m(y)|^2dxdy\\
&+\frac1{6a(p-2)}\iint e^{-\frac{|x-y|}a}|u_m(x)|^2|u_m(y)|^2dxdy,
\end{align*}
and thus from (\ref{lb4}), we get immediately that $\gamma(m)\rightarrow+\infty$ as $m\rightarrow0^+$.
\end{proof}

\begin{proof}[Proof of Proposition \ref{function-properties}]
Obviously, (i), (ii), (iii) of Proposition \ref{function-properties} follow directly
from Lemmas \ref{continuous1}, \ref{nonincreasing} and \ref{strictly-decreasing}. And Lemma \ref{limit-behavior} conclude (iv).
\end{proof}

\renewcommand{\theequation}
{\thesection.\arabic{equation}}
\setcounter{equation}{0}
\section{Concentration behavior of normalized solutions} \noindent

In this section, we consider the concentration behavior of minimizers for $\mathcal{E}(m)$ or $\gamma(m)$, which are given in (\ref{min-pro1}) and (\ref{Pohozaev-manifold}). Before that, we would give accurate descriptions of the minimizers for $\widetilde{\mathcal{E}}(m)$ and $\beta(m)$ given in (\ref{min-pro2}) and (\ref{min-pro3}).

For any $u\in H^1(\mathbb{R}^3)$ and any $s>0$, in what follows we denote
\begin{align}\label{L-scaling}
u^s(x):=s^{\frac32}u(sx).
\end{align}
Then $u^s\in S(m)$ if $u\in S(m)$.

In the following lemma, we will state the exact relation between the minimizers of $\widetilde{\mathcal{E}}(m)$ and $\beta(m)$ and the solution $Q$ of Eq.(\ref{elliptic}), which was given by Ye and Luo in \cite{Ye2018}.
\begin{lemma}[\cite{Ye2018}]\label{relationship}
Let $m>0$, suppose that $v_m\in S(m)$ is either a minimizer of $\widetilde{\mathcal{E}}(m)$ when $p\in(2, \frac{10}3)$ or a minimizer of $\beta(m)$ when $p\in(\frac{10}3, 6)$. Then,
up to translations, $v_m(x)=\frac{m}{||Q||_{L^2}}s_m^{\frac32}Q(s_mx)$ and
\begin{equation}
\left\{
\begin{aligned}\nonumber
&\widetilde{\mathcal{E}}(m)=-\frac{10-3p}{6(p-2)}m^2s^2_m<0, &p\in(2, \frac{10}3),\\
&\beta(m)=\frac{3p-10}{6(p-2)}m^2s^2_m>0, &p\in(\frac{10}3, 6).
\end{aligned}
\right.
\end{equation}
Here $s_m=\left[\frac{3(p-2)}4\left(\frac{m}{||Q||_{L^2}}\right)^{p-2}\right]^{\frac2{10-3p}}$.
\end{lemma}

\begin{lemma}\label{limit-relation}
Let $p\in(2, \frac{10}3)\cup(\frac{10}3, 6)$. Then
\begin{equation}
\left\{
\begin{aligned}\nonumber
&\lim\limits_{m\rightarrow0^+}\widetilde{\mathcal{E}}(m)=0, &p\in(2, \frac{10}3),\\
&\lim\limits_{m\rightarrow+\infty}\widetilde{\mathcal{E}}(m)=-\infty, &p\in(2, \frac{10}3),\\
&\lim\limits_{m\rightarrow0^+}\beta(m)=+\infty, &p\in(\frac{10}3, 6),
\end{aligned}
\right.
\end{equation}
and
\begin{equation}
\left\{
\begin{aligned}\nonumber
&\lim\limits_{m\rightarrow0^+}\frac{s_m}{m^2}=+\infty, &p\in(2, 3)\cup(\frac{10}3, 6),\\
&\lim\limits_{m\rightarrow+\infty}\frac{s_m}{m^2}=+\infty, &p\in(3, \frac{10}3).
\end{aligned}
\right.
\end{equation}
\end{lemma}
\begin{proof}
By direct calculations, we have
$$\widetilde{\mathcal{E}}(m)=-\frac{10-3p}{6(p-2)}\left[\frac{3(p-2)}4\frac{m^{\frac{6-p}2}}{||Q||^{p-2}_{L^2}}
\right]^{\frac4{10-3p}},\quad \beta(m)=\frac{3p-10}{6(p-2)}\left[\frac4{3(p-2)}\frac{||Q||^{p-2}_{L^2}}{m^{\frac{6-p}2}}\right]^{\frac4{3p-10}}$$
and
$$\frac{s_m}{m^2}=\left[\frac{3(p-2)}{4||Q||^{p-2}_{L^2}}\right]^{\frac2{10-3p}}m^{\frac{8(p-3)}{10-3p}}.$$
Then, the lemma follows.
\end{proof}

\begin{lemma}\label{limits1}
\begin{itemize}
 \item[(1)] If $p\in(2, 3)$, then $\lim\limits_{m\rightarrow0^+}\frac{\mathcal{E}(m)}{\widetilde{\mathcal{E}}(m)}=1$.
 \item[(2)] If $p\in(3, \frac{10}3)$, then  $\lim\limits_{m\rightarrow+\infty}\frac{\mathcal{E}(m)}{\widetilde{\mathcal{E}}(m)}=1$.
 \item[(3)] If $p\in(\frac{10}3, 6)$, then $\lim\limits_{m\rightarrow0^+}\frac{\gamma(m)}{\beta(m)}=1$.
\end{itemize}
\end{lemma}
\begin{proof}
When $p\in(2, 3)\cup(3, \frac{10}3)$, we easily see that $\widetilde{\mathcal{E}}(m)\leq\mathcal{E}(m)$ for each $m>0$. Note that $\widetilde{\mathcal{E}}(m)<0$, then $\frac{\mathcal{E}(m)}{\widetilde{\mathcal{E}}(m)}\leq1$.
Since the minimizer $\frac{m}{||Q||_{L^2}}Q^{s_m}$ of $\widetilde{\mathcal{E}}(m)$ belongs to $S(m)$, then $\mathcal{E}(m)\leq E\left(\frac{m}{||Q||_{L^2}}Q^{s_m}\right)$. By Lemma \ref{relationship} and Lemma \ref{limit-relation}, we have
\begin{align*}
\frac{E\left(\frac{m}{||Q||_{L^2}}Q^{s_m}\right)}{\widetilde{\mathcal{E}}(m)}&=1+\frac{m^4s_m}{4\widetilde{\mathcal{E}}(m)||Q||^4_{L^2}}
\iint\frac{1-e^{-\frac{|x-y|}{as_m}}}{|x-y|}|Q(x)|^2 |Q(y)|^2dxdy\\
&=1-\frac{3(p-2)}{2(10-3p)||Q||^4_{L^2}}\iint\frac{1-e^{-\frac{|x-y|}{as_m}}}{|x-y|}|Q(x)|^2 |Q(y)|^2dxdy\frac{m^2}{s_m}\\
&\geq1-\frac{3(p-2)}{2(10-3p)||Q||^4_{L^2}}\iint\frac{|Q(x)|^2 |Q(y)|^2}{|x-y|}dxdy\frac{m^2}{s_m}\rightarrow1,
\end{align*}
either as $m\rightarrow0^+$ if $p\in(2, 3)$ or as $m\rightarrow+\infty$ if $p\in(3, \frac{10}3)$, which implies that $\lim\limits_{m\rightarrow0^+}\frac{\mathcal{E}(m)}{\widetilde{\mathcal{E}}(m)}\geq1$ if $p\in(2, 3)$ or $\lim\limits_{m\rightarrow+\infty}\frac{\mathcal{E}(m)}{\widetilde{\mathcal{E}}(m)}\geq1$
if $p\in(3, \frac{10}3)$. Therefore, (1) and (2) hold.

When $p\in(\frac{10}3, 6)$, for any $m>0$ small enough and any $u\in V(m)$, then $P(u)=0$. Since $\frac{1-e^{-\frac{|x-y|}a}}{|x-y|}-\frac1{a}e^{-\frac{|x-y|}a}\geq0$, $G(u)=P(u)-\frac14\iint\left(\frac{1-e^{-\frac{|x-y|}a}}{|x-y|}-\frac1{a} e^{-\frac{|x-y|}a}\right)|u(x)|^2|u(y)|^2dxdy<0$ (we are only concerned with the case where $G(u)<0$, since the result is obvious when $G(u)=0$) and there exists $s\in(0, 1)$ such that $u^s\in\widetilde{V}(m)$. Thus, $I(u^s)\geq\beta(m)$. Then, by (\ref{EP}), we get
\begin{align*}
E(u)=E(u)-\frac2{3(p-2)}P(u)&\geq\frac{3p-10}{6(p-2)}s^2||\nabla u||^2_{L^2}\\
&=I(u^s)-\frac2{3(p-2)}G(u^s)\\
&=I(u^s)\geq\beta(m),
\end{align*}
which implies that $\gamma(m)\geq\beta(m)$ since $u\in V(m)$ is arbitrary. Note that $\beta(m)>0$, so we have $\frac{\gamma(m)}{\beta(m)}\geq1$.
For each $m>0$ small enough and the corresponding minimizer $v_m=\frac{m}{||Q||_{L^2}}Q^{s_m}$ of $\beta(m)$, then $G(v_m)=0$ and there exists $\eta_m>0$ such that $v^{\eta_m}_m\in V(m)$. From (\ref{elliptic1}) and
\begin{align*}
0=P(v^{\eta_m}_m)\leq&\frac{m^2\eta^2_ms^2_m}{||Q||^2_{L^2}}||\nabla Q||^2_{L^2}+\frac{m^4\eta_ms_m}{4||Q||^4_{L^2}}\iint\frac{|Q(x)|^2|Q(y)|^2}{|x-y|}dxdy\\
&-\frac{3(p-2)}{2p}\frac{m^p\eta^{\frac{3(p-2)}2}_ms^{\frac{3(p-2)}2}_m}{||Q||^p_{L^2}}||Q||^p_{L^p},
\end{align*}
we see that
\begin{align*}
\eta^{\frac{10-3p}2}_m+\eta^{\frac{8-3p}2}_m\frac{m^2}{s_m}\frac1{4||Q||^4_{L^2}}\iint\frac{|Q(x)|^2|Q(y)|^2}{|x-y|}dxdy\geq1.
\end{align*}
Thus, it follows from Lemma \ref{limit-relation} that $\{\eta_m\}$ is bounded for $m>0$ small enough. Therefore,
\begin{align*}
\gamma(m)\leq E(v^{\eta_m}_m)&=\frac{\eta^2_m}2||\nabla v_m||^2_{L^2}+\frac{\eta_m}4\iint\frac{1-e^{-\frac{|x-y|}{a\eta_m}}}{|x-y|}|v_m(x)|^2|v_m(y)|^2dxdy-\frac{\eta_{m}^{\frac{3(p-2)}2}}{p}||v_m||^p_{L^p}\\
&\leq\max\limits_{\eta\geq0}I(v^{\eta}_m)+\frac{\eta_m}4\iint\frac{|v_m(x)|^2|v_m(y)|^2}{|x-y|}dxdy\\
&=\beta(m)+\frac{m^4\eta_m s_m}{4||Q||^4_{L^2}}\iint\frac{|Q(x)|^2|Q(y)|^2}{|x-y|}dxdy.
\end{align*}
From Lemma \ref{limit-relation}, we have
$$\frac{\gamma(m)}{\beta(m)}\leq1+\frac{m^2}{s_m}\frac{3(p-2)\eta_m}{2(3p-10)||Q||^4_{L^2}}\iint\frac{|Q(x)|^2|Q(y)|^2}{|x-y|}dxdy\rightarrow1. \quad\hbox{as}~m\rightarrow0^+,$$
Thus, $\limsup\limits_{m\rightarrow0^+}\frac{\gamma(m)}{\beta(m)}\leq1$. And hence, $\lim\limits_{m\rightarrow0^+}\frac{\gamma(m)}{\beta(m)}=1$.
\end{proof}

\begin{lemma}\label{limits2}
Suppose that $(m, u_m, \omega_m)\in\Lambda$, where $\Lambda$ is given in (\ref{solution-set}), then
\begin{itemize}
 \item[(1)]
\begin{equation}
\left\{
\begin{aligned}\nonumber
&\frac{||\nabla u_m||^2_{L^2}}{m^2s^2_m}\rightarrow1,\\
&\frac{\int\phi_{u_m}|u_m|^2dx}{m^2s^2_m}\rightarrow0,
\end{aligned}
\right.
\end{equation}
\end{itemize}
as $m\rightarrow0^+$ if $p\in(2, 3)\cup(\frac{10}3, 6)$ or $m\rightarrow+\infty$ if $p\in(3, \frac{10}3)$.
\begin{itemize}
 \item[(2)]
$$\frac{\omega_m}{s^2_m}\rightarrow\frac{6-p}{3(p-2)}$$
\end{itemize}
as $m\rightarrow0^+$ if $p\in(2, 3)\cup(\frac{10}3, 6)$ or $m\rightarrow+\infty$ if $p\in(3, \frac{10}3)$.
\end{lemma}
\begin{proof}
(1) We first consider the case where $p\in(2, 3)\cup(3, \frac{10}3)$. Since $E(u_m)=\mathcal{E}(m)$, by (\ref{G-N-ineq}) and Lemma \ref{relationship} we have
\begin{align}\nonumber
\frac{\mathcal{E}(m)}{\widetilde{\mathcal{E}}(m)}&=-\frac{3(p-2)}{10-3p}\frac{||\nabla u_m||^2_{L^2}}{m^2s^2_m}-\frac{3(p-2)}{2(10-3p)}\frac{\int\phi_{u_m}|u_m|^2dx}{m^2s^2_m}-\frac1{p}\frac{||u_m||^p_{L^p}}{\widetilde{\mathcal{E}}(m)}\\
\label{bound1}&\leq-\frac{3(p-2)}{10-3p}\frac{||\nabla u_m||^2_{L^2}}{m^2s^2_m}-\frac{m^{\frac{6-p}2}}{2||Q||^{p-2}_{L^2}}\frac{||\nabla u_m||^{\frac{3(p-2)}2}_{L^2}}{\widetilde{\mathcal{E}}(m)}\\ \nonumber
&=-\frac{3(p-2)}{10-3p}\left[\frac{||\nabla u_m||^2_{L^2}}{m^2s^2_m}-\frac4{3(p-2)}\left(\frac{||\nabla u_m||^2_{L^2}}{m^2s^2_m}\right)^{\frac{3(p-2)}4}\right]\\ \nonumber
&\leq1.
\end{align}
Then, by Lemma \ref{limits1}, we obtain that
$$\lim\limits_{m\rightarrow0^+}\frac{||\nabla u_m||^2_{L^2}}{m^2s^2_m}=1,~~~\hbox{if}~p\in(2, 3) \quad\hbox{and}\quad \lim\limits_{m\rightarrow+\infty}\frac{||\nabla u_m||^2_{L^2}}{m^2s^2_m}=1,~~~\hbox{if}~p\in(3, \frac{10}3).$$

Set
\begin{equation}
F_1:=\left\{
\begin{aligned}\nonumber
&\lim\limits_{m\rightarrow0^+}\frac{\frac1{p}||u_m||^p_{L^p}}{\widetilde{\mathcal{E}}(m)}, &p\in(2, 3),\\
&\lim\limits_{m\rightarrow+\infty}\frac{\frac1{p}||u_m||^p_{L^p}}{\widetilde{\mathcal{E}}(m)}, &p\in(3, \frac{10}3),
\end{aligned}
\right.
\end{equation}
\begin{equation}
F_2:=\left\{
\begin{aligned}\nonumber
&\lim\limits_{m\rightarrow0^+}\frac{\int\phi_{u_m}|u_m|^2dx}{m^2s^2_m}, &p\in(2, 3),\\
&\lim\limits_{m\rightarrow+\infty}\frac{\int\phi_{u_m}|u_m|^2dx}{m^2s^2_m}, &p\in(3, \frac{10}3),
\end{aligned}
\right.
\end{equation}
then it follows from (\ref{G-N-ineq}) and (\ref{bound1}) that $F_1\in\left[-\frac4{10-3p}, 0\right]$ and $F_2\in[0, +\infty)$ and
$$\frac4{10-3p}+F_1=-\frac{3(p-2)}{2(10-3p)}F_2.$$
Therefore, we have $F_2=0$ and $F_1=-\frac4{10-3p}$.

We next prove the case where $p\in(\frac{10}3, 6)$. Since $E(u_m)=\gamma(m)$ and $u_m\in V(m)$, by Lemma \ref{relationship} we get
\begin{align}\nonumber
\frac{\gamma(m)}{\beta(m)}=&\frac{E(u_m)-\frac2{3(p-2)}P(u_m)}{\beta(m)}\\ \nonumber
=&\frac{3p-10}{6(p-2)}\frac{||\nabla u_m||^2_{L^2}}{\beta(m)}+\frac{3p-8}{12(p-2)}\frac{\int\phi_{u_m}|u_m|^2dx}{\beta(m)}+\frac1{6a(p-2)}\frac{\iint e^{-\frac{|x-y|}a}|u_m(x)|^2|u_m(y)|^2dxdy}{\beta(m)}\\ \label{bound2}
\geq&\frac{3p-10}{6(p-2)}\frac{||\nabla u_m||^2_{L^2}}{\beta(m)}\\ \nonumber
=&\frac{||\nabla u_m||^2_{L^2}}{m^2s^2_m}.
\end{align}
According to $\lim\limits_{m\rightarrow0^+}\frac{\gamma(m)}{\beta(m)}=1$, we see that $\left|\frac{||\nabla u_m||^2_{L^2}}{m^2s^2_m}\right|\leq2$ for $m>0$ small enough. By (\ref{inq}) and Lemma \ref{limit-relation}, we have
$$\frac{\int\phi_{u_m}|u_m|^2dx}{m^2s^2_m}\leq C\frac{m^3||\nabla u_m||_{L^2}}{m^2s^2_m}\leq2C\frac{m^2}{s_m}\rightarrow0,\quad\hbox{as}~m\rightarrow0^+.$$
Then it follows from (\ref{bound2}) that
\begin{align}\label{upper-bound1}
\lim\limits_{m\rightarrow0^+}\frac{||\nabla u_m||^2_{L^2}}{m^2s^2_m}\leq1.
\end{align}
Since $\iint e^{-\frac{|x-y|}a}|u_m(x)|^2|u_m(y)|^2dxdy\leq C||u_m||^4_{L^2}$ and $\int\phi_{u_m}|u_m|^2dx\leq C||u_m||^4_{L^2}$, there exist constants $C_1, C_2>0$ such that
\begin{align*}
\frac{\gamma(m)}{\beta(m)}&\leq\frac{3p-10}{6(p-2)}\frac{||\nabla u_m||^2_{L^2}}{\beta(m)}+C_1\frac{||u_m||^4_{L^2}}{\beta(m)}+C_2\frac{||u_m||^4_{L^2}}{\beta(m)}\\
&=\frac{||\nabla u_m||^2_{L^2}}{m^2s^2_m}+\frac{6(p-2)C_1}{3p-10}\frac{m^2}{s^2_m}+\frac{6(p-2)C_2}{3p-10}\frac{m^2}{s^2_m}.
\end{align*}
From Lemma \ref{limit-relation}, we have
$$\frac{\iint e^{-\frac{|x-y|}a}|u_m(x)|^2|u_m(y)|^2dxdy}{m^2s^2_m}\rightarrow0,\quad\hbox{as}~m\rightarrow0^+.$$
Combining with Lemma \ref{limits1}(3), we obtain
\begin{align}\label{lower-bound1}
\lim\limits_{m\rightarrow0^+}\frac{||\nabla u_m||^2_{L^2}}{m^2s^2_m}\geq1.
\end{align}
Thus, $\lim\limits_{m\rightarrow0^+}\frac{||\nabla u_m||^2_{L^2}}{m^2s^2_m}=1$.

(2) We just prove the case where $p\in(2, 3)$. The proof of the other two cases proceeds similarly.

According to $E(u_m)=\mathcal{E}(m)$, we see that
\begin{align*}
\omega_mm^2=-\langle E'(u_m), u_m\rangle=-p\mathcal{E}(m)+\frac{p-2}2||\nabla u_m||^2_{L^2}+\frac{p-4}4\int\phi_{u_m}|u_m|^2dx.
\end{align*}
So, by Lemma \ref{relationship} and Lemma \ref{limits1}, we have
\begin{align*}
\frac{\omega_m}{s^2_m}=&-p\frac{\mathcal{E}(m)}{m^2s^2_m}+\frac{p-2}2\frac{||\nabla u_m||^2_{L^2}}{m^2s^2_m}+\frac{p-4}4\frac{\int\phi_{u_m}|u_m|^2dx}{m^2s^2_m}\\
=&\frac{10p-3p^2}{6(p-2)}\frac{\mathcal{E}(m)}{\widetilde{\mathcal{E}}(m)}+\frac{p-2}2\frac{||\nabla u_m||^2_{L^2}}{m^2s^2_m}+\frac{p-4}4\frac{\int\phi_{u_m}|u_m|^2dx}{m^2s^2_m}\\
&\rightarrow\frac{10p-3p^2}{6(p-2)}+\frac{p-2}2=\frac{6-p}{3(p-2)},\quad\hbox{as}~m\rightarrow0^+.
\end{align*}
\end{proof}

\begin{lemma}\label{omega-e}
Let $p\in(3, \frac{10}3)$ and $(m, u_m, \omega_m)\in \Lambda$, then $\omega_m>0$.
\end{lemma}
\begin{proof}
For any $(m, u_m, \omega_m)\in \Lambda$, we see that $P(u_m)=0$ and $E(u_m)=\mathcal{E}(m)<0$. Then,
\begin{align*}
\omega_m m^2=&-||\nabla u_m||^2_{L^2}-\int\phi_{u_m}|u_m|^2dx+||u_m||^p_{L^p}-\frac{2(p-2)}{10-3p}P(u_m)\\
=&-\frac{6-p}{10-3p}||\nabla u_m||^2_{L^2}-\frac{18-5p}{2(10-3p)}\int\phi_{u_m}|u_m|^2dx+\frac{2(6-p)}{p(10-3p)}||u_m||^p_{L^p}\\
&+\frac{p-2}{2a(10-3p)}\iint e^{-\frac{|x-y|}a}|u_m(x)|^2|u_m(y)|^2dxdy\\
=&-\frac{2(6-p)}{10-3p}\mathcal{E}(m)+\frac{2(p-3)}{10-3p}\int\phi_{u_m}|u_m|^2dx\\
&+\frac{p-2}{2a(10-3p)}\iint e^{-\frac{|x-y|}a}|u_m(x)|^2|u_m(y)|^2dxdy>0,
\end{align*}
since $p\in(3, \frac{10}3)$.
\end{proof}

\begin{lemma}\label{positive-soution00}
Let $p\in(2,3)\cup(3, \frac{10}3)$ and $m>0$, suppose that $u_m\in S(m)$ is a minimizer of $\mathcal{E}(m)$, then $u_m$ is positive.
\end{lemma}

\begin{proof}
As the proof follows very closely that of the analogous statement
in Lemma \ref{positive-solution0}, we prefer to omit the details here.
\end{proof}

Now, let us consider a scaling of $Q$, the positive, radially symmetric and unique solution of Eq.(\ref{elliptic}), as
\begin{align}\label{scaling-Q}
W(x):=\left(\frac{3(p-2)}4||Q||^{-\frac43}_{L^2(\mathbb{R}^3)}\right)^{\frac3{10-3p}}Q\left(\left(\frac{3(p-2)}4||Q||^{2-p}_{L^2(\mathbb{R}^3)}\right)^{\frac2{10-3p}}x\right).
\end{align}
Obviously, $W\in S(1)$ (i.e., $||W||_{L^2(\mathbb{R}^3)}=1$) and satisfies the following equation
\begin{align}\label{elliptic0}
-\Delta W+\omega_0 W=|W|^{p-2}W,\quad x\in\mathbb{R}^3,
\end{align}
where $\omega_0=\frac{6-p}4\left(\frac{3(p-2)}4||Q||^{-\frac43}_{L^2(\mathbb{R}^3)}\right)^{\frac{3(p-2)}{10-3p}}>0.$ For any $\varsigma>0$ and $u\in S(m)$, we consider the scaling that $u^{(\varsigma)}(x):=\varsigma^{\frac4{10-3p}}u(\varsigma^{\frac{2(p-2)}{10-3p}}x)$. Hence, $u^{(\varsigma)}\in S(\varsigma m)$ and
\begin{align*}
E(u^{(\varsigma)})&=\varsigma^{\frac{2(6-p)}{10-3p}}\left(\frac12||\nabla u||^2_{L^2}+\frac14\varsigma^{\alpha_1(p)}\iint\frac{1-e^{-\frac{|x-y|}{a\varsigma^{\alpha_2(p)}}}}{|x-y|}|u(x)|^2|u(y)|^2dxdy-\frac1{p}||u||^p_{L^p}\right)\\
&=:\varsigma^{\frac{2(6-p)}{10-3p}}E_{\varsigma}(u).
\end{align*}
Here $\alpha_1(p):=\frac{8(3-p)}{10-3p}$, $\alpha_2(p):=\frac{2(p-2)}{10-3p}$. It is clearly that $\mathcal{E}(m)=m^{\frac{2(6-p)}{10-3p}}J_m$, where $J_m$ is given by the minimization problem
\begin{align}\label{equivalent-1}
J_m:=\inf\limits_{u\in S(1)}E_m(u).
\end{align}
Similarly, the minimization problem (\ref{Pohozaev-manifold}) is equivalent to the following problem
\begin{align}\label{equivalent-2}
K_m:=\inf\limits_{u\in \widehat{V}_m}E_m(u),
\end{align}
where
$$\widehat{V}_m:=\{u\in S(1):~ P_{m}(u)=0\},$$
\begin{align}\nonumber
P_{m}(u):=&||\nabla u||^2_{L^2}+\frac14m^{\alpha_1(p)}\iint\frac{1-e^{-\frac{|x-y|}{am^{\alpha_2(p)}}}}{|x-y|}|u(x)|^2|u(y)|^2dxdy\\ \label{Pohozaev-m}
&-\frac1{4a}m^{\alpha_3(p)}\iint e^{-\frac{|x-y|}{am^{\alpha_2(p)}}}|u(x)|^2|u(y)|^2dxdy-\frac{3(p-2)}{2p}||u||^p_{L^p},
\end{align}
and $\alpha_3(p):=\frac{2(14-5p)}{10-3p}$. It is observed that for any $p\in(2, 3)\cup(3, \frac{10}3)\cup(\frac{10}3, 6)$ and $r\geq0$, $e^{-\frac{r}{am^{\alpha_2(p)}}}\leq 1$ holds. Moreover, if $p\in(2, 3)$, then $\alpha_1(p)>0$; if $p\in(3, \frac{10}3)$, then $\alpha_1(p)<0$; if $p\in(\frac{10}3, 6)$, then $\alpha_1(p)>0, \alpha_3(p)>0$. For $m>0$, we denote by $\mathcal{N}_{m},~ \widetilde{\mathcal{N}}_{m}$ the corresponding minimizers, respectively,
$$\mathcal{N}_{m}:=\{v\in S(1): ~E_{m}(v)=J_{m}\}, \quad \widetilde{\mathcal{N}}_{m}:=\{v\in S(1): ~E_{m}(v)=K_{m}\}.$$
Clearly, $\mathcal{N}_{m}\neq\emptyset$ and $\widetilde{\mathcal{N}}_{m}\neq\emptyset$. Define $E_{0}(u):=\frac12||\nabla u||^2_{L^2}-\frac1{p}||u||^p_{L^p}$, $\widehat{V}_0:=\{u\in S(1):~ P_{0}(u)=||\nabla u||^2_{L^2}-\frac{3(p-2)}{2p}||u||^p_{L^p}=0\}$, and
$$\mathcal{N}_{0}:=\{v\in S(1): ~E_{0}(v)=J_{0}:=\inf\limits_{u\in S(1)}E_{0}(u)\},$$
$$\widetilde{\mathcal{N}}_{0}:=\{v\in S(1): ~E_{0}(v)=K_{0}:=\inf\limits_{u\in \widehat{V}_0}E_{0}(u)\}.$$
Thus, we have $W\in\mathcal{N}_{0}$ and $W\in\widetilde{\mathcal{N}}_{0}$.

Next, we will show some properties of $\widehat{V}_m$. Note that $P_{m}(u)=\frac{\partial{E_{m}(\kappa(u, \theta))}}{\partial{\theta}}\mid_{\theta=0}$. According to Lemma \ref{unique}, it is easy to get the following lemma.
\begin{lemma}\label{unique00}
Let $p\in(\frac{10}3, 6)$, then for any $u\in S(1)$, there exists a unique $\theta^{*}_m(u)\in\mathbb{R}$ such that $\kappa(u, \theta^{*}_m)\in\widehat{V}_m$ for all $m>0$. Moreover, $E_{m}(\kappa(u, \theta))<E_{m}(\kappa(u, \theta^{*}_m(u)))$ for any $\theta\neq\theta^{*}_m(u)$.
\end{lemma}

\begin{lemma}\label{converges0}
For $p\in(\frac{10}3, 6)$, let $\{m_n\}^{\infty}_{n=1}$ be a positive sequence and $\lim\limits_{n\rightarrow+\infty}m_n=0$, and let $\theta^{*}_{m_n}=\theta^{*}_{m_n}(W)$ be such that $\kappa(W, \theta^{*}_{m_n})\in\widehat{V}_{m_n}$. Then $\lim\limits_{n\rightarrow+\infty}\theta^{*}_{m_n}=0$.
\end{lemma}

\begin{proof}
The proof is divided into the following steps.\\
\textbf{Step 1:} there exists $\delta>0$ and $d>0$ such that $K_{m_n}>d$ for $0<m_n<\delta$.

Indeed, for any $\varepsilon>0$, there $\delta>0$, for $0<m<\delta$ and any $u\in S(1)$, by applying Gagliardo-Nirenberg inequality, we have
\begin{align*}
E_{m}(u)&\geq\frac12||\nabla u||^2_{L^2}+\frac14m^{\alpha_1(p)}\iint\frac{1-e^{-\frac{|x-y|}{am^{\alpha_2(p)}}}}{|x-y|}|u(x)|^2|u(y)|^2dxdy-\frac{C_{GN}}{p}||\nabla u||^{\frac{3(p-2)}2}_{L^2}\\
&\geq\frac12||\nabla u||^2_{L^2}-\frac{\varepsilon}4-\frac{C_{GN}}{p}||\nabla u||^{\frac{3(p-2)}2}_{L^2}.
\end{align*}
Similar to Lemma \ref{gamma3}, we have
\begin{align*}
K_{m}=\inf\limits_{u\in S(1)}\max\limits_{\theta\in\mathbb{R}}E_{m}(\kappa(u, \theta)).
\end{align*}
Therefore, there exists $d>0$ such that $K_{m_n}>d$ for $0<m_n<\delta$.\\
\textbf{Step 2:} $\{\theta^{*}_{m_n}\}^{\infty}_{n=1}$ is bounded.

We assume by contradiction $\theta^{*}_{m_n}\rightarrow+\infty$ up to a subsequence, we obtain
\begin{align}\nonumber
E_{m_n}(\kappa(u, \theta^{*}_{m_n}))=&\frac{e^{2\theta^{*}_{m_n}}}2||\nabla u||^2_{L^2}-\frac{e^{\frac{3(p-2)}{2}\theta^{*}_{m_n}}}{p}||u||^p_{L^p}\\ \label{converges-estimate}
&+\frac{e^{\theta^{*}_{m_n}}}4m_{n}^{\alpha_1(p)}\iint\frac{1-e^{-\frac{|x-y|}{am_{n}^{\alpha_2(p)}e^{\theta^{*}_{m_n}}}}}{|x-y|}|u(x)|^2|u(y)|^2dxdy.
\end{align}
Since $\frac{3(p-2)}2>2$, we have $E_{m_n}(\kappa(u, \theta^{*}_{m_n}))\rightarrow-\infty$ as $n\rightarrow+\infty$, which contradicts the Step 1. Thus $\{\theta^{*}_{m_n}\}^{\infty}_{n=1}$ is bounded.\\
\textbf{Step 3:} $\lim\limits_{n\rightarrow+\infty}\theta^{*}_{m_n}=0$.

Since $\{\theta^{*}_{m_n}\}^{\infty}_{n=1}$ is bounded, we argue by contradiction and assume that there exists $\theta^{*}_0\neq0$ satisfying $\lim\limits_{n\rightarrow+\infty}\theta^{*}_{m_n}\rightarrow\theta^{*}_0$ up to a subsequence. According to $\kappa(W, \theta^{*}_{m_n})\in\widehat{V}_{m_n}$, we have
\begin{align*}
P_{m}(\kappa(W, \theta^{*}_{m_n}))=&e^{2\theta^{*}_{m_n}}||\nabla W||^2_{L^2}+\frac{e^{\theta^{*}_{m_n}}}4m_{n}^{\alpha_1(p)}
\iint\frac{1-e^{-\frac{|x-y|}{am_{n}^{\alpha_2(p)}e^{\theta^{*}_{m_n}}}}}{|x-y|}|W(x)|^2|W(y)|^2dxdy\\
&-\frac{e^{\theta^{*}_{m_n}}}{4a}m_{n}^{\alpha_3(p)}\iint e^{-\frac{|x-y|}{am_{n}^{\alpha_2(p)}e^{\theta^{*}_{m_n}}}}|W(x)|^2|W(y)|^2dxdy-\frac{3(p-2)}{2p}e^{\frac{3(p-2)}{2}\theta^{*}_{m_n}}||W||^p_{L^p}\\
=&0.
\end{align*}
Since $\alpha_1(p)>0, \alpha_3(p)>0$ and $\lim\limits_{n\rightarrow+\infty}m_n=0$, we get
\begin{align*}
\lim\limits_{n\rightarrow+\infty}E_{m_n}(\kappa(W, \theta^{*}_{m_n}))=e^{2\theta^{*}_0}||\nabla W||^2_{L^2}-\frac{3(p-2)}{2p}e^{\frac{3(p-2)}{2}\theta^{*}_0}||W||^p_{L^p}=0.
\end{align*}
Due to $W\in\widehat{V}_0$, we have $\theta^{*}_0=0$ or $\theta^{*}_0=-\infty$. If $\theta^{*}_0=-\infty$, then by (\ref{converges-estimate}) we obtain
$$E_{m_n}(\kappa(W, \theta^{*}_{m_n}))\rightarrow0 ~\hbox{as}~ n\rightarrow+\infty,$$
which is a contradiction with the Step 1. Hence, $\theta^{*}_0=0$ and the proof is completed.
\end{proof}

We now are ready to establish Theorem \ref{concentration-behavior}.

\begin{proof}[Proof of Theorem \ref{concentration-behavior}]
As mentioned in (\ref{scaling-Q}), (\ref{converges-Q00}) is equivalent to the existence of a sequence $\{\widetilde{z}_n\}^{\infty}_{n=1}\subset\mathbb{R}^3$ under the condition (a) or (b) such that
\begin{align}\label{converges-W}
m_n^{-\frac4{10-3p}}u_{m_n}\left(m_n^{-\frac{2(p-2)}{10-3p}}x+\widetilde{z}_n\right)\rightarrow W(x)~~\hbox{in}~H^1(\mathbb{R}^3), ~~\hbox{as}~n\rightarrow+\infty.
\end{align}
Here $W$ is the ground state solution of Eq.(\ref{elliptic0}). We set $v_n(x):=m_n^{-\frac4{10-3p}}u_{m_n}\left(m_n^{-\frac{2(p-2)}{10-3p}}x\right)$, then $||v_n||_{L^2}=1$, namely, $v_n\in S(1)$. Therefore, to show that (\ref{converges-W}) holds, it suffices to demonstrate the following conclusion:
\begin{itemize}
 \item[(1)] Let $p\in(2, 3)$ and $v_n\in\mathcal{N}_{m_n}$, where $m_n\rightarrow0^+$, as $n\rightarrow+\infty$.
 \item[(2)] Let $p\in(3, \frac{10}3)$ and $v_n\in\mathcal{N}_{m_n}$, where $m_n\rightarrow+\infty$, as $n\rightarrow+\infty$.
 \item[(3)] Let $p\in(\frac{10}3, 6)$ and $v_n\in\widetilde{\mathcal{N}}_{m_n}$, where $m_n\rightarrow0^+$, as $n\rightarrow+\infty$.
\end{itemize}
Then, up to a subsequence there exists $\tau_n\in\mathbb{R}^3$ such that
$$v_n(x+\tau_n)\rightarrow W~~\hbox{in}~H^1(\mathbb{R}^3), ~~\hbox{as}~n\rightarrow+\infty.$$

We first prove $J_{m_n}\rightarrow J_{0}$ as $n\rightarrow+\infty$. Notice that $J_{0}\leq J_{m_n}$ due to the positivity of $m_n$. Hence it is sufficient to prove $\limsup\limits_{n\rightarrow+\infty}J_{m_n}\leq J_{0}$. This fact follows from
\begin{align*}
J_{m_n}&\leq E_{m_n}(W)=E_{0}(W)+\frac{m_{n}^{\alpha_1(p)}}{4}\iint\frac{1-e^{-\frac{|x-y|}{am_{n}^{\alpha_2(p)}}}}{|x-y|}
|W(x)|^2|W(y)|^2dxdy=J_{0}+o_n(1).
\end{align*}
Then, we prove $v_n$ converge to $W$ up to subsequence and translation. By the previous step we deduce that $\{v_n\}^{\infty}_{n=1}$ is a minimizing sequence for $J_{0}$. As a consequence of the results proved in \cite{Cazenave1982, Kwong1989} we deduce that $v_n$ converge strongly (up to translation) to $W(x)$ in $H^1(\mathbb{R}^3)$ and hence (1) is proved. Similarly, we can show that (2) holds.

We denote $\kappa(W, \theta^{*}_{m_n})\in\widehat{V}_{m_n}$, thus by Lemma \ref{converges0} we obtain $\theta^{*}_{m_n}\rightarrow0$ as $n\rightarrow+\infty$. Then
\begin{align*}
K_{m_n}&\leq E_{m_n}(\kappa(W, \theta^{*}_{m_n}))=E_{0}(\kappa(W, \theta^{*}_{m_n}))+\frac{e^{\theta^{*}_{m_n}}}4\alpha_1(p)
\iint\frac{1-e^{-\frac{|x-y|}{am_{n}^{\alpha_2(p)}e^{\theta^{*}_{m_n}}}}}{|x-y|}|W(x)|^2|W(y)|^2dxdy\\
&= K_{0}+o_n(1).
\end{align*}
Therefore, $\limsup\limits_{n\rightarrow+\infty}K_{m_n}\leq K_{0}$. Let $v_n\in\widetilde{\mathcal{N}}_{m_n}$, similar to Lemma \ref{unique1}, we have
\begin{align*}
K_{m_n}=E_{m_n}(v_n)\geq\frac2{3(p-2)}P_{m_n}(v_n)+\frac{3p-10}{6(p-2)}||\nabla v_n||^2_{L^2}=\frac{3p-10}{6(p-2)}||\nabla v_n||^2_{L^2}.
\end{align*}
It is important to note that $K_{m_n}$ is bounded, thus $||\nabla v_n||^2_{L^2}$ is also bounded.

We claim that there exists $d_1>0$ such that $||v_n||^p_{L^p}\geq d_1$. If up to a subsequence $||v_n||^p_{L^p}\rightarrow0$, by the fact that $P_{m_n}(v_n)=0$, we get
\begin{align*}
||\nabla v_n||&^2_{L^2}+\frac{m_{n}^{\alpha_1(p)}}4\iint\frac{1-e^{-\frac{|x-y|}{am_n^{\alpha_2(p)}}}}{|x-y|}|v_n(x)|^2|v_n(y)|^2dxdy\\
&=\frac{m_n^{\alpha_3(p)}}{4a}\iint e^{-\frac{|x-y|}{am_n^{\alpha_2(p)}}}|v_n(x)|^2|v_n(y)|^2dxdy+\frac{3(p-2)}{2p}||v_n||^p_{L^p}\\
&\leq C(a)m_n^{\alpha_3(p)}||v_n||^4_{L^2}+\frac{3(p-2)}{2p}||v_n||^p_{L^p}\rightarrow0, ~~\hbox{as}~n\rightarrow+\infty,
\end{align*}
where $C(a)$ is a positive constant. Therefore, $\liminf\limits_{n\rightarrow+\infty}E_{m_n}(v_n)\leq0$, which is impossible because of $v_n\in\widetilde{\mathcal{N}}_{m_n}$.
Assume $\kappa(v_n, \theta^{*}_0(v_n))\in\widetilde{V}_0$, that is, $||\nabla v_n||^2_{L^2}=\frac{3(p-2)}{2p}e^{\frac{3p-10}{2}\theta^{*}_0(v_n)}||v_n||^p_{L^p}$. Combining the fact that $||v_n||^p_{L^p}$ has a positive lower bound and that $||\nabla v_n||^2_{L^2}$ has a positive upper bound, $\{\theta^{*}_0(v_n)\}^{\infty}_{n=1}$ is bounded. Hence,
\begin{align*}
K_{m_n}=&E_{m_n}(v_n)\geq E_{m_n}(\kappa(v_n, \theta^{*}_0(v_n)))\\
&+\frac{e^{\theta^{*}_0(v_n)}}4m_n^{\alpha_1(p)}\iint\frac{1-e^{-\frac{|x-y|}{am_n^{\alpha_2(p)}e^{\theta^{*}_0(v_n)}}}}{|x-y|}|v_n(x)|^2|v_n(y)|^2dxdy\\
=&K_{0}+o_n(1).
\end{align*}
Then, $\liminf\limits_{n\rightarrow+\infty}K_{m_n}\geq K_{0}$. So $K_{m_n}\rightarrow K_{0}$ as $n\rightarrow+\infty$. Standard methods(see in \cite{Cazenave1982, Kwong1989}) complete the proof of (3).

Finally, we conclude from Lemma \ref{limits2}(2) and the definition of $s_{m_n}$ that
\begin{align*}
\lim\limits_{n\rightarrow+\infty}
\left(\frac{||Q||_{L^2}}{m_n}\right)^{\frac{4(p-2)}{10-3p}}\omega_{m_n}=\frac{6-p}{3(p-2)}\left(\frac{3(p-2)}4\right)^{\frac4{10-3p}}.
\end{align*}
Then, the proof of Theorem \ref{concentration-behavior} is completed.
\end{proof}

\renewcommand{\theequation}
{\thesection.\arabic{equation}}
\setcounter{equation}{0}
\section{The radial symmetry and uniqueness} \noindent

This section is devoted to discussing the radial symmetry and uniqueness of the normalized ground states for Eq.(\ref{main5}). At the beginning, let us fix some notations. We define the following linear operator $L: H^1(\mathbb{R}^3)\rightarrow H^{-1}(\mathbb{R}^3)$ by the linearization of Eq.(\ref{elliptic0}) at $W$
$$L:=-\Delta+\omega_0-(p-1)W^{p-2}.$$
It is well-known that
$$T_{W}:=\ker{L}=\hbox{span}\left\{\frac{\partial{W}}{\partial{x_1}}, \frac{\partial{W}}{\partial{x_2}}, \frac{\partial{W}}{\partial{x_3}}\right\}.$$
For any $\tau\in\mathbb{R}^3$, we set $W_{\tau}(x):=W(x+\tau)$, and denote the orthogonal space of $T_{W_\tau}$ in $H^1(\mathbb{R}^3)$ with respect to the $L^2(\mathbb{R}^3)$ scalar product by $T^{\bot}_{W_\tau}$. If $\mathcal{M}$ is a vector space, then we use $\pi_\mathcal{M}$ to denote the orthogonal projection on the vector space $\mathcal{M}$ with respect to the $L^2(\mathbb{R}^3)$ scalar product. Finally, $H^1_r(\mathbb{R}^3)$ denotes the functions in $H^1(\mathbb{R}^3)$ that are radially symmetric.

\begin{lemma}\label{converges-Q0}
\begin{itemize}
 \item[(1)] Let $p\in(2, 3)$, for every $\varepsilon_1>0$, there exists $m(\varepsilon_1)>0$ such that
     $\sup\limits_{\omega\in\mathcal{A}_m}|\omega-\omega_0|<\varepsilon_1, ~\forall 0<m<m(\varepsilon_1)$.
 \item[(2)] Let $p\in(3, \frac{10}3)$, for every $\varepsilon_2>0$, there exists $m(\varepsilon_2)>0$ such that $\sup\limits_{\omega\in\mathcal{A}_m}|\omega-\omega_0|<\varepsilon_2, ~\forall m>m(\varepsilon_2)$.
 \item[(3)] Let $p\in(\frac{10}3, 6)$, for every $\varepsilon_3>0$, there exists $m(\varepsilon_3)>0$ such that $\sup\limits_{\omega\in\mathcal{B}_m}|\omega-\omega_0|<\varepsilon_3, ~\forall 0<m<m(\varepsilon_3)$.
\end{itemize}
Here,
$$\mathcal{A}_m:=\left\{-\Delta v+\omega v+m^{\alpha_1(p)}\int\frac{1-e^{-\frac{|x-y|}{am^{\alpha_2(p)}}}}{|x-y|}|v(y)|^2dyv=|v|^{p-2}v, ~v\in\mathcal{N}_{m}\right\},$$
$$\mathcal{B}_m:=\left\{-\Delta v+\omega v+m^{\alpha_1(p)}\int\frac{1-e^{-\frac{|x-y|}{am^{\alpha_2(p)}}}}{|x-y|}|v(y)|^2dyv=|v|^{p-2}v, ~v\in\widetilde{\mathcal{N}}_{m}\right\}.$$
\end{lemma}

\begin{proof}
We just prove the item (1), the other two items can be similarly proved. By looking at the equation satisfied by $v\in\mathcal{N}_{m}$, we deduce
\begin{align}\label{converges-o}
\omega=\frac{||v||^p_{L^p}-||\nabla v||^2_{L^2}-m^{\alpha_1(p)}\iint\frac{1-e^{-\frac{|x-y|}{am^{\alpha_2(p)}}}}{|x-y|}|u(x)|^2|u(y)|^2dxdy}{||v||^2_{L^2}}.
\end{align}
The proof can be concluded since by Theorem \ref{concentration-behavior} we get that as $m\rightarrow0^+$ (\ref{converges-o}) converges to
\begin{align*}
\frac{||W||^p_{L^p}-||\nabla W||^2_{L^2}}{||W||^2_{L^2}}=\omega_0.
\end{align*}
\end{proof}

\begin{proposition}\label{implicit-function0}
There exists a constant $\varepsilon_4>0$ such that, for any $u\in B_{\varepsilon_4}(W):=\{u\in H^1(\mathbb{R}^3):~ ||u-W||_{H^1}\leq\varepsilon_4\}$, there are a unique $\tau(u)\in\mathbb{R}^3$ and a unique $R(u)\in T^{\bot}_{W_{\tau(u)}}$ such that
$$u=W_{\tau(u)}+R(u).$$
In addition, $||\tau(u)||_{\mathbb{R}^3}<\varepsilon_4$ and $||R(u)||_{H^1}<\varepsilon_4$ with $\tau(W)=0$, $R(W)=0$. Here we define an operator $\mathcal{P}:~ B_{\varepsilon_4}(W)\rightarrow H^1(\mathbb{R}^3)$ by $\mathcal{P}(u):=W_{\tau(u)}$, and $\mathcal{P}\in C^1(B_{\varepsilon_4}(W), H^1(\mathbb{R}^3))$.
\end{proposition}

\begin{proof}
First we define a map $\Phi_1:~ H^1(\mathbb{R}^3)\times H^1(\mathbb{R}^3)\rightarrow H^1(\mathbb{R}^3)\times\mathbb{R}^3$ by
\begin{align*}
\Phi_1(W_{\tau}, h):=\left(W_{\tau}+h, \int v_1(x+\tau)\bar{h}dx, \int v_2(x+\tau)\bar{h}dx, \int v_3(x+\tau)\bar{h}dx\right),
\end{align*}
where $v_1, v_2, v_3$ belong to $H^1(\mathbb{R}^3)$ such that $\hbox{span}\{v_1, v_2, v_3\}=T_{W}$. It is simple to find that $\Phi_1(W, 0)=(W, 0)$, and for any $h, k\in H^1(\mathbb{R}^3)$,
\begin{align}\label{implicit-function0-1}
d\Phi_1(W, 0)(h, k)=\left(h+k, \int v_1\bar{k}, \int v_2\bar{k}, \int v_3\bar{k}\right),
\end{align}
where $d\Phi_1$ denotes the differential of $\Phi_1$ at the point $x$. We shall prove that the linear operator $d\Phi_1(W, 0)\in\mathcal{L}(T_{W}\times H^1(\mathbb{R}^3), H^1(\mathbb{R}^3)\times\mathbb{R}^3)$ is invertible. The boundedness of the operator $d\Phi_1(W, 0)$ is obvious by the definition (\ref{implicit-function0-1}). Observe that
\begin{align*}
&\{d\Phi_1(W, 0)(-h, h):~ h\in T_{W}\}=\{0\}\times\mathbb{R}^3,\\
&\{d\Phi_1(W, 0)(0, k):~ k\in T^{\bot}_{W}\}=T^{\bot}_{W}\times\{0\},\\
&\{d\Phi_1(W, 0)(h, 0):~ h\in T_{W}\}=T_{W}\times\{0\}.
\end{align*}
This means that $d\Phi_1(W, 0)$ is surjective. Next we prove that $d\Phi_1(W, 0)$ is injective. Assume that $(h, k)\in T_{W}\times H^1(\mathbb{R}^3)$ such that $d\Phi_1(W, 0)(h, k)=(0, 0)\in H^1(\mathbb{R}^3)\times\mathbb{R}^3$, from (\ref{implicit-function0-1}), then $h+k=0$ and $k\in T^{\bot}_{W}$. Noting that $h\in T_{W}$, then we have $h=k=0$. Thus, $d\Phi_1(W, 0)$ is injective. By applying the implicit function theorem,
the conclusion can be derived.
\end{proof}

\begin{lemma}\label{converges-N}
We define $N:~H^1(\mathbb{R}^3)\rightarrow\mathbb{R}$ by $N(u):=\iint\frac{1-e^{-\frac{|x-y|}{am^{\alpha_2(p)}}}}{|x-y|}|u(x)|^2|u(y)|^2dxdy=:\int\phi_{u, m}|u|^2dx$, where $\phi_{u, m}:=\frac{1-e^{-\frac{|x|}{am^{\alpha_2(p)}}}}{|x|}\ast|u|^2=\int\frac{1-e^{-\frac{|x-y|}{am^{\alpha_2(p)}}}}{|x-y|}|u(y)|^2dy$. Then,
\begin{itemize}
  \item[(1)] $N':~ H^1(\mathbb{R}^3)\rightarrow H^{-1}(\mathbb{R}^3)$ is weakly sequentially continuous;
  \item[(2)] $N''(u)\in\mathcal{L}(H^1(\mathbb{R}^3), H^{-1}(\mathbb{R}^3))$ is compact for any $u\in H^1(\mathbb{R}^3)$.
\end{itemize}
\end{lemma}

\begin{proof}
(1) Let $\{u_n\}^{\infty}_{n=1}\subset H^1(\mathbb{R}^3)$ and $u_n\rightharpoonup u$ in $H^1(\mathbb{R}^3)$ as $n\rightarrow+\infty$, we need to prove that for any $\varphi\in H^1(\mathbb{R}^3)$, $\langle N'(u_n)-N'(u), \varphi\rangle_{L^2}\rightarrow0$.

In fact,
\begin{align*}
\langle N'(u_n)-N'(u), \varphi\rangle_{L^2}=4\int(\phi_{u_n, m}u_n \varphi-\phi_{u, m}u\varphi)dx=4\int\left(\phi_{u_n, m}(u_n-u)\varphi+(\phi_{u_n, m}-\phi_{u, m})u\varphi\right)dx.
\end{align*}
By Lemma \ref{property2}, we obtain
\begin{align*}
\int|\phi_{u_n, m}(u_n-u)|^2dx\leq||\phi_{u_n, m}||^2_{L^6}||u_n-u||^2_{L^3}\leq C||u_n||^4_{H^1}||u_n-u||^2_{H^1}\leq C,
\end{align*}
where $C$ is a positive constant. Since $u_n\rightarrow u$ a.e. in $\mathbb{R}^3$, we have $\phi_{u_n, m}(u_n-u)\rightharpoonup0$ in $L^2(\mathbb{R}^3)$ as $n\rightarrow+\infty$. Thus,
$$\int\phi_{u_n, m}(u_n-u)\varphi dx\rightarrow0,~ \hbox{as}~n\rightarrow+\infty.$$
Similarly, $\int(\phi_{u_n, m}-\phi_{u, m})u\varphi dx\rightarrow0,~ \hbox{as}~n\rightarrow+\infty$. Therefore, we have (1).

(2) Let $\{v_n\}^{\infty}_{n=1}\subset H^1(\mathbb{R}^3)$ be a bounded sequence. Without loss of generality, we can assume that up to a subsequence, $v_n\rightharpoonup0$, as $n\rightarrow+\infty$. Next, we need to show that for any $\varphi\in H^1(\mathbb{R}^3)$ and $||\varphi||_{H^1}\leq1$ such that when $n\rightarrow+\infty$,
\begin{align*}
\langle N''(u)v_n, \varphi\rangle_{L^2}\rightarrow0,~~\hbox{uniformly with respect to}~ \varphi.
\end{align*}
Indeed, let $G_m(x, y):=\frac{1-e^{-\frac{|x-y|}{am^{\alpha_2(p)}}}}{|x-y|}$, by Minkowski inequality, we have
\begin{align*}
\langle N''(u)v_n, \varphi\rangle_{L^2}=&\int\left(8u(y)v_n(y)G_m(x, y)u(x)\varphi(x)+4|u(y)|^2G_m(x, y)v_n(x)\varphi(x)\right)dx\\
\leq&8||\varphi||_{L^2}\left[\int\left(|u(y)v_n(y)G_m(x, y)u(x)|dy\right)^2dx\right]^{\frac12}\\
&+4||\varphi||_{L^3}\left[\int\left(||u(y)|^2G_m(x, y)v_n(x)|dy\right)^{\frac32}dx\right]^{\frac23}\\
\leq&8||\varphi||_{L^2}\int|\widetilde{v}(y)|^{\frac12}u(y)v_n(y)dy+4||\varphi||_{L^3}\int|\widetilde{w_n}(y)|^{\frac23}|u(y)|^2dy,
\end{align*}
where $\widetilde{v}(y)$ and $\widetilde{w_n}(y)$ are defined respectively by
$$\widetilde{v}(y):=\int|u(x)G_m(x, y)|^2dx, \quad \widetilde{w_n}(y):=\int|v_n(x)G_m(x, y)|^{\frac32}dx.$$
From Lemma \ref{HLS}, we can easily get $\widetilde{v}\in L^2(\mathbb{R}^3)$. Thus,
$$\int\left(|\widetilde{v}(y)|^{\frac12}u(y)\right)^2dy\leq||\widetilde{v}||_{L^2}||u||^2_{L^4}<+\infty,$$
which means
$$\int|\widetilde{v}(y)|^{\frac12}u(y)v_n(y)dy\rightarrow0, ~~\hbox{as}~n\rightarrow+\infty.$$
Again, by Lemma \ref{HLS}, we have $\widetilde{w_n}\in L^4(\mathbb{R}^3)$ and is bounded in $L^4(\mathbb{R}^3)$. For any $y\in\mathbb{R}^3$ fixed and $R>0$, since
\begin{align*}
\widetilde{w_n}(y)\leq&\left(\int_{|x-y|\leq R}|v_n|^{\frac92}dx\right)^{\frac13}\left(\int_{|x-y|\leq R}|G_m(x, y)|^{\frac94}dx\right)^{\frac23}\\
&+\left(\int_{|x-y|\geq R}|v_n|^2dx\right)^{\frac34}\left(\int_{|x-y|\geq R}|G_m(x, y)|^6dx\right)^{\frac19}.
\end{align*}
Letting $n\rightarrow+\infty$ and $R\rightarrow+\infty$, we get $\widetilde{w_n}(y)\rightarrow0$ a.e. in $\mathbb{R}^3$, as $n\rightarrow+\infty$.
Hence, $|\widetilde{w_n}|^{\frac23}\rightharpoonup0$ in $L^6(\mathbb{R}^3)$, as $n\rightarrow+\infty$, and from $|u|^2\in L^{\frac65(\mathbb{R}^3)}$, we have
$$\int|\widetilde{w_n}(y)|^{\frac23}|u(y)|^2dy\rightarrow0, ~~\hbox{as}~n\rightarrow+\infty.$$
Therefore,
\begin{align*}
\langle N''(u)v_n, \varphi\rangle_{L^2}\rightarrow0,~~\hbox{uniformly with respect to}~ \varphi.
\end{align*}
\end{proof}

\begin{lemma}\label{implicit-function1}
There exists constants $\varepsilon_5>0$ and $m_1, m_2, m_3>0$ such that the equation
\begin{equation*}
-\Delta w+\omega w+m^{\alpha_1(p)}\int\frac{1-e^{-\frac{|x-y|}{am^{\alpha_2(p)}}}}{|x-y|}|w(y)|^2dyw=|w|^{p-2}w
\end{equation*}
has a solution $w(m, \omega)\in H^1_{r}(\mathbb{R}^3)$ for any $\omega\in(\omega_0-\varepsilon_5, \omega_0+\varepsilon_5)$. Here $m$ satisfying one of the following conditions:
\begin{itemize}
  \item[(1)] if $p\in(2, 3)$, $m\in(0, m_1)$;
  \item[(2)] if $p\in(3, \frac{10}3)$, $m\in(m_2, +\infty)$;
  \item[(3)] if $p\in(\frac{10}3, 6)$, $m\in(0, m_3)$.
\end{itemize}
Moreover, when $p\in(2, 3)$ or $p\in(\frac{10}3, 6)$, $\lim\limits_{(m, \omega)\rightarrow(0, \omega_0)}w(m, \omega)=W$ in $H^1_{r}(\mathbb{R}^3)$; when $p\in(3, \frac{10}3)$, $\lim\limits_{(m, \omega)\rightarrow(+\infty, \omega_0)}w(m, \omega)=W$ in $H^1_{r}(\mathbb{R}^3)$
\end{lemma}

\begin{proof}
We just prove the case where $p\in(2,3)$, the other two cases can be similarly proved. To achieve this, we introduce a map $\Phi_2:~ \mathbb{R}^+\times\mathbb{R}^+\times H^1_{r}(\mathbb{R}^3)\rightarrow H^{-1}_{r}(\mathbb{R}^3)$ as
$$\Phi_2(m, \omega, w):=\nabla_{w}\Gamma^{m}(w),$$
where $\Gamma^m$ is defined by
\begin{align}\label{implicit-function1-1}
\Gamma^m(w):=E_{m}(w)+\frac{\omega}2||w||^2_{L^2}.
\end{align}
We derive that $\Phi_2(m, \omega, w)=(-\Delta+\omega)w+m^{\alpha_1(p)}\int\frac{1-e^{-\frac{|x-y|}{am^{\alpha_2(p)}}}}{|x-y|}|w(y)|^2dyw-|w|^{p-2}w$. By Lemma \ref{converges-N}, we can obtain
$\Phi_2(m, \omega, w)\rightarrow\Phi_2(0, \omega_0, w)$ in $H^{-1}_{r}(\mathbb{R}^3)$, as $m\rightarrow0$. We have $\Phi_2\in C(\mathbb{R}^+\times\mathbb{R}^+\times H^1_{r}(\mathbb{R}^3), H^{-1}_{r}(\mathbb{R}^3))$. Notice that $\Phi_2(0, \omega_0, W)=0$, and for any $h\in H^1_{r}(\mathbb{R}^3)$,
\begin{align*}
d\Phi_2(0, \omega_0, W)h=-\Delta h+\omega_0 h-(p-1)W^{p-2}h.
\end{align*}
We shall deduce that the linear operator $d\Phi_2(0, \omega_0, W)\in\mathcal{L}(H^1_{r}(\mathbb{R}^3), H^{-1}_r(\mathbb{R}^3))$ is invertible. Since the embedding $H^1_{r}(\mathbb{R}^3)\hookrightarrow L^q(\mathbb{R}^3)$ is compact for any $q\in(2, 6)$, and the function $W$ decays exponentially as $|x|$ goes to infinity, then it is easy to see that $d\Phi_2(0, \omega_0, W)$ is surjective, namely, for any $f\in H^{-1}_r(\mathbb{R}^3)$, there exists $h\in H^1_{r}(\mathbb{R}^3)$ solving the equation
$$-\Delta h+\omega_0 h-(p-1)W^{p-2}h=f.$$
Recall that $W$ is nondegenerate, then $\ker(d\Phi_2(0, \omega_0, W))=\{0\}$ (see in \cite{Ni1993, Weinstein1985}), which yields that $d\Phi_2(0, \omega_0, W)$ is injective. We complete the proof by the implicit function theorem.
\end{proof}

\begin{lemma}\label{implicit-function2}
Assuming that $m$ satisfies the conditions of Lemma \ref{implicit-function1}, there exist constants $\varepsilon_6, \varepsilon_7>0$ such that, for any $\omega\in(\omega_0-\varepsilon_6, \omega_0+\varepsilon_6)$, there exists a unique $u=u(m, \omega)\in B_{\varepsilon_7}(W)$ so that
$$\mathcal{P}(u)=W,\quad \pi_{T^{\bot}_{W}}(\nabla_{u}\Gamma^m(u))=0,$$
where the operator $\mathcal{P}$ is defined in Proposition \ref{implicit-function0} and $\Gamma^m(u)$ is defined in (\ref{implicit-function1-1}).
\end{lemma}

\begin{proof}
We define a map $\Phi_3:~ \mathbb{R}^+\times\mathbb{R}^+\times B_{\varepsilon_4}(W)\rightarrow T^{\bot}_{W}\times H^1(\mathbb{R}^3)$ by
\begin{align*}
\Phi_3(m, \omega, u):=\left(\pi_{T^{\bot}_{W}}(\nabla_{u}\Gamma^m(u)), \mathcal{P}(u)\right),
\end{align*}
where the constant $\varepsilon_4>0$ is given in Proposition \ref{implicit-function0}. By the definition of the operator $\mathcal{P}$, we know that $\mathcal{P}(W)=W$ and $\mathcal{P}(W_{\tau})=W_{\tau}$, then it is immediate to see that $\Phi_3(0, \omega_0, W_{\tau})=(0, W_{\tau})$. Therefore, for any $h\in T_{W}$,
\begin{align}\label{implicit-function2-1}
d\Phi_3(0, \omega_0, W)h=(0, h).
\end{align}
We shall assert that $d\Phi_3(0, \omega_0, W)\in\mathcal{L}(H^1(\mathbb{R}^3), T^{\bot}_{W}\times T_{W})$ is invertible. By the definition
of the map $\Phi_3$, it is not difficult to verify that $d\Phi_3(0, \omega_0, W)\in\mathcal{L}(T^{\bot}_{W}, T^{\bot}_{W})$ is surjective. This together with (\ref{implicit-function2-1}) indicates that $d\Phi_3(0, \omega_0, W)\in\mathcal{L}(H^1(\mathbb{R}^3), T^{\bot}_{W}\times T_{W})$
is suejective. We next assume that there is $h\in H^1(\mathbb{R}^3)$ such that $d\Phi_3(0, \omega_0, W)h=(0, 0)$, which implies that $h\in T_{W}$. From (\ref{implicit-function2-1}), it follows that $h=0$. Consequently, $d\Phi_3(0, \omega_0, W)$ is invertible. Thus the lemma follows that by the implicit function theorem.
\end{proof}

\begin{proof}[Proof of Theorem \ref{unique-normalized-solution}]
According to Lemmas \ref{positive-solution0} and \ref{positive-soution00}, every minimizer to problem (\ref{min-pro1}) or (\ref{Pohozaev-manifold}) is positive. Therefore, it suffices for us to show the radial symmetry and uniqueness. Here, we just prove the item (i), the other two items can be proved similarly. The proof is divided into the following two steps.\\
\textbf{Step 1:} Prove that, for any $m>0$ small enough, every minimizer to (\ref{min-pro1}) is radially symmetric up to translations.

Indeed, it is equivalent to prove that, for any $m>0$ small enough, every minimizer to (\ref{equivalent-1}) is radially symmetric up to translations. Let $u\in\mathcal{N}_{m}$. According to Theorem \ref{concentration-behavior}, for every $\varepsilon>0$, there exists $m_1:=m(\varepsilon)>0$ such that (up to translations) $u\in B_{\varepsilon}(W)$ provided that $0<m<m_1$, where $\varepsilon\leq\varepsilon_i$ for $1\leq i\leq7$. Moreover, $u$ solves the following equation
\begin{align}\label{equivalent-equation}
-\Delta u+\omega u+m^{\alpha_1(p)}\int\frac{1-e^{-\frac{|x-y|}{am^{\alpha_2(p)}}}}{|x-y|}|u(y)|^2dyu=|u|^{p-2}u
\end{align}
or equivalently
$$\nabla_{u}\Gamma^m(u)=0,$$
for a suitable $\omega$ such that $\omega\in(\omega-\varepsilon, \omega+\varepsilon)$ provided that $0<m<m_1$ (see Lemma \ref{converges-Q0}).
Thus, it follows from Proposition \ref{implicit-function0} that there exist $\tau\in\mathbb{R}^3$ and $R(u)\in T^{\bot}_{W_{\tau}}$ with $||R(u)||_{H^1}<\varepsilon$ such that $u=W_{\tau}+R(u)$, that is, $u(x-\tau)=W+R(u(x-\tau))$. Hence,
\begin{align}\label{equivalent-equation-1}
\mathcal{P}(u(x-\tau))=W,\quad \pi_{T^{\bot}_{W}}(\nabla_{u(x-\tau)}\Gamma^m(u))=0,
\end{align}
where the second one is a consequence of the fact that the functional $\Gamma^m$ is
invariant under any translation in $\mathbb{R}^3$. Furthermore, one can check that the solution $w(m, \omega)\in H^1_{r}(\mathbb{R}^3)$ obtained in Lemma \ref{implicit-function1} satisfies (\ref{equivalent-equation-1}) as well. Consequently, by Lemma \ref{implicit-function2}, we know that $u(x-\tau)=w(m, \omega)$, which then suggests that $u$ is radially symmetric up to translations.\\
\textbf{Step 2:} Prove that, for any $m>0$ small enough, every minimizer to (\ref{min-pro1}) is unique up to translations.

In fact, it is equivalent to show that, for any $m>0$ small enough, every minimizer to (\ref{equivalent-1}) is unique up to translations. To this end, we assume that there exist $u_1, u_2\in\mathcal{N}_{m}$. Obviously, $u_j$ solves Eq.(\ref{equivalent-equation}), and $\nabla_{u_j}\Gamma^m(u_j)=0$ for $j=1, 2$. Therefore, for any $0<m<m_1$, one can deduce that there exist $\tau_j\in\mathbb{R}^3$ and $R(u_j)\in T^{\bot}_{W_{\tau_j}}$ with $||R(u_j)||_{H^1}<\varepsilon$ such that $u_j=W_{\tau_j}+R(u_j)$, where the constants $m_1$ and $\varepsilon$ are given in Step 1. This means that $u_j(x-\tau_j)=W+R(u_j(x-\tau_j))$, then
\begin{align*}
\mathcal{P}(u_j(x-\tau_j))=W,\quad \pi_{T^{\bot}_{W}}(\nabla_{u_j(x-\tau_j)}\Gamma^m(u_j))=0,~~j=1, 2.
\end{align*}
Hence, as a result of Step 1, we have that $u_1(x)=u_2(x-\tau_2+\tau_1)$, and the proof is completed.
\end{proof}

\renewcommand{\theequation}
{\thesection.\arabic{equation}}
\setcounter{equation}{0}
\section{Strong instability and global existence} \noindent

In this section, we will show the strong instability of standing waves and the global existence of Eq.(\ref{main6}). To this end, we consider the following minimization problem:
\begin{align}\label{minimization-problem1}
\bar{\gamma}(m)=\inf\{\bar{E}(u):~ u\in S(m), P(u)\leq0\},
\end{align}
where
\begin{align*}
\bar{E}(u):=&E(u)-\frac2{3(p-2)}P(u)\\
=&\frac{3p-10}{6(p-2)}||\nabla u||^2_{L^2}+\frac{3p-8}{12(p-2)}\int\phi_u|u|^2dx+\frac1{6a(p-2)}\iint e^{-\frac{|x-y|}a}|u(x)|^2|u(y)|^2dxdy.
\end{align*}
Then, we have
\begin{align}\nonumber
\bar{\gamma}(m)&=\inf\{\bar{E}(u):~ u\in S(m), P(u)\leq0\}\\ \label{mp2}
&=\inf\{E(u):~ u\in S(m), P(u)=0\}=\gamma(m).
\end{align}
Indeed, if $P(u)<0$, then we have
\begin{align*}
P(\kappa(u, \theta))=&e^{2\theta}||\nabla u||^2_{L^2}+\frac{e^{\theta}}4\iint\frac{1-e^{-\frac{|x-y|}{ae^{\theta}}}}{|x-y|}|u(x)|^2|u(y)|^2dxdy\\
&-\frac1{4a}\iint e^{-\frac{|x-y|}{ae^{\theta}}}|u(x)|^2|u(y)|^2dxdy-\frac{3(p-2)}{2p}e^{\frac{3(p-2)}2\theta}||u||^p_{L^p}>0
\end{align*}
for $\theta<0$ such that $|\theta|$ sufficiently large. Thus, there exists $\theta_0<0$ such that $P(\kappa(u, \theta_0))=0$. Moreover, it follows that
\begin{align*}
\bar{E}(\kappa(u, \theta_0))=&\frac{3p-10}{6(p-2)}e^{2\theta_0}||\nabla u||^2_{L^2}+\frac{3p-8}{12(p-2)}e^{\theta_0}\iint\frac{1-e^{-\frac{|x-y|}{ae^{\theta_0}}}}{|x-y|}|u(x)|^2|u(y)|^2dxdy\\
&+\frac1{6a(p-2)}\iint e^{-\frac{|x-y|}{ae^{\theta_0}}}|u(x)|^2|u(y)|^2dxdy\\
<&\frac{3p-10}{6(p-2)}||\nabla u||^2_{L^2}+\frac{3p-8}{12(p-2)}\iint\frac{1-e^{-\frac{|x-y|}a}}{|x-y|}|u(x)|^2|u(y)|^2dxdy\\
&+\frac1{6a(p-2)}\iint e^{-\frac{|x-y|}a}|u(x)|^2|u(y)|^2dxdy\\
=&\bar{E}(u).
\end{align*}
This implies that (\ref{mp2}) holds.

\begin{proposition}\label{mp3}
Let $p\in(\frac{10}3, 6)$. Then, there exists $u\in V(m)$ and $\bar{E}(u)=\bar{\gamma}(m)$.
\end{proposition}

\begin{proof}
We first show that $\bar{\gamma}(m)>0$. By $P(u)\leq0$, we have
$$||\nabla u||^2_{L^2}+\frac14\iint\left(\frac{1-e^{-\frac{|x-y|}a}}{|x-y|}-\frac1a e^{-\frac{|x-y|}a}\right)|u(x)|^2|u(y)|^2dxdy\leq\frac{3(p-2)}{2p}||u||^p_{L^p}.$$
Since $\frac{1-e^{-\frac{|x-y|}a}}{|x-y|}-\frac1a e^{-\frac{|x-y|}a}\geq0$, we have
$$||\nabla u||^2_{L^2}\leq C(p)||\nabla u||^{\frac{3(p-2)}2}_{L^2}||u||^{\frac{6-p}2}_{L^2},$$
which implies that
$$||\nabla u||^{\frac{3p-10}2}_{L^2}\geq\frac1{C(p)m^{\frac{6-p}2}}.$$
Taking the infimum over $u$, we get $\bar{\gamma}(m)>0$.

We now show that minimizing problem (\ref{minimization-problem1}) is attained. Let $\{v_n\}_{n=1}^{\infty}$ be a minimizing sequence for (\ref{minimization-problem1}), that is,  $\{v_n\}_{n=1}^{\infty}\subset S(m)$, $P(v_n)\leq0$, and $\bar{E}(v_n)\rightarrow\bar{\gamma}(m)$ as $n\rightarrow\infty$. Thus, there exists $C_0>0$ such that
$$\frac{3(p-2)}{2p}\liminf_{n\rightarrow\infty}||v_n||^p_{L^p}\geq\liminf_{n\rightarrow\infty}
||\nabla v_n||^2_{L^2}\geq C_0>0.$$
It is also clear that the sequence $\{v_n\}_{n=1}^{\infty}$ is bounded in $H^1(\mathbb{R}^3)$. Applying compact lemma (see \cite[Theorem 1.1]{Hmidi2005}), there exists a subsequence, still denoted by $\{v_n\}_{n=1}^{\infty}$ and $u\in H^1(\mathbb{R}^3)\setminus\{0\}$ such that
$$u_n:=v_n(\cdot+x_n)\rightharpoonup u\neq0 ~~\hbox{weakly in} ~~H^1(\mathbb{R}^3)$$
for some $\{x_n\}_{n=1}^{\infty}\subset\mathbb{R}^3$.
Moreover, we deduce from Br\'{e}zis-Lieb lemma (see in \cite{Brezis1983}) and Lemma \ref{BL} that
\begin{align}\label{BL11}
P(u_n)-P(u_n-u)-P(u)\rightarrow0,
\end{align}
\begin{align}\label{BL22}
\bar{E}(u_n)-\bar{E}(u_n-u)-\bar{E}(u)\rightarrow0,
\end{align}
\begin{align}\label{BL33}
||u_n||^2_{L^2}-||u_n-u||^2_{L^2}-||u||^2_{L^2}\rightarrow0.
\end{align}
Now, we show that $P(u)\leq0$ and $||u||_{L^2}=m$ by excluding the other possibilities:

$\bullet$ If $P(u)>0$ and $||u||_{L^2}<m$, it follows from (\ref{BL11}) and $P(u_n)\leq0$ that $P(u_n-u)\leq0$ for sufficiently large $n$. Setting $m_1:=\sqrt{m^2-||u||^2_{L^2}}$ and $w_n:=\frac{m_1(u_n-u)}{||u_n-u||_{L^2}}$, we have
$$||u_n-u||_{L^2}\rightarrow m_1,~ w_n\in S(m_1) ~~ \hbox{and}~~ P(w_n)\leq0.$$
Thus, by definition of $\bar{\gamma}(m_1)$, it follows that
$$\bar{E}(w_n)\geq\bar{\gamma}(m_1) \quad\hbox{and}\quad \bar{E}(u_n-u)\geq\bar{\gamma}(m_1).$$
Applying $\bar{\gamma}(m_1)=\gamma(m_1)>\gamma(m)$ and (\ref{BL22}), we can obtain
$$\bar{E}(u)=\frac{3p-10}{6(p-2)}||\nabla u||^2_{L^2}+\frac{3p-8}{12(p-2)}\int\phi_u|u|^2dx+\frac1{6a(p-2)}\iint e^{-\frac{|x-y|}a}|u(x)|^2|u(y)|^2dxdy\leq0,$$
which is in contradiction with $u\neq0$.

$\bullet$ If $P(u)>0$ and $||u||_{L^2}=m$, then $u_n\rightarrow u$ in $L^2(\mathbb{R}^3)$ as $n\rightarrow\infty$. this implies that $u_n\rightarrow u$ in $L^p(\mathbb{R}^3)$ as $n\rightarrow\infty$. On the other hand, we deduce from $P(u)>0$ that $P(u_n-u)\leq0$ for sufficiently large $n$, that is,
\begin{align*}
P(u_n-u)=&||\nabla(u_n-u)||^2_{L^2}+\frac14\iint\left(\frac{1-e^{-\frac{|x-y|}a}}{|x-y|}-\frac1a e^{-\frac{|x-y|}a}\right)|(u_n-u)(x)|^2|(u_n-u)(y)|^2dxdy\\
&-\frac{3(p-2)}{2p}||u_n-u||^p_{L^p}\leq0, \quad\hbox{as}\quad n\rightarrow\infty.
\end{align*}
Thus, we can obtain
$$||\nabla(u_n-u)||^2_{L^2}-\frac{3(p-2)}{2p}||u_n-u||^p_{L^p}\leq0,$$
this implies that $u_n\rightarrow u$ in $\dot{H}^1(\mathbb{R}^3)$ as $n\rightarrow\infty$. This yields that
\begin{align*}
o_n(1)=&||\nabla(u_n-u)||^2_{L^2}-\frac{3(p-2)}{2p}||u_n-u||^p_{L^p}\\
\leq&-\frac14\iint\left(\frac{1-e^{-\frac{|x-y|}a}}{|x-y|}-\frac1a e^{-\frac{|x-y|}a}\right)|(u_n-u)(x)|^2|(u_n-u)(y)|^2dxdy\\
\leq&C||u_n-u||^4_{L^2}\rightarrow0 \quad\hbox{as}\quad n\rightarrow\infty.
\end{align*}
Therefore, $P(u_n-u)\rightarrow0$ as $n\rightarrow\infty$. Thus, it follows form (\ref{BL11}) and $P(u)>0$ that $P(u_n)>0$ for sufficiently large $n$, which is in contradiction with $P(u_n)\leq0$.

$\bullet$ If $P(u)\leq0$ and $||u||_{L^2}<m$, then we conclude form (\ref{BL22}) and $\bar{\gamma}(||u||_{L^2})=\gamma(||u||_{L^2})>\gamma(m)=\bar{\gamma}(m)$ that
$$\bar{E}(u_n-u)=\frac{3p-10}{6(p-2)}||\nabla u||^2_{L^2}+\frac{3p-8}{12(p-2)}\int\phi_u|u|^2dx+\frac1{6a(p-2)}\iint e^{-\frac{|x-y|}a}|u(x)|^2|u(y)|^2dxdy<0,$$
which is a contradiction.\\
Therefore, we have $P(u)\leq0$ and $||u||_{L^2}=m$. It follows from the definition of
$\bar{\gamma}(m)$ and the weak lower semicontinuity of norm that
$$\bar{\gamma}(m)\leq\bar{E}(u)\leq\liminf_{n\rightarrow\infty}\bar{E}(u_n)=\bar{\gamma}(m).$$
This yields that
$$\bar{E}(u)=\bar{\gamma}(m).$$

Finally, we show that $P(u)=0$. Suppose that $P(u)<0$, and set
$$f(\theta):=P(\kappa(u, \theta)),$$
then, $f(\theta)>0$ for $\theta<0$ such that $|\theta|$ sufficiently large and $f(0)=P(u)<0$. Therefore, there exists $\theta_0<0$ such that $P(\kappa(u, \theta_0))=0$. Then, it follows that
$$\bar{E}(\kappa(u, \theta_0))<\bar{E}(u)=\bar{\gamma}(m),$$
which contradicts the definition of $\bar{\gamma}(m)$. Hence, we have $P(u)=0$.
\end{proof}

\begin{lemma}\label{critical-points2}
Let $p\in(\frac{10}3, 6)$, then each critical point of $E|_{V(m)}$ is a critical point of $E|_{S(m)}$.
\end{lemma}
\begin{proof}
The proof of Lemma \ref{critical-points2} is similar to the proof of \cite[Theorem 1.4]{Bellazzini2013}, and it will be omitted.
\end{proof}

\begin{proof}[Proof of Theorem \ref{unstable}]
The proof of Theorem \ref{unstable} is standard and follows the original approach by Glassey \cite{Glassey1977} and Berestycki and Cazenave \cite{Berestycki1981}.

For any $m>0$, let $u_m\in\mathcal{M}_m$ and define the set
$$\mathcal{B}:=\left\{\psi\in S(m):~ E(\psi)<E(u_m), P(\psi)<0\right\}.$$
Note that the set $\mathcal{B}$ contains functions arbitrarily close to $u_m$ in $H^1(\mathbb{R}^3)$. Indeed, letting $\psi_0(x)=\kappa(u_m, \theta)(x)$ with $\theta>0$, we see from Lemma \ref{unique} that $\psi_0\in\mathcal{B}$ and $\psi_0\rightarrow u_m$ in $H^1(\mathbb{R}^3)$ as $\theta\rightarrow0$.

Suppose $\psi(t)$ be the maximal solution of Eq.(\ref{main6}) with initial datum $\psi(0)=\psi_0$. Let us show that $\psi(t)\in\mathcal{B}$ for all $t\in[0, T^*)$. From the conservation
laws
$$||\psi(t)||^2_{L^2}=||\psi_0||^2_{L^2}=||u_m||^2_{L^2}$$
and
$$E(\psi(t))=E(\psi_0)<E(u_m).$$
Thus, it is enough to verify $P(\psi(t))<0$, but $P(\psi(t))\neq0$ for any $t\in[0, T^*)$.
Otherwise, by the definition of $\gamma(m)$, we would get for a $t_0\in[0, T^*)$ that $E(\psi(t_0))\geq E(u_m)$ in contradiction with $E(\psi(t))<E(u_m)$. Now by continuity of $P$, we get that $P(\psi(t))<0$ and thus that $\psi(t)\in \mathcal{B}$ for all $t\in[0, T^*)$. Now, we claim that there exists $\delta>0$ such that
\begin{align}\label{blowup}
P(\psi(t))\leq-\delta \quad\hbox{for all~~ $t\in[0, T^*)$}.
\end{align}
Let $t\in[0, T^*)$ be arbitrary but fixed and set $\psi=\psi(t)$. Since $P(\psi)<0$, we deduce from Proposition \ref{mp3} that
$$\gamma(m)=\bar{\gamma}(m)\leq\bar{E}(\psi)=E(\psi)-\frac2{3(p-2)}P(\psi)<E(\psi_0)-\frac{P(\psi)}2.$$
This implies that
\begin{align}\label{key-e}
P(\psi)\leq 2(E(\psi_0)-\gamma(m))=2(E(\psi_0)-E(u_m)).
\end{align}
Then, letting $\delta=2(E(u_m)-E(\psi_0))>0$ the claim is established. Since $\psi_0(x)=\kappa(u_m, \theta)(x)$, we have that
$$\int|x|^2|\psi_0(x)|^2dx=\int|x|^2|\kappa(u_m, \theta)(x)|^2dx=e^{-2\theta}\int|y|^2|u_m(y)|^2dy.$$
Thus, from Lemma \ref{exponential-decay} and Lemma \ref{critical-points2}, we obtain that
\begin{align*}
\int|x|^2|\psi_0(x)|^2dx<\infty.
\end{align*}
By the virial identity, we have
$$\frac{d^2}{dt^2}||x\psi(t)||^2_{L^2}=8P(\psi).$$
Now, by (\ref{blowup}), we deduce that $\psi(t)$ must blow up in finite time. Recording that $\psi_0$ has been taken arbitrarily close to $u_m$, this ends the proof of the theorem.
\end{proof}

\begin{proof}[Proof of Theorem \ref{global}]
For any $m>0$, let us define the set
$$\mathcal{A}:=\{\psi\in S(m):~ E(\psi)<\gamma(m), P(\psi)>0\}.$$
First, we show that the set $\mathcal{A}$ is not empty. Indeed, for arbitrary but fix $\psi\in S(m)$, set $\psi_{\theta}(x)=\kappa(\psi, \theta)(x)$. Then, we have $\psi_{\theta}\in S(m)$ for all $\theta\in\mathbb{R}$; $E(\psi_{\theta})\rightarrow0$ as $\theta\rightarrow-\infty$ and $P(\psi_{\theta})>0$ for $\theta<0$ such that $|\theta|$ large enough. This proves that $\mathcal{A}$ is nonempty.

In the following, we will prove that $\mathcal{A}$ is a invariant manifold of Eq.(\ref{main6}). Let $\psi(t, x)$ be the solution of Eq.(\ref{main6}) with $\psi(0, x)=\psi_0$. Suppose $\psi_0\in \mathcal{A}$, by Proposition \ref{local}, we see that there exists a unique solution $\psi\in C([0, T^*); H^1(\mathbb{R}^3))$ with initial data $\psi_0$. We deduce from the conservation of energy that
\begin{align}\label{c-E}
E(\psi(t))=E(\psi_0)<\gamma(||\psi_0||_{L^2})
\end{align}
for any $t\in[0, T^*)$. In addition, by the continuity of the function $t\mapsto P(\psi(t))$ and Proposition \ref{mp3}, if there exists $t_1\in[0, T^*)$ such that $P(\psi(t_1))=0$, then $E(\psi(t_1))\geq\gamma(||\psi_0||_{L^2})$, which contradicts with (\ref{c-E}). Therefore, we have $P(\psi(t))>0$ for any $t\in[0, T^*)$. Next classically we either have
$$T^*=+\infty$$
or
\begin{align}\label{T-E}
T^*<+\infty \quad\hbox{and}\quad \lim_{t\rightarrow T^*}||\nabla \psi(t, x)||^2_{L^2}=\infty.
\end{align}
Since
\begin{align*}
E(\psi(t, x))-\frac2{3(p-2)}P(\psi(t, x))=&\frac{3p-10}{6(p-2)}||\nabla \psi(t, x)||^2_{L^2}+\frac{3p-8}{12(p-2)}\int\phi_{\psi(t, x)}|\psi(t, x)|^2dx\\
&+\frac1{6a(p-2)}\iint e^{-\frac{|x-y|}a}|\psi(t, x)|^2 |\psi(t, y)|^2 dxdy,
\end{align*}
and $E(\psi(t, x))=E(\psi_0)$ for all $t\in[0, T^*)$, if (\ref{T-E}) happens, then we have
$$\lim_{t\rightarrow T^*}P(\psi(t, x))=-\infty.$$
By continuity, these exists $t_2\in[0, T^*)$ such that $P(\psi(t_2, x))=0$ with $E(\psi(t_2,x))=E(u_0)<\gamma(m)$. This contradicts the definition $\gamma(m)=\inf\limits_{\psi\in V(m)}E(\psi)$.
\end{proof}

 \par

\end{document}